\documentclass[a4paper,12pt,reqno]{amsart}
\usepackage[applemac]{inputenc}
\usepackage{amsmath,amssymb,mathrsfs,amsthm,stmaryrd}
\usepackage[osf]{newtxtext}
\usepackage[left=3cm,right=3cm,top=3cm,bottom=3cm]{geometry}
\usepackage{empheq}
\usepackage{enumerate,enumitem}
\usepackage{dsfont}
\usepackage[hidelinks]{hyperref}
\usepackage{todonotes}
\usepackage{cancel}
\usepackage{tensor}
\usepackage{esint}

\theoremstyle{plain}
\newtheorem{lemma}{Lemma}[section]
\newtheorem{proposition}{Proposition}[section]
\newtheorem{corollary}{Corollary}[section]
\newtheorem{cordef}{Corollary and definition}[section]
\newtheorem{theorem}{Theorem} 

\theoremstyle{remark}
\newtheorem{remark}{Remark}[section]
\newtheorem{example}{Example}[section]

\theoremstyle{definition}

\newtheorem{notaAdd}{Notation}[section]

\newtheorem{definition}{Definition}[section]
\newtheorem{claim}{Claim}[section]
\newtheorem{assumption}{Assumption}[section]

\makeatletter

\@addtoreset{equation}{section}
\makeatother

\providecommand{\keywords}[1]{{\small \textit{Keywords---} #1}}
\renewcommand{\qedsymbol}{$\blacksquare$}

\title[An energy method for rough PDEs]{An energy method for rough partial differential equations}

\author{Antoine Hocquet}
\address[A. Hocquet]{Institute of Mathematics,  Technical University Berlin, Stra\ss e des 17. Juni 136, 10623 Berlin, Germany}
\email{antoine.hocquet@wanadoo.fr}

\author{Martina Hofmanov\'a}
\address[M. Hofmanov\'a]{Institute of Mathematics,  Technical University Berlin, Stra\ss e des 17. Juni 136, 10623 Berlin, Germany}
\email{hofmanov@math.tu-berlin.de}

\thanks{Financial support by the DFG via Research Unit FOR 2402 is gratefully acknowledged.}


\date{}

\DeclareMathOperator*{\esssup}{ess\,sup}

\DeclareMathOperator*{\card}{\#}

\DeclareMathOperator*{\spt}{Supp}

\DeclareMathOperator{\id}{id}

\newcommand{\N}{\ensuremath{\mathbb N_0}}
\newcommand{\R}{\ensuremath{\mathbb R}}
\renewcommand{\L}{\ensuremath{L}} 
\newcommand{\Zc}{\ensuremath{\mathcal Z}}

\newcommand{\ps}[2]{\ensuremath{\big(#1,#2\big)}}
\newcommand{\PS}[2]{\ensuremath{\Big(#1,#2\Big)}}
\newcommand{\n}[2]{\ensuremath{\,\rule[-1mm]{0.15em}{1em}\,#1\,\rule[-1mm]{0.15em}{1em}\,_{#2,(2)}}
}
\newcommand{\nn}[2]{\ensuremath{\,\rule[-1mm]{0.15em}{1em}\,#1\,\rule[-1mm]{0.15em}{1em}\,_{#2,(\infty)}}
}
\newcommand{\nnn}[2]{\ensuremath{
\llparenthesis\hspace{0.15em}#1\hspace{0.15em}\rrparenthesis_{#2}
}
}
\newcommand{\nnnn}[2]{\ensuremath{\,\rule[-1mm]{0.15em}{1em}\,#1\,\rule[-1mm]{0.15em}{1em}\,_{#2}}
}
\DeclareMathOperator{\divergence}{div}
\renewcommand{\div}{\ensuremath{\divergence}}
\renewcommand{\d}{\ensuremath{\mathrm{d}}}

\newcommand{\D}{\ensuremath{\mathrm{D}}}

\newcommand{\ind}{\mathbf{1}}

\def\tsum{\begingroup\textstyle \sum\endgroup}

\renewcommand{\aa}{\ensuremath{\mathbf a}}
\newcommand{\bb}{\ensuremath{\mathbf b}}
\newcommand{\cc}{\ensuremath{\mathbf c}}
\newcommand{\uu}{\ensuremath{\mathbf u}}
\renewcommand{\u}{\ensuremath{\mathfrak{u}}}
\renewcommand{\v}{\ensuremath{\mathfrak{v}}}
\newcommand{\w}{\ensuremath{\mathfrak{w}}}
\newcommand{\F}{\ensuremath{\mathcal F}}
\newcommand{\Bc}{\ensuremath{\mathscr B}}
\renewcommand{\O}{\ensuremath{{\R^d}}}
\newcommand{\stO}{\ensuremath{{[s,t]\times\R^d}}}
\newcommand{\OO}{\ensuremath{{\R^d\times\R^d}}}
\newcommand{\G}{\ensuremath{\mathcal G}}
\newcommand{\B}{\ensuremath{\mathbf B}}
\newcommand{\BB}{\ensuremath{\mathbb B}}
\newcommand{\Z}{\ensuremath{\mathbf Z}}
\newcommand{\ZZ}{\ensuremath{\mathbb Z}}
\newcommand{\X}{\ensuremath{\mathbf X}}
\newcommand{\XX}{\ensuremath{\mathbb X}}
\renewcommand{\S}{\ensuremath{\mathbf S}}
\renewcommand{\SS}{\ensuremath{\mathbb S}}
\newcommand{\sigmab}{\ensuremath{\boldsymbol{\sigma}}}
\newcommand{\T}{\ensuremath{\mathcal T}}
\newcommand{\sym}{\ensuremath{\mathrm{sym}}}
\newcommand{\anti}{\ensuremath{\mathrm{anti}}}
\newcommand{\cv}[1]{\ensuremath{\underset{#1}{\longrightarrow}}}
\newcommand{\I}{\ensuremath{\mathcal I}}
\newcommand{\AC}{\ensuremath{\mathcal{A}\mathcal {C}}}
\newcommand{\V}{\ensuremath{\mathcal{V}}}
\renewcommand{\tt}{\ensuremath{\boldsymbol\tau\!}}
\newcommand{\m}{\ensuremath{\boldsymbol m_{\epsilon}^{x_-}\! }}
\newcommand{\df}{\ensuremath{\boldsymbol{\Delta} _\epsilon ^{x_-}\!}}
\newcommand{\trace}{\ensuremath{\boldsymbol \gamma_\Gamma  }}
\newcommand{\pp}{\ensuremath{\mathfrak p}}
\newcommand{\PP}{\ensuremath{\mathscr P}}

\begin{document}
 
\maketitle

\begin{abstract}
We present a well-posedness and stability result for a class of nondegenerate linear parabolic equations driven by geometric rough paths. More precisely, we introduce a notion of weak solution that satisfies an intrinsic formulation of the equation  in a suitable Sobolev space of negative order. Weak solutions are then shown to satisfy the corresponding energy estimates which are deduced directly from the equation. Existence is obtained by showing compactness of a suitable sequence of approximate solutions whereas uniqueness relies on a doubling of variables argument and a careful analysis of the passage to the diagonal. Our result is optimal in the sense that the assumptions on the deterministic part of the equation as well as the initial condition are the same as in the classical PDEs theory.
\end{abstract}

\tableofcontents

\section{Introduction}

The so-called variational approach, also known as the energy method, belongs among the most versatile  tools in the theory of partial differential equations (PDEs). It is especially useful for nonlinear problems with complicated structure which do not permit the use of (semi-) linear methods such as semigroup arguments, e.g. systems of conservation laws or equations appearing in fluid dynamics. In such cases,  solutions are often known or expected to develop singularities in finite time. Therefore, weak (or variational) solutions which can accommodate these singularities provide  a suitable framework for studying the behavior of the system in the long run. But even for linear or semi-linear problems, weak solutions are the natural notion of solution in cases where a corresponding mild formulation is not available, for instance due to low regularity of coefficients.

The construction of weak solutions via the energy method relies on basic a priori estimates which can be directly deduced from the equation at hand by considering a suitable test function. The equation is then satisfied in a weak sense, that is, as an equality in certain space of distributions. Within this framework, existence and uniqueness are usually established by separate arguments. The proof of existence often uses compactness of a sequence of approximate solutions. Uniqueness for weak solutions is much more delicate and in some cases even not known. Let us for instance mention problems appearing in fluid dynamics where the questions of uniqueness and regularity of weak solutions remain largely open.

It has been long recognized that addition of stochastic terms to the basic governing equations can be used to model an intrinsic presence of randomness as well as to account for other numerical, empirical or physical uncertainties. Consequently, the field of stochastic partial differential equations massively gained importance over the past decades. It relies on the (martingale based) stochastic It\^o integration theory, which gave a probabilistic meaning to  problems that are analytically ill-posed due to the low regularity of trajectories of the driving stochastic processes. Nevertheless, the drawback appearing  already in the context of stochastic differential equations (SDEs) is that the solution map which assigns a trajectory of the solution to a trajectory of the driving signal, known as the It\^o map, is measurable but in general lacks continuity. This loss of robustness has obvious negative consequences, for instance when dealing with numerical approximations or in filtering theory.

The theory of rough paths introduced by Lyons \cite{lyons1998differential} fully overcame the gap between  ordinary and stochastic differential equations and allowed for a pathwise analysis of SDEs. The highly nontrivial step is lifting the irregular noise to a bigger space in a robust way such that solutions to SDEs depend continuously on this lifted noise. More precisely, Lyons singled out the appropriate topology on the space of rough paths which renders the corresponding It\^o--Lyons solution map continuous as a function of a suitably enhanced driving path. As one of the striking consequences, one can allow initial conditions as well as the coefficients of the equation to be random, even dependent on the entire future of the driving signals - as opposed to the ``arrow of time'' and the associated need for adaptedness within It\^o's theory.
In addition, using the rough path theory one can consider drivers beyond the martingale world such as general Gaussian or Markov processes, in  contrast to It\^o's theory where only semimartingales may be considered.

The rough path theory can be naturally formulated also in infinite dimensions to analyze ODEs in Banach spaces. This generalization is, however, not appropriate for the understanding of rough PDEs. This is due to two basic facts. First, the notion of rough path encodes in a fundamental way the nonlinear effects of time varying signals without any possibility of including signals depending  in an irregular way on more parameters. Second, in an infinite dimensional setting the action of a signal (even finite dimensional) is typically described by differential or more generally unbounded operators. Due to these difficulties, attempts at application of the rough path theory in the  study  rough PDEs have been limited. Namely, it was necessary to  avoid unbounded operators by working with mild formulations or Feynman--Kac formulas or transforming the equation in order to absorb the rough dependence into better understood objects such as flow of characteristic curves.

These requirements pose strong limitations on the kind of results one is able to obtain  and the proof strategies are very different from classical PDE methods.
The most successful approaches to  rough PDEs do not even allow to characterize  solutions directly but only via a transformation to a more standard PDE problem.
However, there has been an enormous research activity in the field of rough path driven PDEs lately and the literature is growing very fast. To name at least a few results relevant for our discussion, we refer the reader to the works by Friz et al. \cite{caruana2009partial,caruana2011rough} where flow transformations were applied to fully nonlinear rough PDEs.
A mild formulation was at the core of many other works,  see for instance Deya--Gubinelli--Tindel \cite{deya2012non,gubinelli2010rough} for a semigroup approach to semilinear evolution equations; Gubinelli--Imkeller--Perkowski \cite{gubinelli2015paracontrolled} for the theory of paracontrolled distributions and Hairer \cite{hairer2014theory} for the theory of regularity structures dealing with singular SPDEs.

At this stage, the rough path theory has reached certain level of maturity and it is natural to ask whether one could find rough path analogues to standard PDEs techniques.
From this point of view various authors started to develop \emph{intrinsic} formulations of  rough PDEs which involve relations between certain distributions associated to the unknown and the driving rough path. Let us mention the work of Gubinelli--Tindel--Torrecilla~\cite{gubinelli_controlled_2014} on viscosity solutions to fully nonlinear rough PDEs, that of Catellier~\cite{catellier_rough_2015} on rough transport equations, Diehl--Friz--Stannat~\cite{diehl2017stochastic} for results based on Feynmann--Kac formula. Finally,  Bailleul--Gubinelli~\cite{bailleul2017unbounded} studied  rough transport equations and Deya--Gubinelli--Hofmanov\'a--Tindel \cite{deya2016priori} conservation laws driven by rough paths.

The last two works laid the foundation for the variational approach to rough PDEs: they introduced  a priori estimates for  rough PDEs based on a new rough Gronwall lemma argument. Consequently, it was possible  to derive bounds on various norms of the solution and obtain existence and uniqueness results bypassing the use of the flow transformation or mild formulations.
In addition, these  techniques were used \cite{deya2016one} in order to establish uniqueness for reflected rough differential equations, a problem which remained open in the literature as a suitable Gronwall lemma in the context of rough path was missing.

\bigskip

\paragraph{\bfseries A class of non-degenerate linear parabolic rough PDEs.}
In the present paper, we pursue the line of research initiated in \cite{bailleul2017unbounded,deya2016priori}. Our goal is to develop a variational approach to a class of linear parabolic rough PDEs with possibly discontinuous coefficients. To be more precise, we study existence, uniqueness and stability for rough PDEs of the form 
\begin{equation}\label{rough_PDE}
\left\{\begin{aligned}
&\d u-A(t,x)u\d t=\left(\sigma^{ki}(x) \partial _{i}u +\nu^k(x)u\right)\d \Z^k\enskip, \quad \text{on}\enskip \R_+\times\O\,,
\\
&u(0)=u_0\,,
\end{aligned}\right.
\end{equation}
where
$\Z\equiv((Z^k)_{0\leq k\leq K},(\ZZ^{\ell ,k})_{1\leq \ell ,k\leq K})$ is a \emph{geometric} rough path of finite $1/\alpha -$variation, with $\alpha \in(1/3,1/2]$.
Here and below a summation convention over repeated indexes is used.
Regarding the assumptions on the deterministic part of \eqref{rough_PDE}, we consider an elliptic operator $A$ in divergence form, namely,
\begin{equation}\label{nota:A}
A(t,x)u=\partial _{i} \big(a^{ij}(t,x)\partial _{j} u\big)+b^i(t,x)\partial _{i}u +c(t,x)u.
\end{equation}
The coefficents $a=(a^{ij})_{1\leq i,j\leq d}$, $b=(b^i)_{1\leq i\leq d}$, $c$ are possibly discontinous. More precisely, we assume that $a$ is symmetric and fulfills a uniform ellipticity condition (see Assumption  \ref{ass:a}). Moreover integrability conditions depending on the dimension $d$ of $\O $ are assumed for $b,c$  (see Assumption \ref{ass:b_c}). The coefficients in the noise term $\sigma=(\sigma^{ki})_{1\leq k\leq K,1\leq i\leq d}$ and $\nu=(\nu^k)_{1\leq k\leq K}$ possess $W^{3,\infty}$ and $W^{2,\infty}$ regularity, respectively. The initial condition $u_0$ belongs to $L^2$.

Let us emphasize that the geometricity of $\Z$ is essential to avoid any problem related to the so-called ``strong parabolicity'' requirement (in the case of It\^o calculus), see Remark \ref{rem:sto_par_1} and Remark \ref{rem:sto_par_2} below.
Similarly, working in the $1/\alpha $-variation setting rather than in H\"older spaces turns out to be crucial in order to deal with coefficients that are only integrable, see Remark \ref{rem:controls_abc}.

One can easily see that the above mentioned available approaches to rough PDEs (mild formulation, flow transformation, Feynman--Kac formula) do not apply in this setting. Let us stress that our assumptions on the deterministic part of \eqref{rough_PDE} coincide with the classical (deterministic) theory as presented for instance in the book by Ladyzhenskaya, Solonnikov and Ural'tseva \cite{ladyzhenskaya1968linear}. Consequently, there is no doubt that the very natural way to establish existence and uniqueness is the energy method.
For completeness, let us mention that problems similar to \eqref{rough_PDE} were  studied in \cite{caruana2009partial,diehl2017stochastic} (note however that both these references concern equations written in non-divergence form). In comparison to these results, the energy method has clear advantages in several aspects. First, it allows to significantly weaken the  required regularity of the coefficients and initial datum. Furthermore, the method does not rely on linearity  and thus represents the natural starting point towards more general nonlinear problems.

More precisely, the (unique) solution constructed in \cite{caruana2009partial} was obtained as a transformation of a \emph{classical} solution to a certain deterministic equation. For that reason, the coefficients $a,b$  needed to be of class $C^2_b$ with respect to the space variable and the initial condition had the same regularity, whereas the coefficient $\sigma$ belonged to $\mathrm{Lip}^\gamma$ for some $\gamma>\frac{1}{\alpha}+3$ (note that $c=0$, $\nu= 0$ in \cite{caruana2009partial}). Besides, the equation was solved in a limiting sense only: a solution is defined as a limit point of classical solutions to the  PDE obtained by replacing the driving rough path $\mathbf{Z}$ by its smooth approximation. Uniqueness  then corresponded to the fact that there was at most one limit point. We point out that our notion of uniqueness based on an intrinsic formulation of the equation (see Definition \ref{def:weak_sol}) is stronger as it compares solutions regardless of the way they were constructed.

In the paper \cite{diehl2017stochastic},  an intrinsic weak formulation of an equation of the form \eqref{rough_PDE} was introduced and existence of a unique weak solution proved. The approach was based on the Feynman--Kac formula and therefore the equation was solved backward in time. The result  required $a,\sigma\in C^3_b$, $b,c\in C^1_b$, $\nu\in C^2_b$ and the terminal condition in $C^0_b$. Uniqueness was obtained in the class of continuous and bounded weak solutions.

\bigskip

\paragraph{\bfseries An intrinsic notion of solution}
In order to conclude this introductory part, let us be more precise about our approach and results.
We recall that, at a heuristic level, the entries of the geometric rough path $\Z\equiv(Z,\ZZ)$  mimic the first and second order iterated integrals
\[
\int_s^t\d Z_r\qquad\text{and}\qquad \int_s^t\int_s^r\d Z_{r'}\otimes \d Z_r,
\]
respectively.
These quantities appear naturally in the process of expanding the equation describing $u.$
Namely, assuming that $u$ solves \eqref{rough_PDE}, we get formally
\begin{equation}\label{davie_form}
\begin{split}
u_{t}-u_s&=\int_s^t A(r )u_r\d r
+ Z_{st}^k\big(\sigma ^{ki}\partial_i+ \nu^k\big)u_s\\
&\qquad+\ZZ_{st}^{k\ell}\big(\sigma^{k i}\partial_i+\nu^k\big)\big(\sigma^{\ell j}\partial_j +\nu^\ell \big)u_s+o(t-s),\qquad 0\leq s\leq t\leq T.
\end{split}
\end{equation}
Following Davie's \cite{davie2007differential} interpretation of rough differential equations, one can actually consider \eqref{davie_form} as a definition of \eqref{rough_PDE}, where the smallness of the remainder has to be understood in a suitable Sobolev space of negative order. Roughly speaking, a function $u\in C([0,T];L^2)\cap L^2(0,T;W^{1,2})$ will be called a weak solution to \eqref{rough_PDE} provided \eqref{davie_form} holds true as an equality in $W^{-3,2}$. We remark that the corresponding functional setting is similar to the classical theory, i.e.\ we recognize the usual energy space $\mathscr B:=C([0,T];L^2)\cap L^2(0,T;W^{1,2})$ where weak solutions live. Nevertheless, the regularity required from the test functions is higher ($W^{3,2}$ contrary to $W^{1,2}$ in the classical theory). This is a consequence of the low regularity of the driving signal and the consequent need for a higher order expansion.

The first challenge is to derive the corresponding energy estimates leading to the proof of existence. In view of the formulation \eqref{davie_form}, it is clear that the main difficulty is to estimate the remainder term. Indeed, all the other terms in the equation are explicit and can be easily estimated. However, the only information available on the remainder is the equation \eqref{davie_form} itself. In fact, the definition of a weak solution is to be understood as follows: $u$ is a weak solution to \eqref{rough_PDE} provided the 2-index map given by
\[
u^\natural_{st}:=u_t-u_s-\int_s^t A(r )u_r\d r
- Z_{st}^k\big(\sigma ^{ki}\partial_i+ \nu^k\big)u_s
-\ZZ_{st}^{k\ell}\big(\sigma^{ki}\partial_i +\nu^k\big)\big(\sigma^{\ell j}\partial_j+\nu^\ell\big)u_s
\]
has finite $(1-\kappa)$-variation,  for some $\kappa\in(0,1)$, as a mapping with values in $W^{-3,2}$. It was observed in \cite{bailleul2017unbounded,deya2016priori} that there is a trade-off between space and time regularity and a suitable interpolation argument can be used in order to establish sufficient time regularity of the remainder estimated in terms of the energy norm. This is the core of the so-called rough Gronwall lemma argument which in turn yields the desired energy bound for the solution.

We point out that in view of the required regularity of test functions for \eqref{davie_form}, it is  remarkable that uniqueness in the class of weak solutions can be established. Indeed, this task requires to test the equation by the weak solution itself and it is immediately seen that the $W^{3,2}$-regularity  is far from being satisfied.
Nevertheless, as in \cite{bailleul2017unbounded,deya2016priori}, it is possible perform a tensorization argument which  corresponds to the doubling of variables technique known in the context of conservation laws: one considers  the equation satisfied by the product $u_t(x)u_t(y)$ and tested by a mollifier sequence $\epsilon ^{-d}\psi (\tfrac{x-y}{\epsilon }).$ The core of the proof is then  to  derive estimates uniform in $\epsilon$ in order to be able to pass to the diagonal $x=y$, i.e. to send $\epsilon \to0$. Once this is done, one obtains the equation for $u^2$ and proceeds similarly as in the existence part to derive the energy estimate. 

Nevertheless, there is a major difference between the derivation of the energy estimates in the existence part and in the proof of uniqueness. Namely, in order to establish a priori estimates needed for existence, one works on the level of sufficiently smooth approximations. This can be done e.g.\ by mollifying the driving signal and using classical PDE theory. Consequently, deriving the evolution of $u^2$ is not an issue and can be easily justified. On the other hand, within the proof of uniqueness, the only available regularity is that of weak solutions and the most delicate part is thus to  show that $u^2$ satisfies the right equation.

As discussed above, an important advantage of the rough path theory, as opposed to the stochastic integration theory, is the continuity of the solution map in appropriate topologies. Also in our setting, we obtain the following  Wong-Zakai type result which follows immediately from our construction. Let $(Z^\epsilon)$ be a sequence of smooth paths whose canonical lifts $\mathbf{Z}^\epsilon\equiv(Z^\epsilon,\ZZ^\epsilon)$  approximate  $\mathbf{Z}\equiv(Z,\ZZ)$ in the rough path sense. Let $u^\epsilon$ be the  weak solution of \eqref{rough_PDE} driven by $Z^\epsilon$ obtained by classical arguments. Then we show that $u^\epsilon$ converges in $L^2(0,T;L^2_{\rm{loc}})$ to $u$, which is a solution to \eqref{rough_PDE} driven by $Z$.

\subsection*{Outline of the paper}
In Section \ref{sec:preliminaries},
we introduce the main concepts and notations that we use throughout the article,
and we state our main results, Theorem \ref{thm:solvability} and Theorem \ref{thm:continuity}.
Section \ref{sec:analysis} is devoted to the presentation of the main tools necessary to obtain a priori estimates for rough PDEs.
The so-called energy inequality, appears  at the core of our variational approach.
It arises as a consequence of the a priori estimates, Proposition \ref{pro:apriori}, applied to the remainder term  in the equation governing the evolution of the square of the solution. This is discussed in Section \ref{sec:energy}. 
In Section \ref{sec:tensorization} we introduce  the above mentioned tensorization argument, which is required in the proof of uniqueness. We present it in a rather general way,  motivating the particular choice of function spaces.
The uniqueness part, which is treated in Section \ref{sec:uniqueness}, is the most delicate part of our proof. Finally, the proof of existence as well as stability  is presented in Section \ref{sec:continuity}.
Several auxiliary results are collected in the Appendix.

\section{Preliminaries}
\label{sec:preliminaries}

\subsection{Notation}

We will denote by $\N$ the set of all non-negative integers, that is $\N:=\{0,1,2,\dots\},$
and we will write $\R_+$ to denote the set of non-negative real numbers, that is
$\R_+:=[0,\infty).$
Let us recall the definition of the increment operator, denoted by $\delta$. If $g$ is a path defined on $[0,T]$ and $s,t\in[0,T]$ then $\delta g_{st}:= g_t - g_s$, if $g$ is a $2$-index map defined on $[0,T]^2$ then $\delta g_{s\theta t}:=g_{st}-g_{s\theta }-g_{\theta t}$.
For a closed time interval $I\subset\R_{+},$ 
we denote by $\Delta ,\Delta ^2$ the simplexes
\begin{equation}\label{nota:simplexes}
\Delta=\Delta_I :=\{(s,t)\in I^2\,,\,s\leq t\}\,,\quad \Delta ^2=\Delta^2_I:= \{(s,\theta ,t)\in I^3\,,\,s\leq \theta \leq t\}\,.
\end{equation}
We call \emph{control on $I$} any superadditive map $\omega :\Delta \to \R_+,$ that is, for all $(s,\theta ,t)\in \Delta ^2$ there holds
\begin{equation}
\label{axiom:control}
\omega (s,\theta )+\omega (\theta ,t)\leq \omega (s,t).
\end{equation}
(Note that the property \eqref{axiom:control} implies in particular that $\omega (t,t)=0$ for any $t\in[0,T].$)
We will call $\omega $ \emph{regular} if in addition $\omega $ is continuous.

Given a Banach space $E$ equipped with a norm $|\cdot |_{E}$, and $\alpha >0,$
we denote by 
$\V_ 1^{\alpha }(I;E)$
the set of paths $g:I\to E$ admitting left and right limits with respect to each of the variables, and such that there exists a regular control $\omega:\Delta \to\R_+$ with
\begin{equation}
\label{def:omega_a}
|\delta g_{st}|_{E}\leq \omega (s,t)^{\alpha  }\,,
\end{equation} 
for every $(s,t)\in\Delta .$
Similarly, we denote by $\V_ 2^{\alpha }(I;E)$ the set of $2$-index maps $g:\Delta \to E$ such that $g_{tt}=0$ for every $t\in I$ and 
\begin{equation}
\label{def:omega_a_2}
|g_{st}|_{E} \leq \omega (s,t)^{\alpha }\,,
\end{equation} 
for all $(s,t) \in \Delta ,$ and some regular control $\omega .$ Note that $g\in \V^{\alpha }_1(I;E)$ if and only if $\delta g\in \V^{\alpha }_2(I;E)$.
If $I=[0,T],$ the corresponding semi-norm $|\cdot |_{\V_2^\alpha }$ in $\V^a_2(I;E)$ is given by the infimum of $\omega (0,T)^\alpha $ over every possible control $\omega $ such that \eqref{def:omega_a_2} holds. Alternatively, it is equivalently defined as the $1/\alpha $-variation of $g,$ that is
\begin{equation}
\label{nota:p-var}
|g|_{1/\alpha \mathrm{-var};I;E}:=\left(\sup_{\pp\equiv (t_i)\in\PP(I)}\sum_{(\pp)}|g_{t_i t_{i+1}}|_{E}^{1/\alpha }\right)^\alpha ,
\end{equation}
where 
\[
\PP(I):=\Big\{\pp\subset I:\exists l\geq 2,\pp=\{t_1=\inf I<t_1<\dots<t_l=\sup I\}\Big\}
\]
is the set of partitions of $I,$
and where, throughout the paper, we use the notational convention:
\begin{equation}
\label{summation_pp}
\sum_{(\pp)}h_{t_i t_{i+1}}\enskip  \overset{\mathrm{def}}{=}\enskip \sum\nolimits_{i=1}^{\#\pp-1}h_{t_i t_{i+1}}
\end{equation}
for any 2-index element $h.$
The equivalence between the semi-norms $|\cdot |_{\V^\alpha }$ and $|\cdot |_{1/\alpha \mathrm{-var}}$ will be investigated in Remark \ref{rem:control} below (these quantities are in fact equal).

By $ \V^{\alpha }_{2,\text{loc}}(I;E)$ we denote the space of maps $g:\Delta\to E$ such that there exists a countable covering $\{I_k\}_k$ of $I$ satisfying $g \in  \V^{\alpha }_2(I_k;E)$ for any $k$.
We also define the set $\V^{1+}_2(I;E)$ of ``negligible remainders'' as
\[
\V^{1+}_2(I;E):=\bigcup_{\alpha >1} \V_ 2^{\alpha }(I;E),
\]
and similarly for $\V^{1+}_{2,\text{loc}}(I;E).$

Furthermore, we denote by $\AC(I;E)\subset \V^1_1(I;E)$ the set of \emph{absolutely continuous functions}, that is: $f\in \AC(I;E)$ if and only if for every $\epsilon >0$ there exists $\delta >0$ such that for every non-overlapping family $(s_1,t _1),\dots,(s _n,t _n)\subset I$ with $\sum (t_i-s_i)<\delta ,$ then 
\[
\sum_{1\leq i\leq n}|\delta f_{s_i t_i}|_{E}<\epsilon \,.
\]

Given $\alpha \in(1/3,1/2]$ and $K\in\N,$ recall that a continuous ($K$-dimensional) $1/\alpha $-rough path is a pair \begin{equation}
\label{analytic_conditions}
\Z\equiv(Z^{k },\ZZ^{k \ell })_{1\leq k ,\ell \leq K}\quad \text{in}\quad \V^{\alpha }_2(I;\R^{K}) \times \V^{2\alpha}_2(I;\R^{K\times K}),                                                                                                                                                                                                                                                  \end{equation}
such that Chen's relations hold, namely:
\begin{equation}\label{chen_relations_gene}
\delta Z_{s\theta t}^k =0\,,\quad \delta \ZZ_{s\theta t}^{k \ell }=Z_{s\theta }^k Z_{\theta t}^\ell \,,\quad\text{for}\enskip (s,\theta ,t)\in\Delta ^2\,,\enskip 
\enskip 1\leq k ,\ell \leq K.
\end{equation}
We refer the reader to the monographs \cite{friz2010multidimensional,friz2014course} for a thorough introduction to the rough path theory.
We will denote by $\mathscr C^\alpha (I;\R^K) $ the set of all continuous rough paths as above.
It is endowed with the metric $d_{\mathscr C^\alpha }$ defined by
\begin{equation}
\label{RP_metric}
d_{\mathscr C^\alpha }(\Z^1,\Z^2):=|Z_{0\cdot }^1-Z_{0\cdot }^2|_{L^\infty(I)} + |Z^1-Z^2|_{1/\alpha -\mathrm{var}} +|\ZZ^1-\ZZ^2|_{1/(2\alpha) -\mathrm{var}},
\end{equation} 
for which it is complete. Note that, although $d_{\mathscr C^\alpha }$ is a function of the difference $\Z^1-\Z^2,$ it is definitely not a norm, because $\mathscr C^\alpha (I;\R)$ is \emph{not} a linear space.
For any element $z\in\V^1(I;\R^K),$ there is a canonical lift $S_{2}(z)\equiv(Z,\ZZ)$ in $\mathscr C(I;\R^K)$ defined as
\[
Z:=\delta z, \quad \text{and for}\enskip k,\ell \in\{1,\dots,K\}:\quad 
\ZZ^{k\ell }_{st}:=\iint_{\Delta _{[s,t]}}\d z_{r_2}^{\ell }\d z_{r_1}^{k},\quad (s,t)\in\Delta _I.
\]

We shall denote by $\mathscr C^\alpha _g(I;\R^K)\subset \mathscr C^\alpha (I;\R^K)$ the subset consisting of \emph{geometric rough paths}. By definition, $\mathscr C^\alpha _g(I;\R)$ corresponds to the closure of the canonical lifts $S_2(z),$ where $z\in\V^1(I;\R^K)$, with respect to the rough path metric \eqref{RP_metric}.
For a geometric rough path $\Z\equiv(Z,\ZZ),$ the symmetric part of the $2$-tensor $(\ZZ^{k\ell })_{1\leq k,\ell \leq K}$ is fully determined by the first component as follows:
\begin{equation}
\label{geometricity}
\sym\ZZ_{st}^{k\ell }:=\frac{\ZZ^{k \ell }+\ZZ^{\ell k }}{2}=\frac12Z^{k }Z^{\ell },\quad \text{for any}\enskip 1\leq k ,\ell \leq K
\end{equation}
(see \cite[Chapter 9]{friz2010multidimensional}).

We will consider the usual Lebesgue and Sobolev spaces in the space variable: $L^p\equiv L^p(\O)$, $W^{k,p}\equiv W^{k,p}(\O),$ for $k\in\N,$ and $p\in[1,\infty],$ and denote their respective norms by $|\cdot |_{L^p},|\cdot |_{W^{k,p}}.$
The notation $\|\cdot\|_{r,q}$ will be used for the norm in $L^r(I;L^q(\O)),$ namely:
\[
\|f\|_{r,q}:=\left(\int_{I}\left(\int_\O |f(t,x)|^q\d x\right)^{r/q}\d t\right)^{1/r}
\]
(note that in contrast  to the literature on deterministic PDEs, we write the time variable first, or with a subscript). To emphasize the domain of time integrability we sometimes write
$
\|\cdot\|_{r,q;I}.
$
We recall that $W^{k,p}_{\rm{loc}}(\R^d)$ is the space of functions $f$ such that  for every compact set $K\subset\R^d$ there holds $f|_K\in W^{k,p}(K)$.

We also write $C(I;E)$ for the space of continuous function with values in some Banach space $E,$ endowed with the norm $\|f\|_{C(I;E)}:=\sup_{r\in I}|f_r|_{E}.$
Throughout the paper we will make extensively use of an energy space which is the Banach space
\begin{equation}
 \label{nota:Bc}
 \Bc=\Bc_{I}:=C(I;L^2(\R^d))\cap L^2(I;W^{1,2}(\R^d))
\end{equation}
and we will sometimes write $\Bc_{s,t}$ as an abbreviation for $\Bc_{[s,t]},$  $s<t.$

Given Banach spaces $X,Y,$
we will denote by $\L(X,Y)$ the space of linear, continuous maps from $X$ to $Y,$ endowed with the operator norm.
For $f$ in $X^*:=\L(X,\R),$ we denote the dual pairing by
\[
~_{X^*}\big\langle f,g \big\rangle{_X}
\]
(i.e.\ the evaluation of $f$ at $g\in X$).
When they are clear from the context, we will simply omit the underlying spaces and write $\langle f,g\rangle$ instead.

\subsection{Unbounded rough drivers}
In the sequel, we call a \emph{scale} any sequence $(\G_k,$ $\nnnn{\cdot }{k})_{k\in \N}$ of Banach spaces such that $\G_{k+1}$ is continuously embedded into $\G_k,$ for each $k\in \N$.

For each $k\in \N,$ we will also denote by $\G_{-k}$ the topological dual of $\G_k,$ i.e.\
\begin{equation}\label{nota:E_k_negative}
\G_{-k}:= (\G_{k})^*\,.
\end{equation}
Except for the case $\G_k:= W^{k,2},$ we do \emph{not} identify $\G_0$ with its dual, hence 
a (minor) disadvantage of the latter notation is that in general 
\[
\G_0\neq \G_{-0}\,.
\]
\begin{definition}\label{def:rough_driver}
For a given $\alpha \in(1/3,1/2]$, a pair of $2$-index maps $\B\equiv(B,\BB)$ is called a \emph{continuous unbounded $1/\alpha$-rough driver} with respect to the scale $(\G_k)_{k\in \N}$, if 
\begin{enumerate}[label=(RD\arabic*)]
 \item \label{RD1}
$B_{st}\in \L\left(\G_{-k},\G_{-k-1}\right)$ for $k\in\{0,1,2\},$
$\BB_{st}\in\L\left(\G_{-k},\G_{-k-2}\right)$ for $k\in\{0,1\},$
and there exists a regular control $\omega_B :\Delta\to\R_+ $ such that
\begin{equation}\label{bounds:rough_drivers}
\left[
\begin{aligned}
&|B_{st}|_{\L(\G_{-0},\G_{-1})}\,,\quad
|B_{st}|_{\L(\G_{-2},\G_{-3})}
\leq\omega _B(s,t)^\alpha\,, 
\\
&|\BB_{st}|_{\L(\G_{-0},\G_{-2})}\,,\quad
|\BB_{st}|_{\L(\G_{-1},\G_{-3})}
\leq \omega_B (s,t)^{2\alpha }\,,
\end{aligned}
\right.
\end{equation}
for every $(s,t)\in\Delta .$
\item \label{RD2}
Chen's relations hold true, namely, for every $(s,\theta ,t)\in\Delta ^2:$
\begin{equation}\label{chen}
\delta B_{s\theta t}=0\,,\quad \delta \BB_{s\theta t}=B_{\theta t}B_{s\theta } \,,
\end{equation}
as linear operators on $\G_{-k},k=0,1,2,$ resp.\ $k=0,1.$
\end{enumerate}
\end{definition}

We will always understand the driver $\B$ \emph{in the sense of distributions}, namely 
we assume that each $\G_k$ for $k\in \N$ is canonically embedded into $\mathscr D'(\O),$ and that for $u\in \G_{-0},$ $(s,t)\in\Delta ,$
the element $B_{st}u$ (resp.\ $\BB_{st}u$) is \emph{defined}
as the linear functional on $\G_{1}$ (resp.\ $\G_2$)
given by
\[
\begin{aligned}
&\langle B_{st}u,\phi\rangle  =\langle u,B^*_{st}\phi \rangle\,,\quad \forall \phi \in \G_{1}\,,
\\
&\text{resp.\ }\quad \langle\BB_{st}u,\psi\rangle =\langle u,\BB^*_{st}\phi\rangle\,,\quad \forall \phi \in\G_{2}\,.
\end{aligned}
\]
In the context of \eqref{rough_PDE}
we let
\begin{equation}\label{nota:formal_adj}
\left[\begin{aligned}
B^*_{st}\phi &:=
Z^{k}_{st}\left(-\partial _i(\sigma ^{ki}\phi) +\nu^k \phi\right) \,,
\\
\BB^*_{st}\phi &:=
\ZZ^{k\ell}_{st}\Big(\partial _j(\sigma ^{\ell j}\partial _i(\sigma ^{ki}\phi ))
- \partial _j(\sigma ^{\ell j}\nu^k  \phi )-\nu^{\ell }\partial _i(\sigma ^{ki}\phi )+\nu ^\ell \nu ^k \phi 
\Big)
\end{aligned}\right.
\end{equation}
for a.e.\ $x\in\O$ and every $\phi \in W^{2,\infty},$ assuming that the coefficents $\sigma ,\nu $ are regular enough
(see the assumption \eqref{assumption_sigma_nu} below).
\subsection{Assumptions on the coefficients and the main result}
Throughout the paper, we assume that we are given an elliptic operator $A$ under the form \eqref{nota:A}, which is to be understood weakly, namely for $u\in\Bc,$ $\phi \in W^{1,2}(\R^d),$ and $(s,t)\in\Delta _I,$ we let
\begin{multline}
\label{def:A}
\big\langle \int_s^tAu_r\d r,\phi\big\rangle :=\iint_{[s,t]\times\R^d}\Big[-a^{ij}(r,x)\partial _ju_r(x)\partial _i\phi(x)
+b^i(r,x)\partial _iu_r(x)\phi(x)
\\
+ c(r,x)u_r(x)\phi(x) \Big]\d x\d r,
\end{multline}
where the assumptions below will ensure in particular that the former makes sense.
\begin{assumption}[Uniform ellipticity condition]
\label{ass:a}
The matrix $(a^{ij}(t,x))_{1\leq i,j\leq d}$ is symmetric, measurable with respect to each of its variables and there
exist constants $M,m>0$ such that for a.e.\ $(t,x):$
\begin{equation}\label{uniform_ellipticity}
m \sum_{i=1}^d\xi _i^2\leq \sum_{1\leq i,j\leq d} a^{ij}(t,x)\xi _i\xi _j\leq M\sum_{i=1}^d\xi _i^2\,,\quad \xi \in \R^d\,.
\end{equation}
\end{assumption}
We also need assumptions on integrability of the coefficients $b$ and $c$, depending on the spatial dimension $d\in\mathbb{N}.$
\begin{assumption}
 \label{ass:b_c}
We assume
\begin{equation}\label{assumption:b_c}
b\in L^{2r}\big(I;L^{2q}(\O;\R^d)\big)\,\quad \text{and}\quad c\in L^r\big(I;L^q(\O;\R)\big)\,,
\end{equation}
where the numbers $r\in[1,\infty)$ and $q\in(\max(1,\tfrac{d}{2}),\infty)$ are such that
\begin{equation}
\label{values:r_q}
\frac{1}{r}+\frac{d}{2q}\leq 1\,.
\end{equation}
\end{assumption}
The reason for these restrictions will appear in the use of the following interpolation inequality.
\begin{proposition}\label{pro:interpolation_inequality}
If $f$ belongs to $L^\infty(I;L^2)\cap L^2(I;W^{1,2}),$ then  one has also $f\in L^\rho (I;L^\kappa  )$ for every $\rho, \kappa $ such that
\begin{equation}\label{conditions}
\frac1\rho +\frac{d}{2\kappa }\geq \frac d4\quad\text{and}\quad \left[\begin{aligned}
&\rho \in[2,\infty]\,, \quad \kappa \in[2,\tfrac{2d}{d-2}]\quad\text{for}\enskip d>2
\\
&\rho \in(2,\infty]\,,\quad \kappa \in[2,\infty)\quad \text{for}\enskip d=2
\\
&\rho \in[4,\infty]\,,\quad \kappa \in[2,\infty]\quad \text{for}\enskip d=1\,.
\end{aligned}\right.
\end{equation}
In addition, there exists a constant $\beta>0 $ (not depending on $f$ in the above space) such that
\begin{equation}\label{interpolation_inequality}
\|f\|_{L^\rho (I;L^\kappa )}\leq \beta \|f\|_{L^\infty(I;L^2)\cap L^2(I;W^{1,2})}\equiv\beta \left(\|\nabla f\|_{L^2(I;L^2)}+\esssup_{s\in I}|f_r|_{L^2}\right)\,.
\end{equation}
\end{proposition}
\begin{proof}
The proof relies on the complex interpolation
(see \cite{lunardi2009interpolation})
\begin{equation}\label{interpolating}
L^\rho L^\kappa =[L^\infty L^2,L^2L^p]_{\theta }\,,
\end{equation}
for $\theta :=\frac 2\rho $ and $p:=2(1+ \rho (\frac 1\kappa -\frac12))^{-1}.$
Then, thanks to Young Inequality, write 
\[
\|f\|_{L^\rho L^\kappa }\leq C\|f\|^{1-2/\rho  }_{L^\infty L^2} \|f\|^{2/\rho }_{L^2L^p}\leq C'\left(\|f\|_{L^\infty L^2}+ \|f\|_{L^2L^p}\right)\,,
\]
and \eqref{interpolation_inequality} follows from the Sobolev embedding theorem. For instance when $d>2,$
we have $W^{1,2}(\O )\hookrightarrow L^p(\O )$ if
\begin{equation}
 \label{ineq_p}
2\leq p\equiv \frac{2}{1-\rho (\frac12-\frac1\kappa )}\leq \frac{2d}{d-2}\,,
\end{equation}
but from $\frac1\rho +\frac{d}{2\kappa }\geq \frac d4,$ it holds $\rho \leq \frac2d(\frac12-\frac{1}{\kappa })^{-1},$ and thus
$p\leq 2/(1- \frac2d)\equiv\frac{2d}{d-2},$ and since $p\geq 2,$ it implies \eqref{ineq_p}.
The cases $d=1,2$ are left to the reader.
For a proof under the stronger assumption that $\frac1\rho +\frac{d}{2\kappa }=\frac d4,$ we refer to Theorem 2.2 in \cite[Chap.\ II (3.4)]{ladyzhenskaya1968linear}.
\end{proof}
As an immediate consequence of Proposition \ref{pro:interpolation_inequality}, we have the following.
Let $r$ and $q$ be as in \eqref{values:r_q} and let $u$ in $\Bc.$
It is easily seen that \eqref{values:r_q} implies \eqref{conditions} for the exponents $\rho := \frac{2r}{r-1}$ and $\kappa :=\frac{2q}{q-1}.$
Hence, for some universal constant $\beta \equiv\beta (r,q),$ one has 
\begin{equation}
\label{consequence_rq}
\|u\|_{\frac{2r}{r-1},\frac{2q}{q-1}}\leq \beta \|u\|_{\Bc}\,.
\end{equation}

Concerning the coefficients of the driver, we assume the following.
\begin{assumption}
\label{ass:sigma_nu}
The coefficients $\sigma ,\nu $ are such that
\begin{equation}\label{assumption_sigma_nu}
\sigma \in W^{3,\infty}(\O ,\R^{d\times K})
\quad\text{and}\quad 
\nu \in W^{2,\infty}(\O ,\R^K)\,.
\end{equation}
\end{assumption}

Throughout the paper, we will extensively make use of the following scales
\begin{equation}\label{scale1}
\left[\begin{aligned}
&W^{k,2}(\O)\,,\quad \n{\cdot }{k}:=|\cdot |_{W^{k,2}},
\\
&W^{k,\infty}(\O )\,,\quad \nn{\cdot }{k}:=|\cdot |_{W^{k,\infty}},
\end{aligned}\right.
\end{equation}
for $k\in \N,$ and their corresponding negative-exponent counterparts as in \eqref{nota:E_k_negative}. Note that, except when $p=\in\{1,\infty\},$ Sobolev spaces of negative order are usually defined by the relation $W^{-k,(p)}=\left(W^{k,p/(p-1)}\right)^*,$ hence here we have for instance $\nnnn{\cdot }{-1,(p)}=|\cdot |_{W^{-1,\frac{p}{p-1}}}.$
Owing to Leibniz rule, it is seen that for a.e.\ $x$ in $\O$ and every $(s,t)$ in $\Delta :$
\[
|\nabla ^{k}B_{st}^*\phi|
\leq \omega _Z(s,t)^\alpha \left(|\sigma |_{W^{k+1,\infty}}+|\nu |_{W^{k,\infty}}\right)\sum_{0\leq \ell \leq k+1}|\nabla ^\ell \phi |
\,,\quad k=0,1,2\,,
\]
whereas
\[
|\nabla ^k\BB_{st}^*\phi|
\leq \omega _Z(s,t)^{2\alpha }\left(|\sigma |_{W^{k+2,\infty}}+|\nu |_{W^{k+1,\infty}}\right)\sum_{0\leq \ell \leq k+2}|\nabla ^\ell \phi |
\,,\quad k=0,1\,.
\]
The driver $\B=(B,\BB)$ defined in \eqref{nota:formal_adj} fulfills the properties of Definition \ref{def:rough_driver}, namely
\begin{equation}\label{B_rough_driver}
\left[
\begin{aligned}
&\B\enskip \text{is an $1/\alpha $-unbounded rough driver, with respect to}\enskip 
\\
&\text{each of the scales $(W^{k,2})_{k\geq 0}$ and $(W^{k,\infty})_{k\geq 0}.$}
\end{aligned}\right.
\end{equation}
Moreover, we can set
\begin{equation}\label{nota:omega_B}
\omega _B(s,t):=C\left(|\sigma |_{W^{3,\infty}},|\nu |_{W^{2,\infty}}\right)\omega _Z(s,t)
\end{equation}
for a constant depending on the indicated quantities.

We now need a suitable notion of solution for the problem \eqref{rough_PDE}.
The following definition corresponds to that given in \cite{bailleul2017unbounded}
(see also \cite{diehl2017stochastic}).
\begin{definition}\label{def:weak_sol}
Let $T>0$, $I:=[0,T]$ and $\alpha \in(1/3,1/2]$.
Let $\B=(B,\BB)$ be a continuous $1/\alpha $-unbounded rough driver with respect to a given scale $(\G_k)_{k\in \N}, $
and let $\mu \equiv\mu_{t}$ be a path of finite variation in $\G_{-1}.$

A continuous path $g:I\to \G_{-0}$ is called a \emph{weak solution} to the
rough PDE
\begin{equation}
\label{rough_PDE_gene}
\d g=\d \mu+\d \B g
\end{equation}
on $I\times \O,$
with respect to the scale $(\G_k)_{k\in \N}$, if for every $\phi \in \G_3$, and every $(s,t)\in\Delta $, there holds
\begin{equation}\label{nota:solution}
\langle\delta g_{st},\phi \rangle=\langle\delta \mu _{st},\phi \rangle+ \langle g_s,B^*_{st}\phi \rangle +\langle g_s,\BB_{st}^*\phi \rangle +\langle g_{st}^\natural,\phi \rangle\,,
\end{equation}
for some $g^\natural\in \V_ {2,\text{loc}}^{1+}(I;\G_{-3}).$
\end{definition}

It will be seen in particular that for $u\in\Bc,$ the drift $\mu _t:=\int_0^tAu_r\d r,$ where $Au$ is as in \eqref{def:A}, defines indeed an element of $\V^1_1(I;W^{-1,2})$ (in fact $\mu\in \AC(I;W^{-1,2})$). Hence the problem \eqref{rough_PDE} formulates now as finding $u\in\Bc$ such that
\begin{equation}
\label{zakai}
\begin{cases}
\d u=(Au)\d t +\d \B u\,,\quad \text{in}\enskip I\times \R^d\,,
\\
u_0\quad \text{given in}\enskip L^2(\R^d).
\end{cases}
\end{equation}
where the equation must be understood in the sense of Definition \ref{def:weak_sol}, with respect to the scale $(W^{k,2}(\R^d))_{k\in\N}.$
We have now all in hand to state our main results.
\begin{theorem}\label{thm:solvability}
Fix $T>0,$ $I:=[0,T],$ assume that $u_0\in L^2$, and consider coefficients $a,b,c,\sigma ,\nu $ such that the assumptions \ref{ass:a},\ref{ass:b_c} and \ref{ass:sigma_nu} hold.
There exists a unique weak solution $u$ to \eqref{zakai} such that 
\begin{equation}\label{nota:F}
u\in\Bc_{0,T}:=C(I;L^2)\cap L^2(I;W^{1,2})\,.
\end{equation}
In addition the following It\^o formula holds for the square of $u$:
\begin{equation}\label{ito_formula}
\langle\delta u^2_{st},\phi \rangle=2\int_{s}^t\langle Au,u\phi \rangle\d r+\langle u^2_s,\hat B_{st}^*\phi \rangle +\langle u^2_s,\hat\BB^*_{st}\phi \rangle+ \langle u^{2,\natural}_{st},\phi \rangle\,,
\end{equation}
for every $\phi $ in $W^{3,\infty}$ and $(s,t)$ in $\Delta ,$
where $\hat\B$ is the unbounded rough driver obtained by replacing $\nu $ by $\hat \nu :=2\nu $ in \eqref{nota:formal_adj}, and where the remainder $u^{2,\natural}$ belongs to $\V_ {2,\rm{loc}}^{1+}(I;(W^{3,\infty})^*).$

Finally the $\Bc$-norm of $u$ is estimated as
\begin{equation}\label{energy_bound}
\|u\|_{\Bc_{0,T}}\leq C\Big(\alpha ,T,m,M,\|b\|_{2r,2q},\|c\|_{r,q},\omega _Z,|\sigma |_{W^{3,\infty}},|\nu |_{W^{2,\infty}}\Big)|u_0|_{L^2}\,,
\end{equation}
for a constant depending on the indicated quantities, but not on $u_0\in L^2.$
\end{theorem}
The uniqueness and existence parts of the above theorem will be proven separately.
Existence will be addressed via an approximation argument, and it will be jointly proved with the following continuity result (see Section \ref{sec:continuity}).

\begin{theorem}\label{thm:continuity}
Under the conditions of Theorem \ref{thm:solvability},
let $\mathcal{P}_{m,M}$ be defined as those coefficents $a^{ij}\in L^\infty(I\times\O)$ such that Assumption \ref{ass:a} holds, and let $\mathscr C_g^\alpha $ be the space of continuous geometric rough paths of finite $1/\alpha $-variation.
The solution map 
\begin{equation}\label{cont}
\left[
\begin{aligned}
&\mathfrak{S}:L^2\times \mathcal{P}_{m,M}\times L^{2r}L^{2q}\times L^rL^q\times W^{3,\infty}\times W^{2,\infty}\times \mathscr C_g^\alpha 
\longrightarrow C\big(I;W^{-1,2}_{\mathrm{loc}}\big)\cap L^2(I;L^2_{\mathrm{loc}})
\\
&(u_0,a,b,c,\sigma ,\nu,\Z)\longmapsto \mathfrak{S}(u_0,a,b,c,\sigma ,\nu,\Z):=
\begin{cases}
\text{the solution given}\\
\text{by Theorem \ref{thm:solvability}}
\end{cases}
\end{aligned}\right.
\end{equation}
is continuous.
\end{theorem}

Some remarks are in order.

\begin{remark}
Note that by interpolation it follows from \eqref{energy_bound} and \eqref{cont} that the solution map is continuous in $L^\kappa(I;W^{\gamma,2}_{\mathrm{loc}})$ whenever $\gamma=\theta-(1-\theta)$ and $\kappa\leq 2/\theta$ for some $\theta\in (0,1)$.
\end{remark}

\begin{remark}
The map $u\equiv u_t(x)$ given by Theorem \ref{thm:solvability}  solves $\d u=A u\d t + \d \B u$
in the sense that for every $\phi  $ in $W^{3,2}$
and all $(s,t)$ in $\Delta ,$ it holds
\begin{equation}
 \label{integral_equation}
\langle\delta u_{st},\phi \rangle 
=\int_s^t\langle Au,\phi  \rangle\d r
+\int_s^t\langle (\sigma \cdot \nabla  +\nu )u,\phi \rangle \d \Z\,,
\end{equation} 
where the latter makes sense as a rough integral -- note that, as a by-product of Proposition \ref{pro:apriori} below, we have that for each $1\leq \ell \leq K$ the path $t\mapsto\langle(\sigma^\ell  \cdot \nabla +\nu^\ell  )u_t,\phi  \rangle$ \emph{is controlled by $(Z^k)_{1\leq k\leq K}$} with Gubinelli derivative $t\mapsto\big\langle (\sigma^\ell  \cdot\nabla +\nu^\ell )(\sigma^k \cdot\nabla +\nu^k )u_t,\phi \big\rangle\,,1\leq k\leq K.$
\end{remark}
\begin{remark}[the case of time-dependent coefficents]
It should be possible to assume that $\sigma ,\nu $ are coefficients depending on space and time, under the assumption that the path $t\mapsto(\sigma (t,\cdot ),\nu (t,\cdot )) $ be ``controlled by $Z$''. This stems from the fact that, roughly speaking,
rough integrals are themselves rough paths (up to a canonical lift). See \cite[Chap.\ 7]{friz2014course} for related results in finite dimensions.

To be more precise, assume for simplicity that $\nu =0,$ $K=1,$
and let $V:=W^{3,\infty}(\O;\R^{d}).$ Consider $\sigma \in \V^\alpha ([0,T];V),$ \emph{controlled by $Z,$} in the sense that there is some $\sigma _t'(x)$ in $\V^\alpha ([0,T];V)$
such that
\[
\Big((s,t)\in\Delta\mapsto\sigma _s-\sigma _s'Z_{st}\Big)\quad \text{belongs to}\quad  \V^{2\alpha }_2([0,T];V)\,.
\]
We can then define the driver $\B$ as the 2-index family of unbounded operators given
for $\varphi $ in $W^{1,2}$ by
\[
B_{st}\varphi := \int_s^t \sigma \cdot \nabla \varphi \d \Z 
=\lim_{\substack{|\pp|\to0\\ \pp\in\PP([s,t])}}\sum_{(\pp)}\sigma _{t_i}\cdot \nabla \varphi Z_{t_i t_{i+1}}+\sigma '_{t_{i}}\cdot \nabla \varphi  \ZZ_{t_i t_{i+1}}\,,
\]
where we take the limit in the space $W^{-1,2},$ and make use of the summation convention \eqref{summation_pp}.
Next, one defines a second component for $\B,$ via the rough integral
\[
\BB_{st}\varphi :=\int_{s}^t B_{s,r} \d B_r(\varphi)  =
\lim_{\substack{|\pp|\to0\\ \pp\in\PP([s,t])}}\sum_{(\pp)}
B _{st_i}B_{t_i t_{i+1}}\varphi 
+\sigma _{t_i}\cdot \nabla\big(\sigma _{t_i}\cdot \nabla \varphi \big)\ZZ_{t_i t_{i+1}}\,,
\]
for $\varphi $ in $W^{2,2},$
where it can be easily checked that the former limit makes sense as an element of $W^{-3,2}.$

With this definition at hand, it is a simple exercise to check that:
(i) $\B\equiv(B,\BB)$ is an $1/\alpha $-unbounded rough driver on the scale $(W^{k,2})_{k\in \N};$
(ii) any weak solution of the equation ``$\d u =Au\d t +\d\B u$'' (in the sense of Definition \ref{def:weak_sol}), satisfies the integral equation \eqref{integral_equation}.

However, our existence and uniqueness results do not immediately apply, because counterparts of Propositions \ref{pro:renormalizable} and \ref{pro:drift} are missing in this context.
Due to the size of the present paper, and because it would make the presentation more cumbersome,
we refrain from giving their proof.
\end{remark}

\begin{remark}[Geometricity and the stochastic parabolicity assumption]
\label{rem:sto_par_1}
Equations of the form \eqref{rough_PDE} are well studied in the case where the driving path is a Brownian motion,
and they are known to be solvable in the It\^o sense, only under the so-called stochastic parabolicity condition.
It is interesting to understand why our results do not apply in a non-geometric context, unless one makes some similar assumption.

For simplicity let us consider the case where $K=d=1,$ $\nu =0=b=c,$ $a>0$ and $\sigma\in\R,$ then the same formal computations as before lead to the following first order approximation:
\begin{equation}
 \label{solution_theory}
u_t-u_s=\int_s^ta\partial _{xx}u_r\d r +\sigma \partial _xu_sZ_{st} +\sigma ^2\partial _{xx}u_s\ZZ_{st} + u_{st}^{\natural},
\end{equation} 
where as before we expect the ``error term'' $u^\natural$ to be at most of size $o(t-s),$ because $\alpha>1/3.$
It turns out that this intuition is wrong in general.

Recall that, at an informal level, $\ZZ_{st}$ should be thought of as an ``offline interpretation for $\iint_{\Delta _{st}}\d Z\d Z$'', and therefore should be subject the algebraic conditions \eqref{chen_relations_gene}.
These are just the translation of the presumable additivity property ``$\int_s^\theta +\int_\theta ^t=\int_{s}^t$'', together with the linearity of the integral map $f\mapsto\int f\d Z.$
While it seems natural to postulate that $Z_{st}:=Z_t-Z_s$ for any reasonable definition of the term ``$\int_s^t\d Z$'',
there are in fact infinitely many possibilities for the second entry if one only imposes \eqref{chen_relations_gene}, in which we add the analytic conditions \eqref{analytic_conditions}.
There is however a priviledged choice consisting in letting
\begin{equation}
\label{geometric_enhancement}
\ZZ_{st}:=\ZZ^{\mathrm{geo}}_{st} \enskip \overset{\mathrm{def}}{=}\enskip  \frac12(Z_{st})^2,\quad \quad \text{for}\enskip (s,t)\in\Delta _I,
\end{equation}
in which case $\Z$ is easily seen to be geometric (see \eqref{geometricity}).

As it turns out, every enhancement fulfilling \eqref{chen_relations_gene} and \eqref{analytic_conditions} is given by 
\begin{equation}
\label{bracket}
\ZZ_{st}:=\ZZ^{\mathrm{geo}}_{st}- \frac12\delta [\Z]_{st},
\end{equation} 
where $[\Z]$ is called the \emph{bracket} of $\Z$ and denotes any element of $\V_ 1^{2\alpha }.$
(Note that \eqref{bracket} only defines the bracket up to the initial value $[\Z]_0,$ which will be taken equal to $0$ by convention.)

Next, applying the chain rule for rough paths (in the form given by \cite[Proposition 7.6]{friz2014course}), we formally obtain the following equation for $u^2:$
\begin{equation}
\label{ito_formula_intro}
u_t^2-u_s^2=\int_s^t2au_r\partial _{xx}u_r\d r + \sigma \partial _x(u_s^2)Z_{st} + \sigma ^2\partial _{xx}(u^2_{s})\ZZ_{st}
+\int_s^t\sigma ^2(\partial _xu_r)^2\d [\Z]_r +u^{2,\natural}_{st},
\end{equation}
where the equality should be understood for any $(s,t)\in\Delta _I,$ in some Sobolev space of negative order, say $(W^{3,\infty})^*.$

To be more explicit, let us consider the case where the driving path is an enhancement $\mathbf W\equiv (W,\mathbb W)$ of a Brownian motion over some probability space. Then, the choice \eqref{geometric_enhancement} is nothing but the Stratonovitch iterated integral,
and it can be shown that any solution in the sense \eqref{solution_theory} is indeed a solution in the usual Stratonovitch sense.  See \cite[Chap.\ 5]{friz2014course} for related results.
In this case, no particular assumption on the coefficients is necessary. If however one formulates the above equation in the sense of It\^o, then one has to choose the It\^o enhancement, that is $\mathbb W_{st}:=\frac12W_{st}^2 - (t-s)/2,$ or equivalently $[\mathbf W]_t:=t.$
Now, testing \eqref{ito_formula_intro} with the constant function $1,$ integrating by parts,
it is seen that in order to get energy dissipation, one has to make the so-called \emph{stochastic parabolicity} assumption:
\begin{equation}
\label{stochastic_parabolicity}
a-\frac12\sigma ^2>0.
\end{equation}
\end{remark}

\section{Analysis of rough partial differential equations}
\label{sec:analysis}

In this section, we introduce the basic tools necessary for the study of rough PDEs of the form \eqref{rough_PDE_gene}, namely, the rough Gronwall Lemma and an a priori estimate on the remainder in \eqref{nota:solution}. The results were originally introduced in \cite{bailleul2017unbounded, deya2016priori} where we also refer the reader for a more detailed introduction. The statements we present below are slightly different than in \cite{bailleul2017unbounded, deya2016priori} and hence for readers convenience we also include the proofs. These tools represent the core of our analysis and will be repeatedly used in order to obtain a priori estimates leading to existence as well as uniqueness of weak solutions.

\subsection{Rough Gronwall Lemma}
An important ingredient in order to obtain uniform estimates on weak solutions of \eqref{zakai} is the following generalized Gronwall-like estimate.
\begin{lemma}[Rough Gronwall]
\label{lem:gronwall}
Let $G:I\equiv[0,T]\to \R_+.$ Assume that we are given a regular control $\omega ,$ and a constant $L>0$
such that provided $\omega (s,t)\leq L,$
\begin{equation}\label{rel:gron}
\delta G_{st}\leq \left(\sup_{s\leq r\leq t} G_r\right)\omega (s,t)^{1/\kappa }+\varphi (s,t)\,,
\end{equation}
for some superadditive map $\varphi:\Delta _I\to\R,$ and a given constant $\kappa >0.$

Then, there exists a constant $\tau _{\kappa ,L}>0$ depending
on $\kappa $ and $L$ only such that
\begin{equation}
\label{concl:gron}
\sup_{0\leq t\leq T}G_t\leq 2\exp\left(\frac{\omega (0,T)}{\tau _{\kappa ,L}}\right)\left[G_0+\sup_{0\leq t\leq T}\left|\varphi (0,t)\right|\exp\left(\frac{-\omega (0,t)}{\tau _{\kappa ,L}}\right)\right].
\end{equation}
\end{lemma}
\begin{remark}
A proof under slightly different hypotheses can be found in \cite{deya2016priori}.
Note that here we allow for $\varphi $ which has no sign. This may be relevant in the context of stochasic PDEs, where typically relations such as \eqref{rel:gron} may involve $\varphi (s,t):=M_t-M_s,$ the increments of a martingale $M.$
\end{remark}
\begin{proof}
Let $\tau :=L\wedge (2e^{2})^{-\kappa },$
Since the control $\omega$ is regular,
there exists an integer $K\geq 2$ and a sequence $t_0\equiv0<t_1<\dots<t_{K-1}<t_{K}\equiv T$ such that for each 
$k$ in $\{1,\dots,K-1\},$
\begin{equation}
\label{tk}
\omega (0,t_k)=k\tau\,,
\end{equation}
while for $k=K$ it holds $\omega (0,t_K)\equiv\omega (0,T)\leq K\tau .$
For $k\in\{0,\dots,K-1\},$ using superadditivity we obtain the property:
\begin{equation}
\label{tk_tk1}
\omega (t_k,t_{k+1})\leq\tau\,.
\end{equation}

Next, for $t\in[0,T],$ we let:
\[
G_{\leq t}:=\sup_{0\leq r\leq t}G_r\,,\quad 
H_t:=G_{\leq t}\exp\left(-\frac{\omega (0,t)}{\tau }\right)\,,\quad
H_{\leq t}:=\sup_{0\leq r\leq t}H_r\,.
\]
Fix $t\in[t_{k-1},t_k]$ for some $k\in\{1,\dots,K\}$. Note that since $\tau \leq L$, we may apply the estimate \eqref{rel:gron} on each subinterval $[t_{i},t_{i+1}]$. Hence, using \eqref{rel:gron}, \eqref{tk_tk1} and the superadditivity of $\varphi ,$ it holds
\begin{align*}
G_t&=G_0+\sum\nolimits_{i=0}^{k-2}\delta G_{t_i t_{i+1}} +\delta G_{t_{k-1}t}
\\
&\leq G_0+\tau ^{1/\kappa }\Big(\sum\nolimits_{i=0}^{k-2}G_{\leq t_{i+1}}
+G_{\leq t}\Big)+\sum\nolimits_{i=0}^{k-2}\varphi (t_i,t_{i+1})+\varphi (t_{k-1},t)
\\
&\leq G_0+\tau ^{1/\kappa }\sum\nolimits_{i=0}^{k-1}H_{t_{i+1}}\exp\Big(\frac{\omega (0,t_{i+1})}{\tau }\Big) +\varphi (0,t)
\intertext{which, according to \eqref{tk} and the properties of the exponential map, is bounded above by}
& \quad G_0+\tau ^{1/\kappa }H_{\leq T}\exp(k+1) +\varphi (0,t)\,.
\end{align*}
By the fact that $\omega (0,t)\geq \omega (0,t_{k-1}),$
we deduce the following estimate on $H:$
\[
\begin{aligned}
H_t\leq \left\{G_0 + |\varphi (0,t)| +\tau ^{1/\kappa }\exp(k+1)H_{\leq t}\right\}\exp\left(\frac{-\omega (0,t)}{\tau }\right)
\\
\leq G_0+\sup_{t\leq T}\left\{|\varphi (0,t)|\exp\left(\frac{-\omega (0,t)}{\tau }\right)\right\} + \tau ^{1/\kappa }e^2 H_{\leq T}\,,
\end{aligned}
\]
According to our definition of $\tau ,$ this yields the bound:
\[
H_{\leq T}\leq \frac{1}{1-e^2\tau ^{1/\kappa }}\left(G_0+\sup_{t\leq T}\left\{|\varphi (0,t)|\exp\left(\frac{-\omega (0,t)}{\tau }\right)\right\}\right)\,,
\]
from which \eqref{concl:gron} follows.
\end{proof}

\subsection{Remainder estimates}

As in the classical theory, the rough Gronwall Lemma presented above is a simple tool that, among others, permits to obtain a priori estimates for rough PDEs of the general form \eqref{rough_PDE_gene}. It should be stressed however that the most delicate part of this argument is to estimate  the remainder in such a way that Lemma \ref{lem:gronwall} is indeed applicable. This step is by no means trivial, in particular, due to unboundedness of the involved operators (in the noise terms as well as in the deterministic part of the equation) and the corresponding loss of derivatives. The key observation is that there is a tradeoff between space and time regularity which can be balanced using a suitable interpolation technique. To this end, let us introduce the notion of smoothing operators on a given scale $(\mathcal{G}_k)$.

\begin{definition}
Assume that we are given a scale $(\G_k)_{k\in\N}$ with a topological embedding
$$\cup_{k\in\N}\G_k\hookrightarrow\mathscr D',$$
and let $J_\eta:\mathscr D'\to \mathscr D',\eta \in(0,1),$ be a family of linear maps.
For $m\geq 1$
we say that $(J_\eta )_{\eta \in(0,1)}$ is an \emph{$m$-step family of smoothing operators on $(\G_k)$}  provided
for each $k\in\N:$
\begin{enumerate}[label=(J\arabic*)]
\item \label{J1}
$J_\eta$ maps $\G_k$ onto $\G_{k+m} ,$
for every $\eta \in(0,1),$
\end{enumerate}
and there exists a constant $C_J>0$ such that
for any $\ell\in\N$ with $|k-\ell|\leq m:$ 
\begin{enumerate}[label=(J\arabic*)]
\setcounter{enumi}{1} 
\item \label{J2}
if $0\leq k\leq \ell\leq m+1 ,$ then
\begin{equation}
\label{estimates_J}
|J_\eta |_{\L(\G_k, \G_\ell )}\leq \frac{C_J}{\eta ^{\ell-k} }\,,\quad \text{for all}\quad \eta \in(0,1)\,;
\end{equation} 
\item \label{J3}
if $0\leq \ell \leq k\leq m+1,$ then
\begin{equation}
\label{estimates_J2}
|\id-J_\eta |_{\L(\G_{k}, \G_{\ell })}\leq C_J\eta^{k-\ell } \,,
\quad \text{for all}\quad \eta \in(0,1)\,.
\end{equation}
\end{enumerate}
\end{definition}

\begin{remark}
Whenever the spaces $\G_k$ are Sobolev-like with exponents of integrability different from $1,\infty$, examples of $1$-step families of smoothing operators
are provided by
\begin{equation}\label{families_sobolev}
J_\eta :=(\id-\eta^2 \Delta )^{-1}\,\quad\text{or} \quad J_\eta :=e^{\eta^2 \Delta }
\end{equation}
(under suitable assumptions on the domain of $\Delta $). In $W^{k,2}(\R^d)$ this is easily seen using the Fourier transform: for instance, concerning the first family we can use the inequality
\[
\frac{1}{1+(\eta|\xi |)^2}-1\leq C_\alpha (\eta |\xi |)^{2\alpha }\,,
\]
which holds for every $\alpha \in[0,1],$
and then apply Parseval Identity (the cases $\alpha=\frac12,1 $ yield \ref{J3}). Note that smoothing operators similar to the second family above are also extensively used in \cite{otto2016quasilinear}.

If $\G_k$ consists of functions $\phi $ supported on the whole space $\R^d,$ one can simply let $J_\eta \phi :=\varrho_\eta *\phi ,$ where $\varrho_\eta $ is a well-chosen approximation of the identity.
The existence of such smoothing families when elements of $\G_k$ are compactly supported is not trivial
and is therefore treated in Appendix \ref{ss:smoothing}.
\end{remark}

Let us now formulate the main result of this section.

\begin{proposition}[Estimate of the remainder]\label{pro:apriori}
Let $\alpha \in(1/3,1/2]$ and fix an interval $I\subset [0,T].$ Let $\B=(B,\BB)$ be a continuous unbounded $1/\alpha $-rough driver on a given scale $\G_k,\nnnn{\cdot }{k},k\in \N,$ endowed with a two-step family of smoothing operators $(J_\eta )_{\eta \in(0,1)}.$
Consider a drift $\mu \in \V^1_1(I;\G_{-1})$ and let $\omega _\mu $ be a regular control such that
\begin{equation}\label{hyp:apriori}
\nnnn{\delta \mu _{st}}{-1}
\leq \omega _\mu (s,t),\quad \text{for every}\enskip (s,t)\in\Delta _I.
\end{equation}
Let $g$ be a weak solution of
\eqref{rough_PDE_gene} in the sense of Definition \ref{def:weak_sol},
such that $g$ is controlled over the whole interval $I$, that is:
$g^\natural\in \V^{1+}_2(I;\G_{-3}).$

Then, there exist constants $C,L>0,$ such that, if the interval $I$ satisfies the smallness condition
 $\omega _B(I)\leq L,$ then it holds for each $(s,t)\in\Delta_I :$
\begin{equation}
\label{estimate_remainder}
\nnnn{g^\natural_{st}}{-3}\leq C\left(\sup_{s\leq r\leq t}\nnnn{g_r}{-0}\omega _B(s,t)^{3\alpha }+\omega _\mu (s,t)\omega _B(s,t)^{\alpha }\right)\,.
\end{equation}

Furthermore, 
define for each $(s,t)\in\Delta _I $ the first order remainder
\begin{equation}\label{nota:sharp}
g^{\sharp}_{st}:=\delta g_{st} -B_{st}g_s.
\end{equation}
Then, under the smallness condition $(\omega _\mu +\omega _B)(I)\leq L,$ it holds true that for every $(s,t)\in\Delta _I:$
\begin{align}
\label{bounds:gubinelli}
\nnnn{g^\sharp_{st}}{-1}
&\leq C\left(\omega _\mu(s,t)+ \sup_{s\leq r\leq t}\nnnn{g}{-0}\big(\omega _\mu(s,t)^\alpha +\omega _B(s,t)^\alpha\big) \right)
\,,
\\
\label{bounds:gubinelli3}
\nnnn{g^{\sharp}_{st}}{-2}
&\leq C\left(\omega _\mu(s,t) +\sup_{s\leq r\leq t}\nnnn{g}{-0}\omega _B(s,t)^{2\alpha }\right)\,,
\\
\label{bounds:gubinelli2}
\nnnn{\delta g_{st}}{-1}
&\leq C\left(\omega _\mu(s,t)+ \sup_{s\leq r\leq t}\nnnn{g}{-0}\big(\omega _\mu(s,t)^\alpha +\omega _B(s,t)^\alpha\big) \right)\,.
\end{align}
\end{proposition}

Before we proceed to the proof of Proposition \ref{pro:apriori}, we need to establish some properties related to controls and $1/\alpha $-variation spaces.
Working with $\V^\alpha $ rather than $C^\alpha $ is necessary here, in order to deal with the low-regularity assumptions \eqref{uniform_ellipticity}-\eqref{assumption:b_c}.
The $1/\alpha $-variation setting also turns out to be very convenient because of the fact that control functions enjoy some ``nice properties''. For instance, it is easily seen that a product
\begin{equation}
\label{product_controls}
\omega _1(s,t)^a\omega _2(s,t)^b
\end{equation} 
where $a+b\geq 1,$ and $\omega _1,\omega _2$ are controls, is also a control, and it is regular if both are regular. See \cite{friz2010multidimensional}.
Another interesting property of controls is as follows.
In Proposition \ref{pro:apriori}, we are interested in taking the ``sharpest control'' majorizing $(s,t)\mapsto|g^{\natural}_{st}|_{\G_{-3}}.$ Note that a supremum of controls is not a control in general, however if for any $(s,t)\in \Delta _I$ one defines
\begin{equation}
\label{nota:w_natural}
\omega _\natural(s,t):=\inf\{\omega (s,t):\omega \in\mathfrak C_{s,t}\}
\end{equation}
\begin{equation}
\label{nota:}
\mathfrak C_{s,t}:=\left\{\omega :\Delta _{[s,t]}\to\R_+\,,\text{control}\enskip |\enskip 
\forall(\theta,\tau)\in\Delta _{[s,t]},\enskip 
\omega (\theta,\tau)\geq \nnnn{g^\natural_{\theta \tau }}{-3}\enskip 
\right\}\,,
\end{equation}
then the following holds.

\begin{lemma}
\label{lem:control}
The mapping $\omega _\natural:\Delta _I\to \R_+$ defined in \eqref{nota:w_natural} is a regular control.
Moreover, it is equal to $(s,t)\in\Delta _I\mapsto |g^\natural|_{1\mathrm{-var},\G_{-3};[s,t]}.$
\end{lemma}

\begin{proof}[Proof of Lemma \ref{lem:control}.]
For $(s,\theta ,t)\in\Delta ^2 ,$ since both $\mathfrak C_{s,\theta },\mathfrak C_{\theta ,t}$ contain $\mathfrak C_{s,t},$
we have by definition:
\begin{equation}
\label{ineq:controls}
\omega _\natural(s,\theta )+\omega_\natural (\theta ,t)\leq \omega (s,\theta )+\omega (\theta ,t)\leq \omega (s,t)\,,
\end{equation}
for every $\omega \in\mathfrak C_{s,t}.$ Taking the infimum in \eqref{ineq:controls}, we see that \eqref{axiom:control} holds, so that $\omega _\natural$ is indeed a control.

Now, the mapping $\omega :(s,t)\in\Delta _I\mapsto|g^\natural|_{1\mathrm{-var},\G_{-3};[s,t]}$ is a regular control (see \cite[Proposition 5.8]{friz2010multidimensional}).
Therefore, letting $(s,t)\in\Delta _I,$ it only remains to prove that
\begin{equation}
\label{prop:omega_var}
\omega(s,t) \leq \omega _\natural(s,t),\quad \quad\enskip (s,t)\in\Delta _I.
\end{equation}
But for every partition $\pi\in\PP([s,t]),$ taking an arbitrary $\bar\omega $ in $\mathfrak{C}_{s,t},$ it holds:
\[
\sum\nolimits_{(\pi)}\nnnn{g^{\natural}_{t_it_{i+1}}}{-3}\leq \sum\nolimits_{(\pi)}\bar\omega (t_i,t_{i+1})\leq \bar\omega (s,t)\,,
\]
Taking sucessively the supremum over $\pi\in\PP([s,t])$ of the left hand side, and then the infimum over $\bar\omega \in\mathfrak C_{s,t},$ we see that \eqref{prop:omega_var} holds. This proves the lemma.
\end{proof}

As a consequence of \eqref{product_controls} and Lemma \ref{lem:control}, the conclusion \eqref{estimate_remainder} of Proposition \ref{pro:apriori} above could be changed to the following:
\begin{equation}
\label{conclusion_alternative}
\omega _{\natural}(s,t)\leq C\left(\sup_{s\leq r\leq t}\nnnn{g_r}{-0}\omega _B(s,t)^{3\alpha }+\omega _\mu (s,t)\omega _B(s,t)^{\alpha }\right),
\end{equation}
which will be the form proved below.

\begin{remark}
\label{rem:control}
In fact, the proof of Lemma \ref{lem:control} is easily modified to yield the following more general property.
Denote by $E$ any Banach space. For any $\alpha >0,$ if $g\in\V^\alpha _{2,\mathrm{loc}}(I;E)$, then
$(s,t)\in\Delta _I\mapsto|g|_{1/\alpha \mathrm{-var}}^{1/\alpha }$ is a regular control and moreover it holds for any $(s,t)\in\Delta _I:$
\[
|g|_{1/\alpha \mathrm{-var},E;[s,t]}^{1/\alpha }=\inf\{\omega (s,t):\omega \in \mathfrak C^\alpha _{s,t}\},
\]
where $\mathfrak C^\alpha _{s,t}:=\{\omega:\Delta _{[s,t]}\to\R_+ \enskip \text{control s.t.\ } \omega(\theta ,\tau ) ^\alpha \geq |g_{\theta \tau }|_{E}\enskip \forall(\theta ,\tau )\in\Delta _{[s,t]}\}.$
\end{remark}

Letting $A$ and $\B$ as in the above discussion, ...

We now have all in hand to prove Proposition \ref{pro:apriori}.
\begin{proof}[Proof of Proposition \ref{pro:apriori}]
\textit{Proof of \eqref{estimate_remainder}.}
To estimate the remainder $g^\natural_{st}$, we apply $\delta $ to \eqref{nota:solution} and use Chen's relations \eqref{chen}, leading to 
\begin{equation}\label{leading_to}
\begin{aligned}
\delta g^\natural_{s\theta t}
&=B_{\theta t}\delta g_{s\theta }-B_{\theta t}B_{s\theta }g_s+\BB_{\theta t}\delta g_{s\theta }
\\
&=B_{\theta t}g^{\sharp}_{s\theta }
+\BB_{\theta t}\delta g _{s\theta }
\\
&=:\T _\sharp+\T _\delta \,,
\end{aligned}
\end{equation}
for every $(s,\theta ,t)\in\Delta ^2_I.$
Note that  by definition of $g^\sharp$ in \eqref{nota:sharp} and the original equation \eqref{nota:solution}, it holds
\begin{equation}
\label{g_sharp_equal}
g^\sharp_{s\theta }\equiv\delta g_{s\theta } -B_{s\theta }g_s=\delta \mu _{s\theta }+\BB_{s\theta }g_s+g^{\natural }_{s\theta }
\end{equation}
hence it is both an element of $\G_{-1}$ and $\G_{-2},$ (with different time regularities).
This basic fact will be exploited in the sequel, in order to apply Proposition \ref{pro:sewing}.

In \eqref{leading_to}, test against $\phi \in \G_3$ such that $\nnnn{\phi }{3}\leq 1.$
Substituting \eqref{g_sharp_equal} into \eqref{leading_to} and then making use of $J_\eta $ for some $\eta \in(0,1)$ (to be fixed later on), there comes
\[\begin{aligned}
\langle\T_\sharp,\phi \rangle
&\equiv\langle \delta \mu _{s\theta }+\BB_{s\theta }g_s + g^\natural_{s\theta },B^*_{\theta t}J_\eta \phi \rangle
+
\langle \delta g_{s\theta }-B_{s\theta }g_s,B^*_{\theta t}(\id-J_\eta)\phi \rangle\,.
\end{aligned}
\]
Each term above can be estimated using the bounds on $\B$ as well as $\omega _\mu $ and the estimates \eqref{estimates_J}-\eqref{estimates_J2}.
Denoting for simplicity
\begin{equation}\label{nota:G}
G:=\sup_{r\in I}\nnnn{g_r}{-0}\,,
\end{equation}
we have for every $(s,\theta ,t)\in\Delta _I:$
\begin{equation}\label{going_back1}
\begin{aligned}
&|\langle\T_\sharp,\phi \rangle|
\leq
\omega _\mu(s,\theta )\nnnn{B_{\theta t}^*J_\eta \phi }{1}
+\langle g_s, \BB_{s\theta }^*B_{\theta t}^*J_\eta\phi \rangle 
+\langle g_{s\theta }^\natural, B_{\theta t}^*J_\eta \phi \rangle 
\\[0.5em]
&\quad \quad \quad \quad \quad \quad \quad \quad \quad \quad 
+\langle \delta g_{s\theta },B_{\theta t}^*(\id-J_\eta )\phi \rangle
+\langle  g_s,B_{s\theta }^*B_{\theta t}^*(\id-J_\eta )\phi \rangle
\\
&\leq C_J\Big(
\omega _\mu(s,t)\omega _B(s,t)^\alpha
+ G\omega _B(s,t)^{3\alpha}
+ \frac{\omega _\natural(s,t)\omega _B(s,t)^\alpha }{\eta } +2G \omega _B(s,t)^\alpha \eta ^2
+ G\omega _B(s,t)^{2\alpha }\eta 
\Big).
\end{aligned}
\end{equation}
This being true for any $\eta \in(0,1),$ we can make a choice that equilibrates the various terms. Namely, we let
\begin{equation}\label{eta_1}
\eta :=4C_J|\Lambda|\omega _B(s,t)^{\alpha }\,,\quad 
\end{equation}
where $|\Lambda|$ is the constant from the Sewing Lemma, see Proposition \ref{pro:sewing}.
Now, the smallness condition
\begin{equation}
\label{close_to_each_other}
\omega _B(I)< L:=\left(\frac{1}{4C_J|\Lambda|}\right)^{1/\alpha }
\end{equation}
guarantees that $\eta$ belongs to $(0,1),$ so that \eqref{eta_1} is indeed a valid choice.
In that case, we end up with the inequality
\begin{equation}\label{borne_delta2}
\nnnn{\T_\sharp}{-3}\leq C\Big(\omega _\mu(s,t)\omega _B(s,t)^\alpha +G\omega _B(s,t)^{3\alpha }\Big) + \frac{\omega _{\natural}(s,t)}{4|\Lambda|} 
\end{equation}
for some constant $C>0$ depending only on $|\Lambda|$ and $C_J,$ where $(s,\theta ,t)\in\Delta _I^2$ is arbitrary.
The previous computations also show that for $\phi \in \G_1$ with $\nnnn{\phi}{1}\leq 1:$
\begin{equation}\label{computations:sharp}
\begin{aligned}
|\langle g^\sharp_{s\theta },\phi \rangle|
&\leq \omega _\mu(s,\theta ) \nnnn{J_\eta \phi}{1}
+G
\omega _B(s,\theta )^{2\alpha }\nnnn{J_\eta \phi }{2}
+\omega _{\natural}(s,\theta )\nnnn{J_\eta \phi }{3}
\\[0.5em]
&\quad \quad \quad \quad \quad 
+\nnnn{\delta g_{s\theta }}{-0}\nnnn{(\id-J_\eta )\phi}{0}
+G\nnnn{B^*_{s\theta }(\id-J_\eta )\phi }{0}
\\
&\leq
C_J\Big(
\omega _\mu(s,\theta )
+G\frac{\omega _B(s,\theta )^{2\alpha }}{\eta}
+\frac{\omega _{\natural}(s,\theta )}{\eta ^2}
+2G
\eta 
+G
\omega _B(s,\theta )^{\alpha }
\Big)
\end{aligned}
\end{equation}
where we have used again \eqref{estimates_J}.
Choosing $\eta $ similarly as in \eqref{eta_1},
we see that
$g^\sharp$ belongs to $\V_ 2^\alpha (I;\G_{-1}),$
with an estimate:
\begin{equation}\label{bound_sharp}
\nnnn{g^\sharp_{s\theta }}{-1}\leq 
C\big(\omega _\mu(s,\theta )  +  G\omega _B(s,\theta )^\alpha \big)
+\frac{\omega _\natural(s,\theta )}{4|\Lambda|\omega _B(s,\theta )^{2\alpha }},\quad \text{for every}\enskip (s,\theta )\in \Delta _I.
\end{equation}
Now, for the second term in \eqref{leading_to} we can use \eqref{bound_sharp}:
taking $\phi \in \G_3$ with $\nnnn{\phi }{3}\leq 1,$ there comes
\begin{equation}\label{borne_delta1}
\begin{aligned}
|\langle \T_\delta,\phi \rangle|
&\equiv |\langle g^\sharp_{s\theta }+B_{s\theta }g_s, \BB_{\theta t}^*\phi  \rangle|
\\
&\leq
\nnnn{g_{s\theta }^{\sharp}}{-1}\nnnn{\BB_{\theta t}^*\phi }{1} + \nnnn{g_s}{-0}\nnnn{B_{s\theta }^*\BB_{\theta t}^*\phi }{0}
\\
&\leq
C\big(\omega _\mu(s,t)\omega _B(s,t)^{2\alpha}   + G\omega _B(s,t)^{3\alpha }\big)+\frac{\omega _\natural(s,t)}{4|\Lambda|}
+ G \omega _B(s,t)^{3\alpha }
\,.
\end{aligned}
\end{equation}
From the bounds \eqref{borne_delta1} and \eqref{borne_delta2}, we obtain
\[
\nnnn{\delta g^\natural_{s\theta t}}{-3}\leq C\left(
\omega _\mu (s,t)\omega _B(s,t)^\alpha  + G\omega _B(s,t)^{3\alpha }
\right)
+\frac{\omega _\natural(s,t)}{2|\Lambda |}\,,
\]
for some absolute constant $C>0,$ independently of $(s,\theta ,t)\in\Delta _I^2.$
We are now in position to apply the Sewing Lemma, Proposition~\ref{pro:sewing}, so that $g^\natural=\Lambda\delta g^\natural$ and moreover for all $(s,t)\in\Delta _{I},$ it holds
\[
\nnnn{g^\natural_{st}}{-3}\leq \omega _{\natural}'\equiv C\left(\omega _\mu (s,t)\omega _B(s,t)^\alpha + G\omega_B (s,t)^{3\alpha }\right) +\frac12\omega _{\natural}(s,t)\,.
\]
Recalling that $\omega _\natural$ is the smallest control $\omega _\natural'$ such that the inequality above holds (see Lemma \ref{lem:control}), we eventually obtain
\[
\nnnn{g^\natural_{st}}{-3}\leq 2C\left(\omega _\mu(s,t)\omega _B(s,t)^\alpha  +G\omega _B^{3\alpha }(s,t)\right)\,,
\]
which proves \eqref{estimate_remainder}.

\item[\indent\textit{Proof of \eqref{bounds:gubinelli}}.]
From \eqref{computations:sharp} and \eqref{estimate_remainder}, there holds (omitting time indexes):
\[
|\langle g^\sharp,\phi \rangle|
\leq
C\left(\omega _\mu
+G\left(\frac{\omega _B^{2\alpha }}{\eta}+\omega _B^\alpha +\eta 
\right)
+\frac{1}{\eta ^2}\left(\omega _\mu\omega _B^\alpha +G\omega _B^{3\alpha }\right)\right)\nnnn{\phi }{1}\,.
\]
Provided that $(\omega_\mu +\omega _B) (I)<L$ (hence guaranteeing that $\eta :=(\omega _\mu+\omega _B)^\alpha$ belongs to $(0,1)$) we end up with the a priori estimate
\[
\nnnn{g^\sharp_{st}}{-1}\leq 
C\left(\omega _\mu(s,t)+ G\big(\omega _\mu(s,t)^\alpha +\omega _B(s,t)^\alpha\big) \right)
\,,
\]
for $(s,t)\in\Delta _I$ (here we have used the trivial bounds $\omega _B\leq\omega _\mu+\omega _B ,$ $1-\alpha >\alpha ,$ as well as $(\omega _\mu+\omega _B)^\alpha \leq C_\alpha (\omega _\mu^\alpha +\omega _B^\alpha )$).

\item[\indent\textit{Proof of \eqref{bounds:gubinelli2}}]
Writing that $\delta g=g^{\sharp} + Bg,$ we see that the same bound holds for $\delta g$ instead of $g^\sharp$, namely
for every $(s,t)\in\Delta _I:$
\[
\nnnn{\delta g_{st}}{-1}\leq 
C\left(\omega _\mu(s,t)+ G\big(\omega _\mu(s,t)^\alpha +\omega _B(s,t)^\alpha\big) \right)
\]
(with another such universal constant $C$).

\item[\indent\textit{Proof of \eqref{bounds:gubinelli3}}]
Proceeding similarly, we have
\[
\langle g^{\sharp},\phi \rangle
\leq C\Big(\omega _\mu+G\left(\omega _B^{2\alpha }+\eta ^2+\omega _B^\alpha \eta \right)
+\frac{1}{\eta }\left(\omega _\mu\omega _B^\alpha +G\omega _B^{3\alpha }\right)\Big)\nnnn{\phi }{2}\,,
\]
where each term above is evaluated at $(s,t)\in\Delta _I.$
Whence, taking $\eta :=\omega _B(s,t)^\alpha ,$ we end up with the estimate
\[
\nnnn{g^{\sharp}_{st}}{-2}\leq C\left(\omega _\mu(s,t) +G\omega _B(s,t)^{2\alpha }\right)
\]
for every $(s,t)\in\Delta _I,$ for some universal constant $C>0.$
\end{proof}
\begin{remark}[On the link between weak solutions and the notion of controlled path]\label{rem:gubinelli}
Following  Gubinelli's approach on rough paths \cite{gubinelli2004controlling}, it would be natural in this setting to define the set 
$\mathscr D_B$ of \emph{controlled paths} as those couples $g,g'$ in $\V_ 1^{\alpha} (I;\G_{-1}),$ such that the first order remainder
\begin{equation}\label{def:controlled_path}
(s,t)\in\Delta \mapsto g^{\sharp}_{st}:=\delta g_{st}-B_{st}g'_s \,
\end{equation}
defines an element of $\V^{2\alpha}_{2,\rm{loc}}(I;\G_{-2})$ (meaning that a cancellation occurs in \eqref{def:controlled_path}).

If $g$ denotes a weak solution of \eqref{rough_PDE_gene}, in the sense of Definition \ref{def:weak_sol},
we have in fact $(g,g')\in\mathscr D_B$ with $g'=g.$
Therefore, given $(\G_k),(J_\eta),\mu ,$ and $\B$ as in Proposition \ref{pro:apriori}, we can alternatively define a weak solution
to \eqref{rough_PDE_gene} as an element $(g,g)$ of $\mathscr D_B$ such that \eqref{nota:solution} holds,
i.e.\ a continuous path
$g:[0,T]\to \G_{-0}$ such that
\[
\left\{
\begin{aligned}
&\delta g\in \V_ {2,\mathrm{loc}}^{\alpha}(I,\G_{-1})\,,
\\
&g^\sharp\equiv\delta g-Bg\in \V_ {2,\rm{loc}}^{2\alpha}(I,\G_{-2})\,,
\\
&g^\natural\equiv\delta g-Bg-\BB g -\delta \mu \in \V_ {2,\rm{loc}}^{3\alpha}(I,\G_{-3})\,.
\end{aligned}
\right.
\]
\end{remark}
\section{The energy inequality}
\label{sec:energy}
In this section we assume that the driving path $z$ is \emph{smooth} and we establish an estimate on the $\Bc$-norm of a weak solution to \eqref{zakai} which only depends on the rough path norm of the corresponding canonical lift $\Z$ of $z$. 
However it should be noted that the conclusion of Proposition \ref{pro:energy} below remains true provided the square $u^2$ satisfies the equation \eqref{ito_formula}, which will be shown to hold for any weak solution $u$, see Section \ref{sec:uniqueness}.

\subsection{The main statement}
\label{ss:main}

Using the standard theory for non-degenerate parabolic PDEs
(see  \cite[Chap.\ III]{ladyzhenskaya1968linear}), we know that
there exists a unique $u$ in the Banach space $\Bc$ (note that this space is denoted by $V_ 2^{1,0}$ in the latter reference),
solving the the evolution problem
\begin{equation}\label{approximate_pb}
\frac{\partial u}{\partial t} -A u=\left(\sigma^{ki} \partial _i u  +\nu^ku\right)\dot z^{k}
\,,\quad u_0\in L^2\,,
\end{equation}
in the sense that
\begin{multline}\label{solution:ladyzhenskaya}
-\iint_{I\times\O}u\partial _t\eta\d t\d x +\iint_{I\times\O }\left(a^{ij}\partial _ju\partial _i\eta -b^i\partial _iu\eta -cu\eta\right)\d t\d x
\\
=\iint_{I\times\O }\left(\sigma ^{ki}\partial _i\eta +\nu ^ku\eta \right)\dot z^k\d t\d x
\,,
\end{multline}
for every test function $\eta$ in the Sobolev space
\[
\mathscr W^{1,1}_2(I\times\O):=\{\eta \in L^2(I\times\O ):\nabla \eta ,\partial _t\eta \in L^2(I\times\O )\}\,,
\]
and such that $\eta$ vanishes, in the sense of traces at $t=T$ and $t=0.$

Our aim is to prove following.
\begin{proposition}[Energy inequality]\label{pro:energy}
Consider a \emph{smooth path} $z,$ together with its canonical geometric lift $\Z\equiv(Z,\ZZ),$
and let $\omega _Z$ be the control function $(s,t)\mapsto |Z|_{1/\alpha \mathrm{-var};[s,t]}^{1/\alpha }+|\ZZ|_{1/(2\alpha )\mathrm{-var};[s,t]}^{1/(2\alpha)}.$
Then, every weak solution of \eqref{zakai} satisfies
\begin{equation}\label{energy_inequality}
\sup_{0\leq t\leq T}|u_t|_{L^2}^2 + \int_0^T|\nabla u _r|^2_{L^2}\d r\leq C|u_0|_{L^2}^2\,,
\end{equation}
for a constant $C>0$ depending on the quantities $\omega _Z,|\sigma |_{W^{3,\infty}},|\nu|_{W^{2,\infty}},m,M,$ and $\|b\|_{2r,2q},$ $\|c\|_{r,q},$ but not on the individual element $u$ in $\Bc$.
\end{proposition}

Although $u$ does not belong to $\mathscr W_2^{1,1}$ a priori, by considering time averages of the form 
\[
u_h(t,x):=\frac1h\int_t^{t+h}u(t,\tau )\d \tau \,,
\]
(extended by zero if $t\notin[0,T-h]$) and passing to the limit $h\to0,$
it is seen that in \eqref{solution:ladyzhenskaya} we can formally test against 
\[
\eta (r,x):=\ind_{[s,t]}(r)\phi (x)u(r,x)
\]
with $\phi \in W^{1,\infty},$
(see the equality (2.13) in \cite[Chap.\ III.2]{ladyzhenskaya1968linear}
for the case where $\eta :=\ind_{[s,t]}u$, the proof being identical for $\eta $ as above).
This yields, for each $(s,t)$ in $\Delta ,$ and every $\phi $ in $W^{1,\infty}:$
\begin{multline}
\label{ladyzhenskaya_2}
\int_\O ((u_t)^2 -(u_s)^2)\phi\d x =2\iint_{[s,t]\times\O } \left(-a^{ij}\partial _ju\partial _i(u\phi )+b^i\partial _iu u\phi  + cu^2\phi \right)\d r\d x
\\
+\iint_{[s,t]\times\O } \left(\sigma ^{ki}\partial _i(u^2)\phi  +2\nu ^ku^2\phi \right)\dot z^k\d r\d x\,.
\end{multline}

\subsection{Proof of Proposition \ref{pro:energy}}
We are going to make use of the tools presented in Section \ref{sec:analysis}. More precisely, we will show that
\begin{itemize}
 \item suitable estimates relative to the scale $(W^{k,\infty})_{k\in \N}$
 hold for the drift part of \eqref{ladyzhenskaya_2}, i.e.\ for
\[
\enskip \int_0 ^{\cdot } uAu\d r\enskip
\]
understood as a linear functional on $W^{1,\infty}$;
\item equation \eqref{ladyzhenskaya_2} implies that $\d \,(u^2)=2\d \left(\int uA u\d r\right) +\d \hat\B (u^2)$ holds in the sense of Definition~\ref{def:weak_sol}.
\end{itemize}

\begin{remark}
\label{rem:controls_abc}
Taking $a,b,c$ such that Assumptions \ref{ass:a}-\ref{ass:b_c} hold true, and $u$ in $\Bc,$
the following quantities are regular controls
\begin{equation}
\label{controls_abc}
\left[
\begin{aligned}
&\aa(s,t):=(1+M^2)\left(\|\nabla u\|_{2,2;[s,t]}^2 + \|u\nabla u\|_{1,1;[s,t]}\right)\,,
\quad \quad 
\bb(s,t):= \left(\|b\|_{2r,2q;[s,t]}\right)^{2r}\,,
\\
&\cc(s,t):= \left(\|c\|_{r,q;[s,t]}\right)^r\,,
\quad \quad 
\uu(s,t):= \Big(\|u\|_{\frac{2r}{r-1},\frac{2q}{q-1};[s,t]}\Big)^{\tfrac{2r}{r-1}}\,.
\end{aligned}
\right.
\end{equation}
These are in fact \emph{absolutely continuous} in the following sense:
if we denote by $\omega$ any of the above, then for every $\epsilon >0,$
there is a constant $\delta _\epsilon >0$ with the property that for any non-overlapping family $(s_1,t_1),\dots(s_n,t_n)\subset I,$ with $\sum(t_i-s_i)\leq \delta _\epsilon ,$ then one has $\sum_{i=1}^n\omega (t_i,t_{i+1})\leq \epsilon .$ These basic facts be extensively used in the sequel.

Note that, without any further assumption on the coefficients,
these terms are in general not bounded above by a constant times $(t-s).$ This explains the necessity of working with the $\V^\alpha $ spaces instead of the H\"older spaces $C^\alpha .$
\end{remark}

An important observation is the following Lemma.
For convenience, and because it will be useful in the proof of Theorem \ref{thm:continuity}, we also include bounds on the drift term of $u$ in \eqref{solution:ladyzhenskaya}.
\begin{lemma}
\label{lem:drift_existence}
Given $u$ in $\Bc,$
define the drift terms
\begin{equation}
\label{nota:lambda_existence}
\langle\lambda _{t},\phi \rangle:=\int_0^t\langle A_ru_r,\phi \rangle\d r
\equiv
\iint_{[0,t]\times\O}(-a_r^{ij}\partial _iu_r\partial _j\phi
+ b_r^i\partial _iu_r\phi +c_ru_r\phi 
)\d r\d x\,,
\end{equation}
for $\phi$  in $W^{1,2},$ 
$(s,t)\in\Delta ,$
and
\begin{multline}
\label{nota:mu_existence}
\langle\mu _{t},\phi \rangle:=
2\int_0^t\langle u_rA_ru_r,\phi \rangle\d r
\equiv
2\iint_{[0,t]\times\O}\Big(-a_r^{ij}\partial _iu_r\partial _ju_r\phi
\\
-u_r a^{ij}_r\partial _iu_r \partial _j\phi 
+ b^i_r\partial _iu_ru_r\phi +c_r(u_r)^2\phi 
\Big)\d r\d x\,,
\end{multline}
for $\phi $ in $W^{1,\infty}.$
Then, there is a constant $C>0,$ depending only on
$T,r,q,M$ but not on $u,a,b,c$ in the spaces $\Bc,L^\infty,L^{2r}(L^{2q}),L^r(L^q),$
such that defining the controls $\aa,\bb,\cc,\uu$ as in Remark \ref{rem:controls_abc},
there holds for every $(s,t)$ in $\Delta _I:$
\begin{align}
\label{borne_controle_lambda}
&
\n{\delta \lambda _{st}}{-1}
\leq (t-s)^{1/2}\aa(s,t)^{1/2}+\bb(s,t)^{1/(2r)}\aa(s,t)^{1/2}(t-s)^{\frac{r-1}{2r}}
\\
\nonumber
&\quad \quad \quad \quad \quad \quad \quad \quad \quad \quad \quad 
\quad \quad \quad \quad 
+\cc(s,t)^{1/r}\uu(s,t)^{\frac{r-1}{2r}}(t-s)^{\frac{r-1}{2r}}
\leq C\left(1+\|u\|^2_{\Bc_{s,t}}\right)\,,
\intertext{and similarly:}
\label{borne_controle_mu}
&
\nn{\delta \mu_{st}}{-1}\leq
\aa(s,t) + \bb(s,t)^{1/(2r)}\aa(s,t)^{1/2}\uu(s,t)^{\frac{r-1}{2r}} 
+\cc(s,t)^{1/r}\uu(s,t)^{\frac{r-1}{r}}
\leq C\|u\|_{\Bc_{s,t}}^2.
\end{align}
\end{lemma}

\begin{proof}[Proof of Lemma \ref{lem:drift_existence}]
\textit{Proof of \eqref{borne_controle_lambda}.}
Take any $\phi \in W^{1,2}.$
For $u$ in $\Bc,$
we have
\[
-\iint_{\stO}a^{ij}\partial _ju\partial _i\phi\d r\d x \leq M \|\nabla u\|_{1,2;[s,t]}\n{\phi }{1}\leq M(t-s)^{1/2}\|\nabla u\|_{2,2;[s,t]}\n{\phi }{1}\,.
\]
By the equality
\begin{equation}
\label{complementary_powers_1}
\frac{1}{2r}+ \frac{1}{2} + \frac{r-1}{2r} =1\,,
\end{equation}
(and similarly for $q$),
H\"older Inequality yields:
\begin{equation}
\label{ineq_b_sobolev}
\iint_{[s,t]\times\O}b^i\partial _iu\phi\d r\d x
\leq 
\|b\|_{2r,2q;[s,t]}\|\nabla u\|_{2,2;[s,t]}(t-s)^{\tfrac{r-1}{2r}}|\phi |_{L^{\frac{2q}{q-1}}}\,.
\end{equation}
Now, in dimension one and two, $W^{1,2}$ embeds into every $L^p$ space for $p\in[1,\infty),$ so the term $|\phi |_{L^{\frac{2q}{q-1}}}$ is bounded by a constant times $\n{\phi }{1}.$
For $d>2,$ since by assumption
\[
q>\max(1,\frac{d}{2})=\frac{d}{2}\,,
\]
it is seen that
\[
\frac{2q}{q-1}<\frac{2d}{d-2}=:p^*\,.
\]
By the the Sobolev embedding theorem, we have
\[
W^{1,2}\hookrightarrow L^{p^*}\subset L^{\frac{2q}{q-1}}\,.
\]
Hence, in both cases, we see from \eqref{ineq_b_sobolev} that
\[
\iint_{[s,t]\times\O}b^i\partial _iu\phi
\leq 
\|b\|_{2r,2q;[s,t]}\|\nabla u\|_{2,2;[s,t]}(t-s)^{\frac{r-1}{2r}}\n{\phi }{1}\,.
\]

Similarly, we have for the last term
\[
\begin{aligned}
\iint_{[s,t]\times\O}cu\phi \d r\d x
\leq&
\|c\|_{r,q;[s,t]}\|u\|_{\frac{2r}{r-1},\frac{2q}{q-1};[s,t]}|\phi |_{L^{\frac{2q}{q-1}}}(t-s)^{\frac{r-1}{2r}}
\\
\leq&\|c\|_{r,q;[s,t]}\|u\|_{\frac{2r}{r-1},\frac{2q}{q-1};[s,t]}(t-s)^{\frac{r-1}{2r}}\n{\phi }{1}\,.
\end{aligned}
\]
Adding the above contributions yields the first part of \eqref{borne_controle_lambda}.

Next, from $\|u\nabla u\|_{1,1}\leq (t-s)^{1/2}\|u\|_{\infty,2}\|\nabla u\|_{2,2}$
it is clear that
\begin{equation}
\label{interpolation_aa}
\aa(s,t)^{1/2}\leq C(M,T) \|u\|_{\Bc_{s,t}}\,,
\end{equation}
whereas for the other terms, we use \eqref{consequence_rq}, so that
$\uu(s,t)^{\frac{r-1}{2r}}\leq \beta \|u\|_{\Bc_{s,t}}.$
This yields the second part of the estimate \eqref{borne_controle_lambda}.

\item[\indent\textit{Proof of \eqref{borne_controle_mu}.}]
Take any $\phi $ in $W^{1,\infty}.$
From H\"older Inequality, it holds true that
\begin{equation}
\label{bound:a_u}
\iint_{[s,t]\times \O}-a^{ij}\partial _ju\partial _i(u\phi ) 
\leq M\left(\|\nabla u\|^2_{2,2;[s,t]}+\|u\nabla u\|_{1,1;[s,t]}\right)\nn{\phi }{1}\,.
\end{equation}
Now, because of
\eqref{complementary_powers_1}
we have
\begin{equation}\label{bound:b_u}
\iint_{[s,t]\times\O}|u||b^i||\partial _iu||\phi |\d r\d x\leq 
\|b\|_{2r,2q;[s,t]}\|\nabla u\|_{2,2;[s,t]}\|u\|_{\frac{2r}{r-1},\frac{2q}{q-1};[s,t]}\nn{\phi }{0}
\end{equation}
as well as
\begin{equation}\label{bound:c_u}
\iint_{[s,t]\times\O}|c||u|^2|\phi |\d r\d x\leq \|c\|_{r,q;[s,t]}\|u\|_{\frac{2r}{r-1},\frac{2q}{q-1};[s,t]}^2
\nn{\phi }{0}\,.
\end{equation}

This yields the first part of the estimate \eqref{borne_controle_mu}

Making use again of the bounds
\eqref{interpolation_aa}-\eqref{consequence_rq}
we obtain the second part of \eqref{borne_controle_mu}.
\end{proof}
As a straightforward, but important consequence, we have the following result.
\begin{corollary}\label{cor:u_natural}
Given a smooth path $z$ and its canonical geometrical lift $\Z\equiv(Z,\ZZ),$ let $u$ be a weak solution of \eqref{approximate_pb}, in the sense of \eqref{solution:ladyzhenskaya}.
Define the path $u^2 :I\to L^1(\R^d)$ by $ u^2_t(x):=u_t(x)^2,$ for a.e.\ $(t,x)\in I\times\O.$

Then, $u^2$ is a weak solution
in the sense of Definition \ref{def:weak_sol} to
\begin{equation}\label{approximate_pb3}
\delta u^2_{st} =2\int_s^tuAu\d r +\hat B_{st} \left(u^2_s\right) +\hat \BB_{st}\left(u^2_s\right) +u^{2,\natural}_{st}\,,
\end{equation}
on the scale $(W^{k,\infty})_{k\in \N},$
where we denote by $\hat B\equiv(\hat B,\hat \BB)$ the $1/\alpha $-unbounded rough driver given by \eqref{nota:formal_adj}, with $\nu $ replaced by $\hat \nu :=2\nu .$
\end{corollary}
\begin{proof}
For simplicity, in this proof we let for $k\in\{1,\dots,K\}:$
$\sigmab^k:=\sum\nolimits_{i} \sigma ^{k ,i}(x)\partial _{i}.$
Define the 2-index distribution-valued map
\[
u^{2,\natural}_{st}:=\delta u^2_{st}-2\int_s^t(Au)u\d r - \hat B_{st}(u^2_s)-\hat \BB_{st}(u^2_s)\,.
\]
Using the equation \eqref{ladyzhenskaya_2} twice we see that for any $\phi \in C_c^\infty(\O):$
\begin{equation}
\label{expansion_existence}
\begin{aligned}
\langle u^{2,\natural}_{st},\phi \rangle=&
\int_s^t\big\langle u^2_r -u^2_s ,(\sigmab^{k,*}+\hat \nu^k )\phi \big\rangle\d z^k_r
-\langle u^2_s,\hat\BB_{st}^*\phi \rangle
\\
=&
\iint_{\Delta ^2_{[s,t]}}\Big\langle  u^2_\tau  ,(\sigmab^{\ell ,*} +\hat\nu^{\ell } )(\sigmab^{k,*}+\hat\nu^k )\phi \Big\rangle \d z^\ell _\tau \d z^k_r
 -\langle u^2_s,\hat\BB_{st}^*\phi \rangle
\\
&\quad \quad \quad \quad
+ 2\iint_{\Delta ^2_{[s,t]}}\langle u_\tau A_\tau u_\tau ,(\sigmab^{k,*}+\hat\nu^k )\phi \rangle\d \tau \d z^k_r
\\
=&\iint_{\Delta ^2_{[s,t]}} \Big\langle \delta u^2_{s\tau }  ,(\sigmab^{\ell ,*} +\hat\nu^\ell  )(\sigmab^{k,*}+\hat\nu^k )\phi \Big\rangle\d z^\ell _\tau \d z^k_r
\\
&\quad \quad \quad \quad 
+ 2\iint_{\Delta ^2_{[s,t]}}\langle u_\tau A_\tau u_\tau ,(\sigmab^{k,*}+\hat\nu )\phi \rangle\d \tau \d z^k_r
\quad =:\T_{\delta u^2}+\T_A
\,.
\end{aligned}
\end{equation} 
From Assumption \ref{ass:sigma_nu} on $\sigma ,\nu ,$ and the fact that, by the classical theory for \eqref{solution:ladyzhenskaya}, $u$ belongs to the space $\Bc,$ it is immediately seen that every term above makes sense.
It remains to show that each of the terms above belongs to $\V^{1+}_{2,\rm{loc}}(I;\R),$ with a bound depending linearly on $\nn{\phi }{3}.$

For the first term, observe that
\begin{multline*}
\sup_{|\phi |_{W^{1,\infty}}\leq 1}
\langle \phi,\delta u^2_{st}\rangle 
\leq
\nn{\mu _{st}}{-1}
+\sup_{|\phi |_{W^{1,\infty}}\leq 1}
\int_s^t\langle\sigma ^{ki}\partial _i(u^2),\phi \rangle\d z^k
\\
+\sup_{|\phi |_{W^{1,\infty}}\leq 1}\int_s^t
\langle\hat\nu ^ku^2,\phi \rangle\d z^k
\quad \leq
\varepsilon(s,t)\,.
\end{multline*}
where $\varepsilon(s,t)$ is a control depending on $|z|_{1-\rm{var}},|\nu |_{L^\infty},|\sigma |_{W^{1,\infty}},\sup_{r\in[s,t]}|u_r|^2_{L^2}$ and the control $\omega _\mu $ given in Lemma \ref{lem:drift_existence}.
Consequently, we have the bound
\begin{equation}
\label{borne_1_var}
\begin{aligned}
\T_{\delta u^2}&\leq
\Big(\sum\nolimits_{k,\ell }\iint_{\Delta ^2_{[s,t]}}\d |z^\ell |\d |z^k|\Big)\varepsilon (s,t)\n{(\sigmab^*+\nu )((\sigmab^*+\nu ))\phi }{1}
\\
&\leq
C(|\sigma |_{W^{3,\infty}},|\nu |_{W^{2,\infty}})\left(|z|_{1-\mathrm{var};[s,t]}\right)^2\varepsilon (s,t)\n{\phi }{3}\,.
\end{aligned}
\end{equation}

Similarly, we have
\begin{equation}
\label{borne_1_var_2}
\begin{aligned}
\T_A
&\leq \Big(\sum\nolimits_k\int_s^t\omega _\mu (s,r)\d |z_r^k|\Big)\nn{(\sigmab^*+\hat\nu )\phi }{1}
\\
&\leq C(|\nu |_{W^{1,\infty}},|\sigma |_{W^{2,\infty}})|z|_{1-\mathrm{var};[s,t]}\omega _\mu (s,t)\,.
\end{aligned}
\end{equation}

The conclusion follows from \eqref{borne_1_var}, \eqref{borne_1_var_2}, and Remark \ref{rem:control}, we have:
\[
u^{2,\natural}\in \V_ {2,\rm{loc}}^{1+}(I;(W^{3,\infty})^*)\,,
\]
which proves the corollary.
\end{proof}
\begin{proof}[Proof of Proposition \ref{pro:energy}]
Testing against $\phi =1\in W^{3,\infty}$ in \eqref{approximate_pb3}, 
we have, using \eqref{uniform_ellipticity} and the inequality $|\sum b^i\partial _iu|\leq m/2\sum (b^i)^2 +1/(2m)\sum (\partial _iu)^2:$
\begin{equation}\label{rel:utile}
\begin{aligned}
\delta \left(|u|^2_{L^2}\right)_{st}+2m\int_s^t|\nabla u_r|^2_{L^2}\d r
\leq
\iint_{[s,t]\times\O}\left(m|\nabla u_r|^2+\big(\tfrac{1}{m}\tsum_{i\leq d} (b^i)^2 +2|c|\big)u^2\right)\d r\d x
\\
+\left(\omega _B(s,t)^\alpha +\omega _B(s,t)^{2\alpha }\right)|u_s|^2_{L^2}
+\langle u^{2,\natural}_{st},1\rangle\,.
\end{aligned}
\end{equation}
Note that by \eqref{interpolation_inequality},
\begin{multline*}
\iint_{[s,t]\times\O}\tfrac{1}{m}\big(\tsum(b^i)^2 +|c|\big)u^2\d r\d x
\leq
\left\|\tfrac{1}{m}\tsum(b^i)^2 +|c|\right\|_{r,q}\|u\|^2_{\frac{2r}{r-1},\frac{2q}{q-1}}
\\
\leq
(1+\tfrac{\beta ^2}{2m})\left(\|b\|_{2r,2q;[s,t]}^2+\|c\|_{r,q;[s,t]}\right)\|u\|_{\Bc_{s,t}}^2
\\
=C\left(\bb(s,t)^{1/r} +\cc(s,t)^{1/r}\right)\|u\|_{\Bc_{s,t}}^2
\,,
\end{multline*}
where we make use of the notation \eqref{controls_abc} and we recall that $\beta >0$ denotes the sharpest constant in \eqref{consequence_rq}.
Therefore, defining
$G_t:=|u_t|^2_{L^2}+\min(1,m)\int_0^t|\nabla u_r|^2_{L^2}\d r$
we have
\begin{equation}
\label{rel:utile}
\delta G_{st}
\leq
C\left(\bb(s,t)^{1/r}+\cc(s,t)^{1/r}
+\omega _B(s,t)^\alpha +\omega _B(s,t)^{2\alpha }\right)\|u\|_{\Bc_{s,t}}^2+\langle u^{2,\natural}_{st},1\rangle\,,
\end{equation}
for a constant $C>0$ depending on $m,r,q$ only.
Now, combining Lemma \ref{lem:drift_existence} and Proposition \ref{pro:apriori}, we 
can estimate the remainder as follows
\begin{equation}\label{apriori_estimate}
\nn{u^{2,\natural}_{st}}{-3}
\leq C\left(\omega _B(s,t)^{3\alpha }+\omega _B(s,t)^{\alpha }\right)\|u\|_{\Bc_{s,t}}^2\,,
\end{equation}
where the constant above depends on $\|b\|_{2r,2q},\|c\|_{r,q},$ but also
on $|\sigma |_{W^{3,\infty}},|\nu |_{W^{2,\infty}}.$
Hence, 
using \eqref{rel:utile}, \eqref{borne_controle_mu} and \eqref{interpolation_inequality}, we obtain that
\begin{equation}
 \delta G_{st}\leq 
\omega (s,t)^{1/\kappa }\Big(\sup_{r\in[s,t]}G_r\Big)\,,
\end{equation}
provided $\omega (s,t)\leq L$ is small enough,
where we let $\kappa :=\max(1/\alpha ,r),$ and
\[
\omega :=C\left(\bb ^{\kappa /r} +\cc ^{\kappa /r}+(\omega _Z)^{\kappa \alpha }+(\omega _Z)^{2\kappa \alpha }  +(\omega _Z)^{3\kappa\alpha  } \right)\,,
\]
for an appropriate constant $C=C(\kappa ,M,|\sigma |_{W^{2,\infty}},|\nu |_{W^{1,\infty}}),$

Applying Lemma \ref{lem:gronwall} with $\varphi :=0,$
this gives us the energy inequality \eqref{energy_inequality},
for a constant depending, through \eqref{nota:omega_B} and Proposition \ref{pro:apriori}, on the quantities 
\[
\omega _Z,m,M,T,\|b\|_{2r,2q},\|c\|_{r,q},|\sigma |_{W^{3,\infty}}\text{ and }|\nu |_{W^{2,\infty}}\,.
\qedhere
\]
\end{proof}

\begin{remark}
\label{rem:sto_par_2}
Under the assumptions of Remark \ref{rem:sto_par_1}, consider a smooth approximation $z^\epsilon $ of a path $z\in \V^{\alpha }_1$ for which we are given an enhancement $\Z\equiv(Z,\ZZ)$ in $\mathscr C^{\alpha}(I;\R).$
Assume that the latter is \emph{not geometric}, in the sense that the bracket given by
\[
\delta [\Z]_{st}:= (Z_{st})^2 - 2\ZZ_{st},\quad (s,t)\in \Delta _I\,,
\]
where by definition $Z:=\delta z,$ is not the zero function.
Moreover, for each $\epsilon >0,$ denote by $u_{\mathrm{can}}^{\epsilon}\in\Bc$ the unique solution of \eqref{approximate_pb} associated to the canonical lift $\Z_{\mathrm{can}}^\epsilon \equiv(Z^\epsilon , \frac12(Z^\epsilon )^2)$ of $z^\epsilon .$

Note that $\Z_{\mathrm{can}}^\epsilon $ converges to $\ZZ^{\mathrm{geo}}\equiv(Z,\frac12(Z)^2)$ in $\mathscr C^{\alpha}(I;\R)$ (see \eqref{RP_metric}), so that even if one has a candidate $u^{[\Z]}$ for the equation driven by the non-geometric rough path $\Z$, it is not expected that $u^{\epsilon }_{\mathrm{can}}$ converges to $u^{[\Z]},$ (regardless of the topology considered).
Instead, one has to consider the sequence $u^{\epsilon }:=u^{\varphi ^\epsilon },$ 
where $\varphi ^\epsilon $ denotes some smooth correction term converging to $[\Z],$
and by definition $u^\epsilon $ is supposedly a solution of the following modified version of \eqref{approximate_pb}:
\begin{equation}
\label{modified_eq}
\partial _tu^\epsilon -a\partial _{xx}u^\epsilon  = \sigma\partial _xu^\epsilon  \dot z^\epsilon  - \frac12\sigma ^2\partial _{xx}u^\epsilon \dot\varphi ^\epsilon,
\quad \quad u_0^\epsilon =u_0.
\end{equation}
Proceeding as in \eqref{expansion_existence}, there holds formally,
for each $(s,t)\in\Delta _I:$
\[
u_t^2-u_s^2-\int_s^t2au_r\partial _{xx}u_r\d r
\approx\enskip 
\sigma \partial _x(u_s^2)Z_{st} +\sigma ^2\partial _{xx}(u_s^2) \ZZ_{st} 
+\sigma ^2(\partial _{x}u_s)^2\delta \varphi _{st}.
\]
In this case, it is clear that the computations made in the proof Proposition \ref{pro:energy} fail, unless some smallness assumption on $\sigma $ in terms of $\dot \varphi $ and $a$ is made. Hence, our method to obtain the energy inequality for the rough equation \eqref{approximate_pb} ceases to work.
\end{remark}
\section{Tensorization}
\label{sec:tensorization}

The aim of this section is to introduce the set-up for the proof of uniqueness presented in Section  \ref{sec:uniqueness}. Recall that in Section \ref{sec:energy} we considered a smooth driving signal $Z$ and derived an energy estimate depending  only on the rough path norm of the associated canonical lift $\mathbf{Z}$. Nevertheless, the smoothness of $Z$ was only used in Corollary \ref{cor:u_natural} in order to verify that $u^2$ solves \eqref{approximate_pb3}. Accordingly, the result of Proposition \ref{pro:energy} remains valid in the case of a rough driving signal $Z$ provided one can justify the equation for $u^2$. This is the main challenge of the proof of uniqueness. Indeed, by linearity of \eqref{zakai}, uniqueness follows once we show that 
$\|u\|_{C(I;L^2)}\leq C|u_0|_{L^2}$ is satisfied by every weak solution in the sense of Definition \ref{def:weak_sol}. However, recall that due to Definition \ref{def:weak_sol}, the required regularity of test functions that guarantees smallness of the remainder is out of reach for general weak solutions. Consequently, it is not possible to simply test by the solution and to obtain the equation for $u^2$. Our approach relies on a tensorization procedure which is an analog of the doubling of variables method known from the classical PDE theory.

\subsection{Preliminary material and main result}
For $j =1,2$
consider $\B^j \equiv(B^j ,\BB^j ),$ an unbounded rough driver on the scale $(W^{k,2})_{k\in\N},$ 
a drift term
$\lambda ^j\in \V_ 1^1(I;W^{-1,2})$ and assume the existence of a weak solution $u^j\in C(I;L^2)$ of 
\begin{equation}\label{equation:u_gamma}
\d u^j =\d \lambda^j +\d \B^j u^j \,,
\end{equation}
in the sense of Definition \ref{def:weak_sol} on the scale $(W^{k,2})_{k\in\N}$.
For $R>0$ we define $B_R:=\{x\in\R^d:\sum_{i\leq d}|x_i|^2\leq R^2\}$ and let
\begin{equation}
\label{Omega}
\Omega :=\left\{(x,y)\in\OO:\frac{x-y}{2}\in B_1\right\}\,.                                                           \end{equation}
As the first step, we aim to show that the new unknown
\begin{equation}\label{nota:U}
\u(x,y):=(u^1\otimes u^2)(x,y)=u^1(x)u^2(y)\,,\quad \quad (x,y)\in\Omega \,,
\end{equation}
is itself a solution in the sense of Definition \ref{def:weak_sol} of a rough PDE on a suitable scale. This is the first step towards the proof of uniqueness and can be regarded as a linearization of the product operation $u(x)u(x)$. The second step, which we perform in Section \ref{sec:uniqueness}, then consists of the passage to the diagonal. Namely, we prove that the evolution of $\u(x,x)=u(x)u(x)$ is given by \eqref{ito_formula}.

For $k\in\N,$ define 
\begin{equation}\label{scale2_2}
\F_k:=\left\{\Phi\in W^{k,\infty}(\O),\,\spt\Phi \subset \Omega\right\}\,,\quad \nnn{\cdot }{k}:=|\cdot |_{W^{k,\infty}}\,,
\end{equation}
and additionally, let $\F_{-k}:=(\F_k)^*.$

Denote by $\X  \equiv(X ,\XX )$ the unbounded rough driver given for every $(s,t)\in\Delta$ by 
\begin{equation}\label{nota:Gamma_gene}
\left[
\begin{aligned}
&X_{st}:=B^1_{st}\otimes \id +\id\otimes B^2_{st},
\\
&\XX_{st}:=\BB^1_{st}\otimes \id+\id\otimes\, \BB^2_{st} +B^1_{st}\otimes B^2_{st}\,,
\end{aligned}
\right.
\end{equation}
(the proof that the properties \ref{RD1}-\ref{RD2} are fulfilled is an easy exercise left to the reader).
Furthermore, for every $\Phi \in\F_1$ and $(s,t)\in\Delta ,$
define the approximate drift as the distribution
\begin{equation}\label{nota:Pi}
\pi_{st} := u^1_s\otimes \delta \lambda^2_{st}+ \delta \lambda ^1_{st}\otimes u^2_s \,.
\end{equation}
\begin{remark}
 \label{rem:caracterization}
Let $k\in\N,$ and define 
\[N_k:=\card\{\gamma\in \N^d,\, |\gamma |:=\gamma _1 +\dots +\gamma _d\leq k\}\,.
\]
In the proof below, we will make use of the following well-known characterization of the spaces $W^{-k,2}\equiv(W^{k,2})^*$ (see e.g.\ \cite[Proposition 9.20]{brezis2010functional}).
For each $v$ in $W^{-k,2},$ there exist a (non-unique) $f$ in $(L^2)^{N_k}$ such that
\begin{equation}
\label{representation1}
\text{for every}\enskip \phi \in W^{k,2}\,,\quad \quad 
~_{W^{-k,2}}\big\langle v, \phi  \big\rangle_{W^{k,2}} = \sum_{|\gamma|\leq k}\big(f_\gamma,\D^\gamma \phi\big)_{L^2} 
\end{equation}
where $(\cdot ,\cdot )_{L^2}$ denotes the $L^2$ inner product, and $\D^\gamma \phi :=\partial _{\gamma _1}\cdots\partial _{\gamma _d}\phi .$
Moreover, there holds
\begin{equation}
\label{representation_norms}
|v|_{W^{-k,2}}\leq |f|_{L^2}\quad \text{and}\quad \inf_{f\in (L^2)^{N_k},\,\text{s.t.\ \eqref{representation1} holds}}\left(\sum_{|\gamma|\leq k}|f_\gamma|_{L^2}^2\right)^{1/2}\leq |v|_{W^{-k,2}}\,.
\end{equation}
\end{remark}
First, we need the following.
\begin{lemma}
\label{lem:small_pi}
The distribution-valued 2-index map $\pi$ defined in \eqref{nota:Pi} has finite variation with respect to $\F_{-1},$ and we have the bound
\begin{equation}
\label{pi_finite_var}
\nnn{\pi_{st}}{-1} \leq C\left(\n{\delta \lambda ^1_{st}}{-1} \n{u^2_s}{0} +\n{u^1_s}{0}\n{\delta \lambda ^2_{st}}{-1}  \right)\,,\quad \forall(s,t)\in\Delta \,,
\end{equation}
for some universal constant $C>0$.

Furthermore,
assuming that $\lambda ^1,\lambda ^2\in \AC(I;W^{-1,2})$, then there is a unique $\Xi \in \V^{1}_1(I;\F_{-1})$ such that for every $t\in I$ and every sequence of partitions $|\pp_n|\to0$ of $[0,t]$ we have
\begin{equation}
\label{riemann_sums_Xi}
\lim_{n\to\infty}\sum_{(\pp_n)}\pi_{t_i t_{i+1}}\to \Xi _{t}\quad \text{in}\quad \F_{-1}\,.
\end{equation}
\end{lemma}
\begin{notaAdd}
For $a\in\O$, we will henceforth denote by $\tt _a$ the translation operator, namely for $\psi  \in L^2(\O)$:
\begin{equation}\label{nota:translation}
\tt _a\psi(x) :=\psi (x-a)\,,\quad x\in\O\,.
\end{equation}
We recall that $\tt_a$ is an isometry in every $L^p$ space, $p\in [1,\infty].$
In addition, we have the following property: 
for every $p$ in $[1,\infty),$ and every $f\in L^p,$ 
\begin{equation}
\label{prop:translation_1}
\|\tt_{a}f-f\|_{L^p}\to0 \quad \text{as}\quad a\to0\,;
\end{equation}
(it suffices to check this for $f$ in $C^\infty$ and then to argue by density).
\end{notaAdd}
\begin{proof}[Proof of Lemma \ref{lem:small_pi}]
Fix $(s,t)$ in $\Delta .$
Due to Remark \ref{rem:caracterization},
for $j=1,2$ there exists $(f^j_\gamma )_{|\gamma |\leq 1}$ in $(L^2)^{N_1}$ such that for every $\phi \in W^{1,2}(\O):$
\begin{equation}
\label{representation_pi}
\phantom{\rangle}_{W^{-1,2}(\O)}\big\langle\delta \lambda _{st}^j, \phi \big\rangle_{W^{1,2}(\O)}
=\sum_{|\gamma |\leq 1}\ps{\Lambda ^j_\gamma }{\D^\gamma \phi }_{L^2(\O)}\,.
\end{equation}
Then, for $\Phi\in\F_1$ we have by definition
\begin{equation}
\label{computations_pi}
\begin{aligned}
\langle\pi_{st},\Phi \rangle &=
\int_{\O}u^1(x)~_{W^{-1,2}}\big\langle\delta \lambda ^2, \Phi (x,\cdot )\big\rangle_{\!W^{1,2}}
\d x
+\int_{\O}u^2(y)
~_{W^{-1,2}}
\big\langle\delta \lambda ^1, \Phi (\cdot ,y) \big\rangle
_{W^{1,2}}
\d y
\\
&=\sum_{|\gamma |\leq 1}\int_\O u^1(x)\ps{\Lambda _\gamma ^2}{\D_y^\gamma  \Phi(x,\cdot )}_{L^2_y(\O)}\d x
+\sum_{|\gamma |\leq 1}\int_\O\ps{\Lambda _\gamma ^1}{\D_x^\gamma  \Phi(\cdot ,y)}_{L^2_x(\O)}u^2(y)\d y 
\\
&\leq C\iint_{\Omega}\Big(|u^1(x)||\Lambda  ^2(y)|+|\Lambda ^1(x)||u^2(y)|\Big)(|\Phi |+|\nabla _{x,y}\Phi |)(x,y)\d x\d y
\\
&=C\iint_{\O\times B_1}\big(|u^1(x_++x_-)||\Lambda ^2(x_+-x_-)|+|\Lambda ^1(x_++x_-)||u^2(x_+-x_-)|\big)
\\
&\quad \quad \quad \quad \quad \quad \quad \quad \quad \quad \quad 
\times(|\Phi |+|\nabla _{x,y}\Phi|)(x_++x_-,x_+-x_-)\d x_-\d x_+
\\
&\leq C\int_{B_1}\Big(|\tt_{-x_-}u^1|_{L^2_{x_+}}|\tt_{x_-}\Lambda ^2|_{L^2_{x_+}} +|\tt_{-x_-}\Lambda ^1|_{L^2_{x_+}}|\tt_{x_-}u^2|_{L^2_{x_+}}\Big)\d x_-\nnn{\Phi }{1}
\\
&=C|B_1|\big(|u^1|_{L^2}|\Lambda ^2|_{L^2}+|\Lambda ^1|_{L^2}|u^2|_{L^2}\big)\nnn{\Phi }{1}
\,,
\end{aligned}
\end{equation}
where in the third line we have made the change of variables $(x_+,x_-)=(\frac{x+y}{2},\frac{x-y}{2}).$
Now, the constant above does not depend on the choice of $\Lambda ^1,\Lambda ^2$ in \eqref{representation_pi}, hence we can take the infimum, which, thanks to \eqref{representation_norms}, yields the first part of the Lemma.

We need to justify the existence and uniqueness of $\Xi $ such that \eqref{riemann_sums_Xi} holds.
Recall that since 
$\lambda ^j\in \AC(I;W^{-1,2})$ and since $W^{-1,2}$ is reflexive,
then $\dot\lambda _r\equiv\lim_{\epsilon \to0} (\lambda ^j_{r+\epsilon }-\lambda ^j_r)/\epsilon \in W^{-1,2}$ exists a.e.\ in $I,$ and we have
\[
\delta \lambda ^j_{st}=\int_s^t\dot\lambda ^j_r\d r
\]
(Bochner sense).
On the other hand, from similar computations as in \eqref{computations_pi} we have
$\int_I\nnn{u_r^1\otimes \dot\lambda_r ^2 + \dot\lambda_r ^1\otimes u^2_r}{-1}\d r<\infty.$ 
Observing that for every $r\in I$ the linear map $f_r:=\big(h\in W^{-1,2}\mapsto u_r^1\otimes h+h\otimes u_r^2 \in \F_{-1}\big)$ is continuous with norm not exceeding $|u^1_r|_{L^2}+|u^2_r|_{L^2},$ we can then apply \eqref{bochner_integral}, so that for every $\pp_n\in \PP([0,t]),$ $|\pp_n|\to0:$
\[
\sum_{(\pp_n)} \pi_{t_i t_{i+1}}\to \Xi_{t}\equiv \int_0^t (-u^1_r\otimes \dot\lambda _r ^2 - \dot\lambda ^1_r\otimes u^2_r )\d r\quad \text{strongly in }\F_{-1}\,.
\qedhere
\]
\end{proof}
The main result of this section is the following.
\begin{proposition}
\label{pro:U_natural}
\begin{enumerate}[label=(\alph*)]
\item \label{a}
There exists $\Pi\in \V_ 1^1(I;\F_{-1})$ such that
for every $(s,t)\in\Delta,$ 
 \begin{equation}
 \label{Pi_vs_pi}
\nnn{\delta \Pi_{st} -\pi_{st}}{-2} \leq \omega (s,t)^a 
\end{equation}
for some control $\omega $ and some $a >1.$
If in addition $\lambda ^1,\lambda ^2$ belong to $\AC(I;W^{-1,2}),$ then $\Pi$ is unique and we have $\Pi=\Xi ,$ where $\Xi $ is as  in \eqref{riemann_sums_Xi}.

\item \label{b}
the tensor product $\u\equiv u^1\otimes u^2$ is a weak solution of the rough PDE
\begin{equation}\label{eq:U}
\d \u=\d \Pi +\d \X \u\,,
\end{equation}
on the scale $(\F_k)_{k\in\N},$ in the sense of Definition \ref{def:weak_sol}.
\end{enumerate}
\end{proposition}
\subsection{Proof of Proposition \ref{pro:U_natural}}
\textit{\indent Proof of \ref{a}.}
The first claim  follows by the same arguments as in Lemma \ref{lem:small_pi}, together with an application of the Sewing Lemma (see Appendix \ref{app:A1}).
More precisely, there holds for $(s,\theta ,t)\in\Delta :$
\[
\delta (\pi)_{s\theta t}=-\delta u^1_{s\theta }\otimes \delta \lambda^2 _{\theta t} -\delta \lambda _{\theta t}^1\otimes \delta u^2_{s\theta }\,.
\]
Now, for $j=1,2$ let $\Lambda ^j$ in $(L^2)^{N_1}$ such that \eqref{representation_pi} holds (with $\theta ,t$ instead of $s,t$), and similarly let $(f^j_\beta )_{|\beta |\leq 1}\in (L^2)^{N_1}$ such that for every $\phi $ in $W^{1,2}$
\begin{equation}
\label{representation_delta}
\langle\delta u^j_{s\theta } ,\phi \rangle =\sum_{|\beta  |\leq 1}\ps{f^j_\beta  }{\D^\beta \phi }_{L^2}\,.
\end{equation}
Let $\Phi\in\F_2$. Then we see that
\[
\begin{aligned}
\langle&\delta \pi_{s\theta t},\Phi \rangle  =
\\
&-\sum_{|\gamma |,|\beta |\leq 1}\PS{f^1_\beta  }{\PS{\Lambda ^2_{\gamma }}{\D ^{\beta  }_x\D ^\gamma _y \Phi  }_{L^2_y}}_{L^2_x}
-\sum_{|\gamma |,|\beta |\leq 1}\PS{\Lambda ^1_\gamma }{\PS{f^2_{\beta }}{\D_x^{\gamma }\D_y^{\beta } \Phi }_{L^2_y}}_{L^2_x}
\\
&\leq
C\iint_{\Omega}\big(|f^1(x)||\Lambda ^2(y)|+|\Lambda ^1(x)||f^2(y)| \big)
\big(|\Phi |+|\nabla _{x,y}\Phi |+|\nabla _{x,y}^2\Phi |\big)(x,y)\d x\d y\,.
\end{aligned}
\]
Proceeding as as before with the change of variables $(x_+,x_-)=(\frac{x+y}{2},\frac{x-y}{2}),$ taking the infimum over $\Lambda ^1,\Lambda ^2,f^1, f^2$ such that \eqref{representation_pi}, \eqref{representation_delta} hold, and then using \eqref{representation_norms},
we obtain that
\[
\nnn{\delta \pi_{s\theta t}}{-2}
\leq C
\left(\n{\delta u^1_{s\theta }}{-1}\n{\delta \lambda ^2_{st}}{-1}+\n{\delta \lambda ^1_{\theta t}}{-2} \n{\delta u^2_{s\theta }}{-1} \right)\,,\quad \forall(s,\theta ,t)\in\Delta ^2\,.
\]
for some universal constant $C>0.$
Hence, for every $(s,\theta ,t)$ in $\Delta ^2:$
\begin{equation}
\label{delta_pi}
\nnn{\delta \pi_{s\theta t}}{-2}
\leq C\Big(\omega _{\lambda ^1}(s,t)\omega _{\delta u^2}(s,t)^{\alpha }+\omega _{\delta u^1}(s,t)^\alpha \omega_{\lambda ^2}(s,t)\Big)
\end{equation}
where for $j=1,2$ and $(s,t)\in\Delta,$ we let
$\omega _{\delta u^j}(s,t):=|\delta u^j|_{1/\alpha \mathrm{-var},W{-1,2};[s,t]}^{1/\alpha }$
(see Remark \ref{rem:control}).
Consequently, the r.h.s.\ of \eqref{delta_pi} fulfills the hypotheses of the Sewing Lemma, i.e.\ $\delta \pi\in\Zc^{1+}_3(I;\F_{-2}).$
Hence by Corollary \ref{cor_def}, there is a unique $\Pi^{\dagger}$ in $\V_ 1^{1}(I;\F_{-2})$
such that $(\pi-\delta\Pi^{\dagger})\in \V_ 2^{1+}(I;\mathcal{F}_{-2}).$
It is given by the rough integral
\begin{equation}
\label{definition_Pi}
\Pi^{\dagger}_{t}=
\I_{0t}(\pi )\equiv 
(\F_{-2})\text{ --}\!\!\!\!\lim_{\substack{|\pp|\to0\\ \pp\in\PP([0,t])}}\sum_{(\pp)}\pi_{t_i t_{i+1}}\,.
\end{equation}

We need to justify that $\Pi^{\dagger}$ can be extended in a unique way to an element $\Pi$ in $\V^1_1(I;\F_{-1}),$ which is not trivial since \emph{$\F_2$ is not dense in $\F_1.$}
However, letting $|\pp_n|\to0$ and $\I_n\pi$ be the partial sum associated to $\pp_n$ in the r.h.s.\ of \eqref{definition_Pi}, we have that $\limsup_{n} \nnn{\I_n\pi}{-1}\leq\omega _\pi(s,t)<\infty,$ where $\omega _\pi$ is any control such that $\omega _{\pi}\geq \nnn{\pi}{-1}.$
Hence by the Hahn-Banach Theorem, 
there exists such an extension $\Pi.$
Finally, by Lemma \ref{lem:small_pi}, we have $\I_n\pi\to\Xi $ in $\F_{-1},$ yielding that $\Pi=\Xi .$
This proves part \ref{a}.\hfill\qed

\bigskip
\textit{\indent Proof of \ref{b}.}
Define $\Pi:= \I_{0\cdot }(\pi)$ as above.
We have to show that the distribution-valued $2$-index map  $\u^\natural$ defined for each $(s,t)\in\Delta$ as
\begin{equation}\label{tensorized_eq}
\u^\natural_{st}:=\delta \u_{st}-\delta \Pi_{st}-X_{st}\u_s - \XX_{st}\u_s \,,
\end{equation}
belongs to $\V_ {2,\rm{loc}}^{1+}(I;\F_{-3}).$

A straightforward, but very useful observation is the following.
\begin{claim}
For $(s,t)\in\Delta$ and  $j =1,2$, we define the corresponding first order remainder
\begin{equation}\label{nota:u_sharp}
u^{j,\sharp}_{st}:=\delta u^j _{st}-B^j _{st}u^j _s\,.
\end{equation}
Then we have the identity
\begin{equation}\label{id:U_natural}
\begin{aligned}
\u^{\natural}_{st}=u^{1,\natural}_{st}\otimes u_s^2 +u^1_s\otimes u^{2,\natural}_{st}
+ \pi_{st}-\delta \Pi_{st}
+u^{1,\sharp}_{st}\otimes\delta u^2_{st}
+B^1_{st}u_s^1\otimes u_{st}^{2,\sharp}\,.
\end{aligned}
\end{equation}
\end{claim}
\indent\textit{Proof of Claim.}
First observe that adding and subtracting, we have
\[
\delta \u_{st}
=\delta u^1_{st}\otimes u^2_s +u^1_s\otimes\delta u^2_{st} +\delta u^1_{st}\otimes\delta u^2_{st}\,,
\]
which, omitting time indexes, is equal to :
\begin{multline*}
(\delta u^1-B^1u^1-\BB^1 u^1)\otimes u^2 + u^1\otimes(\delta u^2-B^2u^2-\BB^2 u^2)
\\
+X\u +\XX \u - B^1u^1\otimes B^2u^2 +\delta u^1\otimes\delta u^2
\\
\equiv(\delta u^1-B^1u^1-\BB^1 u^1)\otimes u^2 + u^1\otimes(\delta u^2-B^2u^2-\BB^2 u^2)
+X\u+\XX \u
\\
+(\delta u^1-B^1u^1)\otimes\delta u^2+ B^1u^1\otimes (\delta u^2-B^2u^2)
\,.
\end{multline*}
Similarly, adding and subtracting the drift term and using \eqref{tensorized_eq}, we obtain that:
\[
\begin{aligned}
\u^\natural_{st}&=(\delta u^1_{st}-B^1_{st}u^1_s-\BB^1_{st}u^1_s - \lambda^1_{st})\otimes u^2_s
+u^1_s\otimes(\delta u^2_{st}-B^2_{st}u^2_s-\BB_{st}^2u^2_s - \lambda^2_{st})
\\
&\quad \quad 
+\pi_{st}-\delta \Pi_{st}
+ (\delta u^1_{st}-B^1_{st}u^1_s)\otimes\delta u^2_{st}
+ B^1_{st}u^1_s\otimes(\delta u^2_{st}-B^2_{st}u^2_s)
\,.
\end{aligned}
\]
hence the claim is proved.\hfill \qed

\bigskip
\textit{\indent End of the Proof of Proposition \ref{pro:U_natural}.}
Take any $\Phi $ in $\F_3.$
From the identity \eqref{id:U_natural}, we can decompose $\langle\u^\natural,\Phi  \rangle$ into
\[\begin{aligned}
\langle\u^{\natural}_{st},\Phi \rangle
&=\left\langle u^{1,\natural}_{st}\otimes u_s^2 +u^1_s\otimes u^{2,\natural}_{st},\Phi \right\rangle
+\langle\pi_{st}-\delta \Pi_{st},\Phi  \rangle
\\
&\hspace{12em}+\left\langle u^{1,\sharp}_{st}\otimes\delta u^2_{st},\Phi \right\rangle
+\left\langle B^1_{st}u_s^1\otimes u_{st}^{2,\sharp},\Phi \right\rangle
\\
&=:\T_{\natural}+\T_\lambda +\T_{\sharp}^1 +\T_{\sharp}^2\,.
\end{aligned}
\]
In the above formula, it is immediately seen, according to Remark \ref{rem:gubinelli}, that each term above has the needed size in time, namely $\langle\u^\natural,\Phi\rangle$ belongs to the space $\V^{1+}_{2,\rm{loc}}(I;\R)$. That being said, it is necessary to evaluate $\u^\natural$ as a path with values in $\F_{-3},$ and not in $\mathscr D'(\O)$ only.
For that purpose, we use the characterization of Sobolev Spaces of negative order given by Remark \ref{rem:caracterization}.
Fix $(s,t)$ in $\Delta $ and 
for $j=1,2$ let $g^j,h^j\in (L^2)^{N_2}$ and $h^j\in (L^2)^{N_3}$ be such that 
\begin{equation}
\label{representation_sharp}
\text{for every}\enskip \phi \in W^{2,2}\,,\quad 
\langle u^{j,\sharp}_{st}, \phi  \rangle = \sum\nolimits_{|\gamma |\leq 2}\big(g_\gamma ^j,\D^\gamma \phi\big)_{L^2}
\end{equation}
\begin{equation}
\label{representation_natural}
\text{for every}\enskip \phi \in W^{3,2}\,,\quad 
\langle u^{j,\natural}_{st}, \phi  \rangle = \sum\nolimits_{|\beta |\leq 3}\big(h_\beta ^j,\D^\beta \phi\big)_{L^2}\,,
\end{equation}
and let $f^j$ be as in \eqref{representation_delta}.

For the first term, we have by definition:
\[
\begin{aligned}
\T_\natural=&
\big\langle u^2,\langle u^{1,\natural}, \Phi\rangle_{x}\big\rangle_{y}
+\big\langle u^1,\langle u^{2,\natural}, \Phi\rangle_{y}\big\rangle_{x}
\\
=& \sum\nolimits_{|\beta |\leq 3}\PS{ u^2}{\ps{h^1_\beta }{\D^\beta_x  \Phi}_{L^2_x} }_{L^2_y}
+\sum\nolimits_{|\beta |\leq 3}\PS{ u^1}{\ps{h^2_\beta }{\D^\beta _y \Phi}_{L^2_y}}_{L^2_x}\,.
\end{aligned}
\]
Changing variables
as before, there comes
\[
\begin{aligned}
\T_\natural
&\leq C\iint_{B_1\times\O}(|u^2(x_+-x_-)||h^1(x_++x_-)| +|u^1(x_++x_-)||h^2(x_+-x_-)|)
\\
&\quad \quad \quad \quad \quad 
\times(|\Phi | + |\nabla \Phi| + |\nabla ^2\Phi | +|\nabla ^3\Phi |)(x_++x_-,x_+-x_-)\d x_+\d x_-
\\
&\leq C\left(|u ^2_s|_{L^2}|h^1|_{L^2} +|u^1_s|_{L^2}|h^2|_{L^2}\right)\nnn{\Phi }{3}\,,
\end{aligned}
\]
where again we have used Fubini's theorem, together with the fact that the translations $\tt_{x_-},\tt_{-x_-}$ are isometries in $L^2.$ Hence, taking the infimum over the choice of $h^1,h^2$ in \eqref{representation_natural}, it holds true that for every $(s,t)\in \Delta ,$
\begin{equation}
\label{T1}
\T_\natural \leq C\left(\n{u ^2_s}{0}\n{u^{2,\natural}_{st}}{-3} +\n{u^1_s}{0}\n{u^{2,\natural}_{st}}{-3}\right)\nnn{\Phi }{3}\,,
\end{equation}
for some constant $C>0,$ independent of $(s,t)$ in $\Delta $ and $\Phi $ in $\F_3.$

For the third term, we have
\[
\begin{aligned}
\T_{\sharp}^1=&
\Big\langle u^{1,\sharp}, \langle\delta u^2,\Phi \rangle_y\Big\rangle_x
\\
=&
\sum_{|\gamma |\leq 2,|\beta |\leq 1}\PS{g^1_\gamma }{\ps{f^2_\beta }{\D_x^\gamma \D_y^\beta \Phi}_{L^2_y}}_{L^2_x}
\leq &
C|g^1|_{L^2}|f^2|_{L^2}\nnn{\Phi }{3}\,.
\end{aligned}
\]
Hence, taking the infimum over $g^1,f^2$ gives
\begin{equation}\label{T2}
\begin{aligned}
\langle\T_{\sharp}^1, \Phi \rangle
&\leq C\n{\delta u^2_{st}}{-1}\n{u_{st}^{1,\sharp}}{-2}\nnn{\Phi }{3}\,,
\end{aligned}
\end{equation}
for a constant depending neither on $(s,t)\in\Delta ,$ neither on $\Phi $ in $\F_3.$

Proceeding similarly for the fourth term, there holds:
\[
\T_{\sharp}^2=
\Big\langle B^1u^1,\langle u^{2,\sharp},\Phi \rangle_y\Big\rangle_x
=\sum\nolimits_{|\gamma |\leq 2}\PS{u^1}{\ps{g^2_\gamma }{B^{1,*}_x\D_y^\gamma \Phi }_{L^2_y}}_{L^2_x}
\]
Hence, we have
\begin{equation}
\label{T3}
\T_{\sharp}^2\leq C \omega _{B^1}(s,t)^\alpha \n{u^1_s}{0}\n{u^{2,\sharp}_{st}}{-2}\nnn{\Phi }{3}\,.
\end{equation}
for some universal constant $C>0.$

Now,
note that the drift term has been already estimated in Lemma \ref{lem:small_pi}, namely, we have 
\begin{equation}
\label{T4}
\begin{aligned}
\T_\lambda =&
\langle (\Lambda \delta \pi)_{st},\Phi  \rangle
\\
\leq & C\left(\omega _{\lambda ^1}(s,t)\omega _{\delta u^2}(s,t)^{\alpha }+\omega _{\delta u^1}(s,t)^\alpha \omega_{\lambda ^2}(s,t)\right)\nnn{\Phi }{2}\,.
\end{aligned}
\end{equation}

The conclusion follows by \eqref{T1}-\eqref{T2}-\eqref{T3}-\eqref{T4}.
Indeed, for $j=1,2,$ denote as before by $\omega _{\delta u^j}(s,t):=|\delta u^j|_{1/\alpha \mathrm{-var},W^{-1,2};[s,t]}^{1/\alpha },$ and furthermore define
$\omega _{j,\sharp}(s,t):=|u^{j,\sharp}|^{1/(2\alpha )}_{1/(2\alpha )\mathrm{-var},W^{-2,2};[s,t]},$
and
$\omega _{j,\natural}(s,t):=|u^{j,\natural}|_{1/(3\alpha )\mathrm{-var},W^{-3,2};[s,t]}$
(see Remark \ref{rem:control}).
Then, we see that:
\begin{multline*}
\nnn{\u^\natural_{st}}{-3}\leq C\left(\alpha ,\|u^1\|_{\infty,2},\|u^2\|_{\infty,2}\right) 
\Big(
(\omega _{1,\natural})^{3\alpha }+(\omega _{2,\natural})^{3\alpha }+ (\omega_{\delta u^2})^\alpha (\omega _{1,\sharp})^{2\alpha }  + (\omega _{B^1})^\alpha(\omega _{2,\sharp})^{2\alpha }
\\
+\omega _{\lambda ^1}(\omega _{\delta u^2})^{\alpha }+(\omega _{\delta u^1})^\alpha \omega_{\lambda ^2}
\Big),
\end{multline*}
where all the controls are evaluated at $(s,t)$.
Since each term on the above right hand side is of homogeneity at least $3\alpha ,$ we see that
\[
\u^\natural\in \V_ {2,\rm{loc}}^{3\alpha}(I;\F_{-3})\subset \V^{1+}_{2,\rm{loc}}(I;\F_{-3})\,,
\]
which completes the proof of Proposition \ref{pro:U_natural}.
\hfill\qedsymbol
\begin{remark}
\label{rem:Pi}
Assume that for $j=1,2$ and $t\in I:$
\[
\lambda ^j_{t}:= \int_0^tA^ju^j\d r
\]
where we are given $u^j$ in $\Bc$ and
\[
A^j(t,x):=\partial _\alpha  (a^{j,\alpha \beta }(t,x)\partial _\beta  \cdot ) + b^{j,\alpha } (t,x)\partial _\alpha  +c^j(t,x)\,,
\]
with coefficents $a^j,b^j,c^j$ such that Assumption \ref{ass:a} and Assumption \ref{ass:b_c} hold. 
Using the estimate \eqref{borne_controle_lambda}, we see that $\lambda ^j$ belongs to $\AC(I;\F_{-1}).$ Moreover, making use of the notations of Proposition \ref{pro:U_natural}, we have for every $t\in I:$
\[
\Pi_{t}:=\int_0^t( u_r^1\otimes A^2_ru^2_r+A_r^1u_r^1\otimes u^2_r)\d r\,.
\]
in the Bochner sense, in $\F_{-1}.$
\end{remark}
\section{Uniqueness}
\label{sec:uniqueness}

After the preliminary step of tensorization presented in Section \ref{sec:tensorization} we proceed with the proof of uniqueness. The ultimate goal is to test the tensor equation for $u(x)u(y)$ by a Dirac mass $\delta_{x=y}$ which finally gives the desired equation for $u^2$. To achieve this, we first consider a smooth approximation to the identity $\psi_\epsilon$ which is a legal test function for \eqref{eq:U}. The core of the proof  then consists in the  justification of the passage to the limit as $\epsilon\to0$. More precisely, it is necessary to bound all the terms in the equation uniformly in $\epsilon\in (0,1)$. Similarly to the a priori estimates in Section \ref{sec:energy}, the main challenge is to bound the remainder term. Our approach relies on a suitable blow-up transformation together with uniform bounds for all the other terms in the equation which permits to employ again Proposition \ref{pro:apriori} and yields an estimate uniform in $\epsilon$.

Consider $u\in\Bc$, a weak solution to \eqref{zakai} in the sense of Definition \ref{def:weak_sol} and define
\begin{equation}
\label{nota:U_gamma}
\u(x,y):=u(x)u(y)\,,\quad \quad 
\text{for every}\enskip (x,y)\enskip \text{in}\enskip \OO\,.
\end{equation}
Denote by $\S \equiv(S,\SS)$ the symmetric driver, given for every $(s,t)\in\Delta ,$ by 
\begin{equation}\label{nota:Gamma}
\left[\begin{aligned}
&S_{st}:=B_{st}\otimes \id +\id\otimes B_{st},
\\
&\SS_{st}:=\BB_{st}\otimes \id+\id\otimes \,\BB_{st} +B_{st}\otimes B_{st}\,,
\end{aligned}\right.
\end{equation}
and also by 
\[
\Pi_{t}:= \int_0^t\Big(A_ru_r\otimes u_r +u_r\otimes A_ru_r\Big)\d r\,.
\]
Fix $\epsilon >0.$ Then replacing $\Omega $ by 
\begin{equation}
\label{Omega_epsilon}
\Omega _\epsilon :=\left\{(x,y)\in\R^d\times\R^d:\frac{|x-y|}{2}\leq \epsilon \right\}
\end{equation} 
in Section \ref{sec:tensorization}, then Proposition \ref{pro:U_natural} and Remark \ref{rem:Pi} yield that
\begin{equation}
\label{equation:u_tensorized}
\d \u=\d\Pi + \d\S \u\,,
\end{equation}
holds with respect to the scale $(\F_k(\Omega _\epsilon ))_{k\in\N},$ in the sense of Definition \ref{def:weak_sol}.

We now define the blow-up transformation
$T_\epsilon :\F_0(\Omega )\to \F_0(\Omega_\epsilon )$ as follows:
given $\Phi \in \F_0(\Omega ),$ we let
\begin{equation}
\label{nota:T_epsilon}
T_\epsilon \Phi(x,y):=(2\epsilon )^{-d}\Phi\left (\frac{x+y}{2} +\frac{x-y}{2\epsilon },\frac{x+y}{2} -\frac{x-y}{2\epsilon }\right)\,,
\quad 
\text{for any}\quad (x,y)\in\Omega_\epsilon \,.
\end{equation}
This operation is invertible and we have for $(x,y)\in \Omega :$
\begin{equation}
\label{T_epsilon_inverse}
T_\epsilon ^{-1}\Phi(x,y)=(2\epsilon)^d \Phi\left(\frac{x+y}{2}+\epsilon\frac{x-y}{2},\frac{x+y}{2}-\epsilon \frac{x-y}{2}\right) \,.
\end{equation}

Given $k\in\{0,1,2,3\}$ and $v$ in $\F_{-k}(\Omega _\epsilon ),$
we can define a distribution $T_\epsilon^*v\in\F_{-k}(\Omega )$ by duality, and similary $T_{\epsilon }^{-1,*}v$ makes sense as an element of $ \in \F_{-k}(\Omega _\epsilon ).$

For any $\Psi \in\F_3(\Omega ),$
we can test \eqref{equation:u_tensorized} against 
\[
\Phi :=T_\epsilon \Psi \in \F_3(\Omega_\epsilon )\,.
\]
We deduce that for all $\Psi\in\F_{3}(\Omega)$ and $(s,t)\in\Delta :$
\[
\langle  T_\epsilon ^*\delta \u_{st} ,\Psi\rangle
=
\langle T_\epsilon ^*\delta \Pi _{st},\Psi \rangle+ \langle  T_\epsilon ^*(S_{st}+\SS_{st}) \u^\epsilon _s ,\Psi \rangle+\langle   T_\epsilon ^*\u^{\natural}_{st} ,\Psi \rangle\,,
\]
whence letting
$\u^\epsilon :=T_\epsilon ^*\u,$
$\S^\epsilon :=T_\epsilon ^*\S T_\epsilon ^{-1,*},$
$\Pi^\epsilon := T_\epsilon^*\Pi,$
and
$\u^{\natural,\epsilon } :=T_\epsilon ^*\u^{\natural},$
we see that $\u^\epsilon $ is a weak solution of
\begin{equation}
\label{eq:U_epsilon}
\d \u^\epsilon
=
\d \Pi^\epsilon + \d \S^\epsilon \u^\epsilon \,,
\end{equation}
with respect to the scale $(\F_k(\Omega ))_{k\in\N},$ in the sense of Definition \ref{def:weak_sol}.

As the next step, we  establish uniform bounds for the renormalized driver $\mathbf{S}^\epsilon$ as well as for the drift $\Pi^\epsilon$, which in turn implies a uniform bound for the remainder $u^{\natural,\epsilon}$. The proof of uniqueness is then concluded in Subsection \ref{subsec:uniqueness}.

\subsection{Renormalizability of symmetric drivers}

Let us begin with the uniform bound for the  driver $\mathbf{S}^\epsilon$. Following \cite{deya2016priori}, the following definition will be useful.

\begin{definition}[Renormalizable drivers]
We say that a family $\mathbf{S}^\epsilon\equiv (S^\epsilon ,\SS^\epsilon )$, $\epsilon\in(0,1)$, of $1/\alpha$-unbounded rough drivers is \emph{renormalizable}, with respect to a scale $(\G_k),$ if there exists a control $\omega _{S}$ such that the bounds \eqref{bounds:rough_drivers} hold uniformly with respect to $\epsilon \in(0,1),$ namely for all $(s,t)\in\Delta ,$
\begin{align}\label{bounds:renormalizable_driver}
&|S^\epsilon _{st}|_{\L(\G_{-k},\G_{-k-1})}\leq \omega _{S}(s,t)^\alpha\,,\quad \text{for}\enskip k=0,1,2\enskip \text{and} 
\\
\label{bounds:renormalizable_driver2}
&|\SS^\epsilon _{st}|_{\L(\G_{-k},\G_{-k-2})}\leq \omega _{S}(s,t)^{2\alpha}\,,\quad \text{for}\enskip k=0,1\,.
\end{align}
\end{definition}

For every $k$ we henceforth omit to mention the domain $\Omega $ and write $\F_k$ for $\F_k(\Omega )$ (recall \eqref{Omega}). We have the following.

\begin{proposition}\label{pro:renormalizable}
Consider a driver $\S$ as in \eqref{nota:Gamma} and define for each $\epsilon \in(0,1):$
\[
\S^\epsilon \equiv(S^\epsilon ,\SS^\epsilon ):= (T_\epsilon ^{*}S_{st}T_\epsilon ^{-1,*},T_\epsilon ^{*}\SS_{st}T_\epsilon ^{-1,*})\,.
\]
Then, the family $(\S^\epsilon )_{\epsilon\in(0,1)}$, is renormalizable with respect to the scale $(\F_k).$

Moreover, the bounds \eqref{bounds:renormalizable_driver}-\eqref{bounds:renormalizable_driver2} hold with a control of the form
\begin{equation}\label{omega_uniform}
\omega_{S} (s,t):=C\left(|\sigma |_{W^{3,\infty}},|\nu |_{W^{2,\infty}}\right)\omega _Z(s,t)\,,
\end{equation}
where the constant above only depends on the indicated quantities.
\end{proposition}
We now need to introduce some useful notations.
\begin{notaAdd}
Recall \eqref{nota:translation}.
Given $a\in\O$ and $\epsilon >0,$
it is useful to introduce the ``local mean'' as the linear map:
\begin{equation}
\label{nota:m_epsilon}
\boldsymbol m^a_{\epsilon }:=\frac12\left(\tt_{-\epsilon a}+\tt_{\epsilon a}\right)\,.
\end{equation}
\end{notaAdd}
\begin{notaAdd}
For $a\in \O$, we define the finite-difference operator
\begin{equation}\label{nota:finite_difference}
\boldsymbol\Delta _{\epsilon}^a:=\frac{\tt_{-\epsilon a}-\tt_{\epsilon a}}{2\epsilon }\,.
\end{equation}
For the reader's convenience, the main properties of $\df$ are provided in Appendix \ref{app:finite_difference}.
\end{notaAdd}
\begin{notaAdd}
Similarly to Section \ref{sec:tensorization}, it will be convenient to use the new coordinates $\chi :\Omega\to\O\times B_1$ defined by
\begin{equation}
\label{new_coordinates}
(x_+,x_-)=\chi (x,y):=\left(\tfrac{x+ y}{2},\tfrac{x-y}{2}\right)\,,\quad \text{for}\enskip (x,y)\in\Omega \,.
\end{equation}
Note that $|\det\D\chi|=2^{-d}$ and that $\sqrt2 \chi $ is a rotation.
\end{notaAdd}
\begin{notaAdd}
Given $\Phi:\OO \to\R$, we will occasionally denote by $\check{\Phi }:=\Phi \circ\chi^{-1} ,$ namely the map $\check\Phi :\OO\to\R$ given by:
\begin{equation}\label{nota:phi_pm}
\check{\Phi }(x_+,x_-):=\Phi (x_++x_-,x_+-x_-)\,,
\quad 
\text{for}\enskip (x_+,x_-)\in \OO.
\end{equation}
Provided $\Phi \in \F_1,$ we have the identities
\begin{equation}\label{id_nabla_pm}
\left[\begin{aligned}
[(\nabla _x+\nabla _y)\Phi]\circ \chi ^{-1} =\nabla _{+}\check{\Phi }
\\
[(\nabla _x-\nabla _y)\Phi]\circ\chi ^{-1} =\nabla _{-}\check{\Phi }
\end{aligned}\right.
\end{equation}
where $\nabla _{+},\nabla _{-}$ denote the gradients with respect to the new variables $x_+,x_-.$
In view of these relations, we will henceforth write (with a slight abuse of notation):
\[
\nabla _{\pm}[\Phi (x,y)] = \nabla _x\Phi (x,y) \pm \nabla _y\Phi (x,y)\,.
\]
\end{notaAdd}
\begin{proof}[Proof of Proposition \ref{pro:renormalizable}]
By definition we have 
\[
S^{\epsilon,*}_{st}
=:Z_{st}^kT^{-1}_\epsilon (\Gamma^k_x +\Gamma ^k_y) T_\epsilon \,.
\]
where for $k\leq K,$ $\Gamma^k :W^{1,\infty}(\R^d)\to L^\infty(\R^d)$ is the first order differential operator
\begin{equation}
 \label{nota_Gamma}
\Gamma^k =-\sigma^k\cdot \nabla -\div\sigma^k +\nu^k\,.
\end{equation}

Intuitively, the problematic terms are those that contain derivatives. Indeed, whenever we differentiate $T_\epsilon\Phi$, we obtain a blow up in $\epsilon$. The key observation is then that the blow up only appears in the $x_-$ direction and the bad terms are always multiplied by $\sigma(x)-\sigma(y)$ (or similar), which allows to compensate this blow-up by making use of the higher regularity of $\sigma$.

\item[\indent\textit{Estimate on $S^{\epsilon}_{st}$ in $\L(\F_{-0},\F_{-1}).$}]
For any $\Phi \in \F_1$, we have
\begin{equation}\label{eq:Vstar_epsilon}
\begin{aligned}
(\sigma^k(x)\cdot \nabla_x &+\sigma^k(y)\cdot \nabla _y)(T_\epsilon \Phi )(x,y)\\
&=\sigma ^k(x)\cdot T_\epsilon 
\Big(\tfrac{1}{2}\nabla _+\Phi  
+\tfrac{1}{2\epsilon }\nabla _-\Phi  \Big)
+\sigma ^k (y)\cdot T_\epsilon \Big(\tfrac{1}{2}\nabla _+\Phi - \tfrac{1}{2\epsilon }\nabla _-\Phi\Big)  
\\
&=\Big(\frac{\sigma ^k (x)+\sigma ^k (y)}{2}\Big)\cdot T_\epsilon \nabla _+\Phi 
+\Big(\frac{\sigma ^k (x)-\sigma ^k (y)}{2\epsilon }\Big)\cdot T_\epsilon \nabla _-\Phi.
\end{aligned}
\end{equation}
Now, making use of the notations \eqref{nota:m_epsilon} and \eqref{nota:finite_difference}
we obtain that for a.e.\ $x,y$ in $\OO$
\begin{equation}
\label{Gamma_sym}
T_\epsilon ^{-1}(\Gamma^k_x+\Gamma^k _y) T_\epsilon 
\equiv-\left(\m\sigma ^k\right)(x_+)\cdot \nabla _{+}
-\left(\df\sigma ^k\right) (x_+)\cdot \nabla _-
+2\m(-\div\sigma ^k +\nu^k)
\end{equation}
and we abbreviate
\begin{equation}
\label{abr:x_plus}
x_+:=\frac{x+y}{2}\,,\quad x_-:=\frac{x-y}{2}\,.
\end{equation}
For the first term in \eqref{Gamma_sym}, we have
\[
\begin{aligned}
\nnn{\left(\m\sigma^k\right)\cdot\nabla _{+}\Phi }{0}
&\equiv
\esssup_{x_+,x_-}\Big|\Big(\frac{\tt_{-\epsilon x_-}+\tt_{\epsilon x_-}}{2}\Big)\sigma^k(x_+)\cdot \nabla _{+}\check \Phi (x_+,x_-)
\Big|
\\
&\leq |\sigma |_{L^\infty}\nnn{\Phi }{1}\,.
\end{aligned}
\]
For the second term, using Lemma \ref{lem:finite-difference} and
the fact that a.e., $\spt\check\Phi(x_+,\cdot ) \subset B_1$ we have
\begin{equation}\label{csq:taylor}
\nnn{(\df\sigma^k)\cdot\nabla _-\Phi  }{0}
\leq |\nabla \sigma  |_{L^\infty}\nnn{\Phi }{1}\,.
\end{equation}
Concerning the last term in \eqref{Gamma_sym}, we have
\[
\begin{aligned}
\nnn{2\m(-\div\sigma^k +\nu^k) \Phi}{0}
&\leq \esssup_{x_+,x_-}|\big(\tt_{-\epsilon x_-}+\tt_{\epsilon x_-}\big)(\nu^k -\div\sigma^k) (x_+)\check\Phi (x_+,x_-)|
\\
&\leq 2(|\nu |_{L^\infty} +|\div\sigma |_{L^\infty})\nnn{\Phi }{0}\,.
\end{aligned}
\]
Summing these bounds, we obtain the first estimate, namely:
\begin{equation}\label{estimate_Gamma1}
| S^{\epsilon,*}_{st}|_{\L(\F_1,\F_0)}\leq C(|\sigma |_{W^{1,\infty}},|\nu |_{L^\infty})\omega _Z(s,t)^\alpha \,.
\end{equation}

\item[\indent\textit{Estimate on $S^{\epsilon}_{st}$ in $\L(\F_{-2},\F_{-3}).$}]
Let $\Phi\in \mathcal{F}_3$. First we observe  that since the change of coordinates $\sqrt2\chi $ is a rotation, in order to estimate $\nnn{S^{\epsilon,*}_{st}\Phi}{2}$, it is sufficient to estimate $\nnn{(\nabla_\pm)^2 S^{\epsilon,*}_{st}\Phi}{0}$. To this end, we further note that the only critical term in \eqref{Gamma_sym} is the second one which contains $\epsilon^{-1}$. But in that case, it holds
\begin{align}\label{eq:123}
\begin{aligned}
\nabla_-[(\df\sigma^k)(x_+)\cdot\nabla_-\Phi]&=\boldsymbol m^{x_-}_\epsilon(\nabla\sigma^k)(x_+)\cdot\nabla_-\Phi+(\df\sigma^k)(x_+)\cdot\nabla^2_-\Phi,\\
\nabla_+[(\df\sigma^k)(x_+)\cdot\nabla_-\Phi]&=
\df(\nabla\sigma^k)(x_+)\cdot\nabla_-\Phi+(\df\sigma^k)(x_+)\cdot\nabla_+\nabla_-\Phi,
\end{aligned}
\end{align}
where, similarly as before, Lemma \ref{lem:finite-difference} yields that a.e.\ on $\Omega :$
\[
|\df (\nabla\sigma^k)|\leq |\sigma|_{W^{2,\infty}},\qquad |\df \sigma^k|\leq |\sigma|_{W^{1,\infty}}\,.
\]

By the same arguments we can proceed further and apply  $\nabla_\pm$ to \eqref{eq:123}. This finally leads to
\begin{equation*}
|S^{\epsilon,*}_{st}|_{\L(\F_3,\F_2)}\leq 
C(|\sigma |_{W^{3,\infty}},|\nu |_{W^{2,\infty}})
\omega_Z(s,t)^{\alpha}\,.
\end{equation*}

\item[\indent\textit{Estimates on $\SS^\epsilon$ in $\L(\F_{-0},\F_{-2})$ and $\L(\F_{-1},\F_{-3})$}.]
Using geometricity, renormalizability of the term $\SS^\epsilon $ can be reduced to the previous cases.
This is a consequence of the identity
\begin{equation}
\label{SS_symmetric_2}
\SS^*_{st}
\overset{\mathrm{def}}{=} \ZZ_{st}^{k\ell}(\Gamma ^\ell _x\Gamma ^k_x +\Gamma ^\ell _y\Gamma ^k_y)
+Z^k_{st}Z^\ell_{st} \Gamma ^\ell _y\Gamma ^k_x
=\ZZ_{st}^{k\ell}(\Gamma _x^\ell +\Gamma _y^\ell )(\Gamma _x^k+\Gamma _y^k)
\,,\quad \quad (s,t)\in \Delta \,,
\end{equation}
where $\Gamma^k $ is as in \eqref{nota_Gamma}.

Indeed, emphasizing summations, 
denoting by $\sym \ZZ_{st}^{k\ell }:=\frac12(\ZZ_{st}^{k\ell }+\ZZ_{st}^{\ell k})\equiv \frac12Z_{st}^kZ_{st}^\ell $ and $\anti \ZZ_{st}^{k\ell }:=\frac12(\ZZ_{st}^{k\ell }-\ZZ_{st}^{\ell k}),$ and splitting the term $\sum_{k,\ell }Z_{st}^kZ_{st}^\ell \Gamma ^\ell _y\Gamma ^k_x $ into two equal parts, one can write:
\begin{equation}
\begin{aligned}
\label{SS_symmetric_1}
\SS^*
&=\sum_{k,\ell }(\sym\ZZ_{st} ^{k\ell}+\anti\ZZ_{st}^{k\ell})(\Gamma_x^\ell  \Gamma_x^k +\Gamma_y^\ell \Gamma^k_y)
+ \sum_{k,\ell }\frac{Z^k_{st}Z_{st}^\ell }{2}\Gamma^\ell _y\Gamma^k_x   + \sum_{k',\ell' }\frac{Z_{st}^{\ell'} Z_{st}^{k'}}{2}\Gamma^{\ell '}_x \Gamma^{k'}_y
\\
&=\sum_{k,\ell }\sym\ZZ_{st}^{k\ell}(\Gamma _x^\ell  +\Gamma _y^\ell)(\Gamma _x^k  +\Gamma _y^k) +\sum_{k,\ell }\anti \ZZ_{st}^{k\ell }(\Gamma_x^\ell  \Gamma_x^k +\Gamma_y^\ell \Gamma^k_y)\,.
\end{aligned}
\end{equation}
However, using antisymmetry, the second term above can be written as
$\sum_{k,\ell }\anti\ZZ_{st}^{k\ell}(\Gamma _x^\ell  +\Gamma _y^\ell)(\Gamma _x^k  +\Gamma _y^k).$
Summing in \eqref{SS_symmetric_1}, we see that \eqref{SS_symmetric_2} holds.

Now, let $\Phi $ in $\F_2$ and estimate
\[
\begin{aligned}
\nnn{\SS_{st}^{\epsilon,*}\Phi }{0}
&\leq |\ZZ_{st}^{k\ell}|\nnn{T_\epsilon ^{-1}(\Gamma ^\ell _x+\Gamma ^\ell _y)T_\epsilon T_\epsilon ^{-1}(\Gamma ^k _x+\Gamma ^k _y)T_\epsilon \Phi }{0}
\\
&\leq C(|\sigma |_{W^{1,\infty}},|\nu|_{L^\infty} )\omega _Z(s,t)^{2\alpha }\nnn{T_\epsilon ^{-1}(\Gamma ^k _x+\Gamma ^k _y)T_\epsilon\Phi }{1}
\\
&\leq C(|\sigma |_{W^{2,\infty}},|\nu|_{W^{1,\infty}} )\omega _Z(s,t)^{2\alpha }\nnn{\Phi }{2}\,,
\end{aligned}
\]
where we have used the bounds obtained in the first part.
This yields our first estimate.

The second estimate again reduces to the previous bounds:
we have for $\Phi \in \F_3:$
\[
\begin{aligned}
\nnn{\SS_{st}^{\epsilon,*}\Phi }{1}
&\leq |\ZZ_{st}^{k\ell}|\nnn{T_\epsilon ^{-1}(\Gamma ^\ell _x+\Gamma ^\ell _y)T_\epsilon T_\epsilon ^{-1}(\Gamma ^k _x+\Gamma ^k _y)T_\epsilon \Phi }{1}
\\
&\leq C(|\sigma |_{W^{2,\infty}},|\nu|_{W^{1,\infty}} )\omega _Z(s,t)^{2\alpha }\nnn{T_\epsilon ^{-1}(\Gamma ^k _x+\Gamma ^k _y)T_\epsilon\Phi }{2}
\\
&\leq C(|\sigma |_{W^{3,\infty}},|\nu|_{W^{2,\infty}} )\omega _Z(s,t)^{2\alpha }\nnn{\Phi }{3}\,,
\end{aligned}
\]
which proves the claimed bound.
\end{proof}

\subsection{Uniform bound on the drift.}

We proceed with a uniform estimate for the drift in \eqref{eq:U_epsilon}.
\begin{proposition}
\label{pro:drift}
There exists a control  $\omega _\Pi,$ depending on $u$ in $\Bc,$ $b$ in $L^{2r}L^{2q},$ $c$ in $L^rL^q,$
and on  $M,r,q,$ such that uniformly in $\epsilon \in(0,1)$, for every $(s,t)$ in $\Delta :$
\begin{equation}
\nnn{\delta \Pi_{st}^\epsilon}{-1} \leq \omega _\Pi(s,t)\,.
\end{equation}
Furthermore, we have the bound
\begin{multline}
\omega _{\Pi}(s,t)\leq C(M,r,q)\Big(\|u\|_{\infty,2;[s,t]}\|\nabla u\|_{1,2;[s,t]} 
\\
+ \|\nabla u\|^2_{2,2;[s,t]} + \|b\|_{2r,2q;[s,t]}\|u\|^2_{\Bc_{s,t}} + \|c\|_{r,q;[s,t]}\|u\|^2_{\Bc_{s,t}}\Big)\,,
\end{multline}
uniformly over $(s,t)\in\Delta ,$
where $C>0$ depends only the listed quantities.
\end{proposition}
Let $k\geq 0,$ and assume that we are given a measurable $v(x,y)$ in $\F_{-k}(\Omega _\epsilon ),$
such that its trace $\trace v$ onto the diagonal $\Gamma :=\{x,y\in\R^{2d}:x=y\}$ is a well-defined element in $(W^{k,\infty}(\Gamma ))^*$ (this is the case for instance if $v(x,y)=f^1(x)f^2(y)$ where $f^1\in W^{-k,2}(\O)$ and $f^2\in W^{k,2}(\O)$).
The adjoint of $T_\epsilon $ is given a.e.\ on $\Omega$ by the formula
\begin{equation}
\label{T_epsilon_star}
T_\epsilon^*v(x,y)=
2^{-d}\left(\tt_{-\epsilon \frac{x-y}{2}}\otimes \tt_{\epsilon \frac{x-y}{2}}\right)v\Big(\frac{x+y}{2},\frac{x+y}{2}\Big)\,,
\quad (x,y)\in\Omega \,,
\end{equation}
which, integrating against $\Phi \in \F_k,$
and letting $(x,+,x_-):=\chi (x,y)$ yields the representation
\begin{equation}
\label{representation}
\begin{aligned}
\langle T_\epsilon ^*v,\Phi \rangle
&=\iint_{\O\times B_1} \left(\tt_{-\epsilon x_-}\otimes \tt_{\epsilon x_-}\right)v(x_+,x_+)\Phi (x_++x_-,x_+-x_-)\d x_+\d x_-
\\
&=\int_{B_1}~_{W^{k,\infty}(\O )^*}\Big\langle \trace\left(\tt_{-\epsilon x_-}\otimes \tt_{\epsilon x_-}\right) v,\check\Phi (\cdot ,x_-) \Big\rangle_{W^{k,\infty}(\O)}\d x_-\,.
\end{aligned}
\end{equation}
\begin{proof}
By definition, we have $\delta \Pi_{st}^\epsilon =\int_s^t\langle Au\otimes u+u\otimes Au,T_\epsilon \Phi \rangle\d r.$
For notational simplicity, we now fix $r$ in $[s,t],$ and denote by $u:=u_r,$ $a^{ij}=a^{ij}_r,$ and so on.
For $\Phi \in\F_1$ we have
\begin{equation}\label{eq:drift}
\begin{aligned}
\langle Au\otimes u&+u\otimes Au ,T_\epsilon \Phi \rangle
\\
&=~_{\F_{-1}(\Omega _\epsilon )}\big\langle\div_x(a_x\nabla_x u_x)u_y,T_\epsilon \Phi \big\rangle_{\F_1(\Omega _\epsilon )}
+~_{\F_{-1}(\Omega _\epsilon )}\big\langle\div_y(a_y\nabla_y u_y)u_x,T_\epsilon \Phi \big\rangle_{\F_1(\Omega _\epsilon )}
\\
&\hspace{3em}+\iint_{\Omega_\epsilon } b^i(x)\partial _iu(x)u(y)T_\epsilon \Phi\d x\d y
+\iint_{\Omega_\epsilon } b^i(y)\partial _iu(y)u(x)T_\epsilon \Phi\d x\d y
\\
&\hspace{6em}+\iint_{\Omega_\epsilon } c(x)u(x)u(y)T_\epsilon \Phi \d x\d y
+\iint_{\Omega_\epsilon } c(y)u(x)u(y)T_\epsilon \Phi \d x\d y
\,,
\\
&=:\T_a^1+\T_a^2+\T_b^1+\T_b^2+\T_c^1+\T_c^2\,.
\end{aligned}
\end{equation}

\item[\indent\textit{Estimate on $\T_a.$}]
Using \eqref{representation},
the first term can be written as:
\begin{equation}
\label{Ta}
\begin{aligned}
\T_a^1
&=\int_{B_1}\tensor[_{(W^{1,\infty}(\O))^*}]{\Big\langle \trace\big[\tt _{-\epsilon x_-}\div (a\nabla u)\tt_{\epsilon x_-}u\big],\check\Phi(\cdot ,x_-) \Big\rangle}{_{W^{1,\infty}(\O)}}\d x_-
\\
&=\int_{B_1}\tensor[_{(W^{1,\infty}(\O))^*}]{\Big\langle \trace\big[\div_{x_+}\big( \tt _{-\epsilon x_-}(a\nabla u)\big)\tt_{\epsilon x_-}u\big],\check\Phi (\cdot ,x_-) \Big\rangle}{_{W^{1,\infty}(\O)}}\d x_-
\\
&=\int_{B_1 }\tensor[_{W_{+}^{-1,2}}]{\Big\langle \div_{x_+}\big( \tt _{-\epsilon x_-}(a\nabla u)\big),\tt_{\epsilon x_-}u\check\Phi(\cdot ,x_-) \Big\rangle}{_{W_{+}^{1,2}}}\d x_-
\\
&=-\int_{B_1 }\Big(\tt _{-\epsilon x_-}[a\nabla u],\nabla_{+}\big( \tt_{\epsilon x_-}u(x_+) \big)\check\Phi (\cdot ,x_-)+ \tt_{\epsilon x_-}u \nabla _+\check\Phi (\cdot ,x_-)\Big)_{L^2_+}\d x_-
\\
&=-\int_{B_1 }\Big(\tt _{-\epsilon x_-}[a\nabla u],\tt_{\epsilon x_-}\nabla u \check\Phi (\cdot ,x_-)\Big)_{L^2_+}\d x_-
\\
&\quad\quad \quad \quad \quad  \quad \quad \quad 
-\int_{B_1 }\Big(\tt _{-\epsilon x_-}[a\nabla u], \tt_{\epsilon x_-}u \nabla _+\check\Phi (\cdot ,x_-)\Big)_{L^2_+}\d x_-\,.
\end{aligned}
\end{equation}
Using that
$\tt_{\epsilon x_-}$ leaves the $L^2$ norm invariant for every fixed $x_-$ in $\O,$ we have
\[
\begin{aligned}
\T_a^1
&\leq
\int_{B_1}|\tt _{-\epsilon x_-}[a\nabla u]|_{L^2_+}|\tt_{\epsilon x_-}\nabla u |_{L^2_+}|\check\Phi (\cdot ,x_-)|_{L^\infty_+}\d x_-
\\
&\quad \quad \quad 
+ \int_{B_1}|\tt _{-\epsilon x_-}[a\nabla u]|_{L^2_+}|\tt_{\epsilon x_-}u |_{L^2_+}|\nabla _+\check\Phi (\cdot ,x_-)|_{L^\infty_+}\d x_-
\\
&=
\int_{B_1}|a\nabla u|_{L^2_+}|\nabla u |_{L^2_+}|\check\Phi (\cdot ,x_-)|_{L^\infty_+}\d x_-
+ \int_{B_1}|a\nabla u|_{L^2_+}|u |_{L^2_+}|\nabla _+\check\Phi (\cdot ,x_-)|_{L^\infty_+}\d x_-\,,
\end{aligned}
\]
Hence, doing similar computations for $\T_a^2,$ it follows that 
\begin{equation}
\label{Ta}
\int_s^t \T_a\d r\leq 2M \left(\|\nabla u\|^2_{2,2}\nnn{\Phi }{0}+ \|\nabla u\|_{2,2}\|u\|_{\infty,2}\nnn{\Phi }{1}\right)\,.
\end{equation} 

\item[\indent\textit{Estimate on $\T_b.$}]
By \eqref{T_epsilon_star}, we have
\begin{equation}
\label{drift_Tb_1}
\begin{aligned}
\T^{1}_b
&=\iint_{B_1\times\O}\tt_{-\epsilon x_-}(b^i\partial _iu)(x_+)\tt_{\epsilon x_-}u(x_+)\check\Phi(x_+,x_-) \d x_+\d x_-
\\
&\leq\int_{B_1}|\tt_{-\epsilon x_-}b|_{L^{2q}_+}|\tt_{-\epsilon x_-}\nabla u|_{L^2_+}|\tt_{\epsilon x_-}u|_{L_+^{\frac{2q}{q-1}}}
\d x_-\nnn{\Phi }{0}\,.
\end{aligned}
\end{equation} 
Using H\"older, \eqref{consequence_rq}, and then proceeding similarly for $\T^2_b$, we obtain 
\begin{equation}
\label{Tb}
\begin{aligned}
\int_s^t(\T^1_b+\T^2_b)\d r\leq 2\|b\|_{2r,2q;[s,t]}\|\nabla u\|_{2,2;[s,t]}\|u\|_{\frac{2r}{r-1},\frac{2q}{q-1};[s,t]}\nnn{\Phi }{0}
\\
\leq 2\beta \|b\|_{2r,2q;[s,t]}\|\nabla u\|_{2,2;[s,t]}\|u\|_{\Bc_{s,t}}\nnn{\Phi }{0}\,.
\end{aligned}
\end{equation}

\item[\indent\textit{Estimate on $\T_c.$}]
Similarly, it suffices to show the estimate for $\T^1_c.$
Using again \eqref{T_epsilon_star}, there comes
\begin{equation}
\label{Tc}
\T^1_c=\iint_{B_1\times\O}\tt_{-\epsilon x_-}[c u](x_+)\tt_{\epsilon x_-}u(x_+)\check\Phi\d x_+\d x_-\,.
\end{equation}
Hence, H\"older inequality and \eqref{consequence_rq} yield
\begin{equation}
\label{drift_T3}
\int_s^t(\T_c^1+\T_c^2)\d r
\leq 2\|c\|_{r,q;[s,t]}\|u\|_{\frac{2r}{r-1},\frac{2q}{q-1};[s,t]}^2
\leq 2\beta ^2\|c\|_{r,q;[s,t]}\|u\|_{\Bc_{s,t}}^2\nnn{\Phi }{0}\,.
\end{equation}
Combining \eqref{Ta}, \eqref{Tb} and \eqref{drift_T3},
we obtain the claimed bound.
\end{proof}
\subsection{The proof of uniqueness}
\label{subsec:uniqueness}
Finally, we have all in hand to complete the proof of uniqueness.
To this end, we let $\omega _\Pi (s,t)$ be as in Proposition \ref{pro:drift}
and recall that according to Proposition \ref{pro:apriori}, the following uniform estimate holds true for the remainder term:
\begin{equation}\label{bound:U_natural}
\nnn{\u^{\natural,\epsilon }_{st}}{-3}\leq C\Big(\sup_{r\in[s,t]}\nnn{\u^\epsilon_r}{-0}\omega_{B} (s,t)^{3\alpha }+\omega _\Pi (s,t)\omega _B(s,t)^{\alpha }\Big).
\end{equation}
for $(s,t)\in\Delta ,$ under the smallness condition $\omega _B(s,t)\leq L,$ for some $L>0.$
Note that for every $u^1,u^2\in L^2$ we have
\begin{multline}
\label{T_star_u}
\iint_{B_1\times\O}|T_\epsilon ^*(u^1\otimes u^2)(x,y)|\d x\d y
\\
= \iint_{B_1\times\O}|\tt_{-\epsilon x_-}u^1(x_+)\tt_{\epsilon x_-}u^2(x_+)|\d x_+\d x_-\leq C|u^1|_{L^2}|u^2|_{L^2}\,.
\end{multline}
Since we have the embedding $L^1(\Omega )\subset L^\infty(\Omega )^*,$
using \eqref{T_star_u} with $u^1=u^2=u,$ there comes
\begin{equation}
\label{strict_embedding}
\sup_{r\in[s,t]}\nnn{\u^\epsilon_r}{-0}
\leq C\sup_{r\in[s,t]}|u_r|_{L^2}^2\,,
\end{equation}
uniformly in $\epsilon>0.$
Combining the latter with \eqref{bound:U_natural} yields therefore a uniform bound of the remainder $\u^{\natural,\epsilon }.$

Now, take $\phi\in W^{3,\infty}(\O)$ and $\psi\in C_c^\infty(B_1)$ with $\int_{B_1}\psi \d x=1$
and define
\begin{equation}\label{Phi_R}
\Phi (x,y):=\phi \left(\tfrac{x+y}{2}\right)\psi \left(\tfrac{x-y}{2}\right)\,.
\end{equation}
Observe furthermore that
$\nnn{\Phi }{3}\leq C|\phi |_{W^{3,\infty}}\equiv C\nn{\phi }{3}$ for a positive constant depending on $\psi $ only.
\begin{lemma}
\label{lem:limit}
Let $u^2_t(x):=u_t(x)^2$ which defines an element of
$$C(I;L^1(\O ))\subset C(I;(L^\infty(\O))^*).$$ Then we have for every $\phi $ in $W^{3,\infty}:$
\begin{multline}\label{eq:h}
\langle\delta u^2_{st},\phi \rangle=\int_s^t\big(-2\langle a^{ij}\partial _ju,\partial _i(u\phi )\rangle +\langle b^i\partial _i(u^2)+2cu^2,\phi \rangle\big)\d r
\\
+\langle u^2_s,\hat B_{st}^{*}\phi \rangle
+\langle u^2_s,\hat\BB_{st}^{*}\phi \rangle
+ \langle u_{st}^{2,\natural},\phi \rangle\, ,
\end{multline}
where $\hat \B\equiv(\hat B,\hat \BB)$ is obtained by replacing $\nu $ by $2\nu $ in the definition of $\B,$
and $u^{2,\natural}$ belongs to $\V_ {2,\rm{loc}}^{1+}(I,(W^{3,\infty})^*)$.
Moreover the latter remainder term is estimated by the right hand side of \eqref{bound:U_natural}.
\end{lemma}
\begin{proof}
Recall that,
by definition of $\u^{\natural,\epsilon }:$
\[
\langle\delta \u _{st},T_\epsilon \Phi  \rangle
=
\langle\delta \Pi_{st},T_\epsilon \Phi \rangle
+\langle S _{st}\u_s,T_\epsilon \Phi \rangle
+\langle\SS _{st}\u_s,T_\epsilon \Phi \rangle
+\langle \u^{\natural,\epsilon }_{st},\Phi \rangle\,,
\]
where, gathering the terms in \eqref{Ta}, \eqref{Tb}, \eqref{Tc}, it holds:
\[
\begin{aligned}
\langle\delta \Pi_{st},T_\epsilon \Phi \rangle
&=\iiint\limits_{[s,t]\times B_1\times\O}\Big[
-\tt_{-\epsilon x_-}(a^{ij}\partial _ju)(\tt_{\epsilon x_-}\partial _iu)(x_+)\phi (x_+)\psi (x_-)
\\[-1.4em]
&\quad \quad \quad \quad \quad \quad \quad \quad 
-\tt_{-\epsilon x_-}(a^{ij}\partial _ju)(\tt_{\epsilon x_-}u)(x_+)\partial _i\phi (x_+)\psi (x_-)
\\
&\quad \quad \quad \quad \quad \quad \quad \quad 
-\tt_{\epsilon x_-}(a^{ij}\partial _ju)(\tt_{-\epsilon x_-}\partial _iu)(x_+)\phi (x_+)\psi (x_-)
\\
&\quad \quad \quad \quad \quad \quad \quad \quad 
-\tt_{\epsilon x_-}(a^{ij}\partial _ju)(\tt_{-\epsilon x_-}u)(x_+)\partial _i\phi (x_+)\psi (x_-)
\\
&\quad\quad \quad \quad \quad \quad \quad \quad 
+\tt_{-\epsilon x_-}(b^i\partial _iu)(x_+)\tt_{\epsilon x_-}u(x_+)\phi (x_+)\psi (x_-)
\\
&\quad \quad \quad \quad \quad \quad \quad \quad  
+\tt_{\epsilon x_-}(b^i\partial _iu)(x_+)\tt_{-\epsilon x_-}u(x_+)\phi (x_+)\psi (x_-)
\\
& \quad \quad \quad \quad \quad \quad \quad \quad 
+\tt_{-\epsilon x_-}(cu)(x_+)\tt_{\epsilon x_-}u(x_+)\phi (x_+)\psi (x_-)
\\
&
\quad \quad \quad \quad \quad \quad \quad \quad 
+\tt_{\epsilon x_-}(cu)(x_+)\tt_{-\epsilon x_-}u(x_+)\phi (x_+)\psi (x_-)
\Big]\d x_+\d x_-\d r
\\
&=:\sum\nolimits_{i=1}^{8}\T^i
\,.
\end{aligned}
\]

\item[\indent\textit{Step 1: convergence of the drift.}]
Property \eqref{prop:translation_1},
Assumption \ref{ass:a} and the dominated convergence theorem imply
\[
\begin{aligned}
\T^1+\T^3\to&-2\int_{B_1} \psi (x_-)\Big(\iint_{\stO}a^{ij}(x_+)\partial_j u(x_+)\partial _iu(x_+)\phi (x_+)\d x_+\d r\Big)\d x_-
\\
&\quad \quad \quad \quad \quad 
\equiv -2\int_s^t\left\langle a^{ij}\partial _ju,\partial _iu\phi \right\rangle\d r\,,
\end{aligned}
\]
since $\int_{B_1}\psi\d x_-=1.$
Likewise, it holds
$\T^2+\T^4\to-2\int_s^t\left\langle a^{ij}\partial _ju,u\partial _i\phi \right\rangle\d r.$

Now, because of \eqref{prop:translation_1},
it follows that
\[
\T^5+\T^6\to2\int_s^t\langle b^i\partial _iu,u\phi  \rangle\d r\,,
\quad \text{and}\quad 
\T^7+\T^8\to2\int_s^t\langle cu^2,\phi \rangle\d r\,.
\]
Summing all the terms above, we end up with the claimed convergence.

\item[\indent\textit{Step 2: Convergence of the left hand side.}]
We have:
\[
\begin{aligned}
\langle\delta \u_{st},T_\epsilon \Phi \rangle
&=\iint_{\Omega _\epsilon } \delta u_{st}(x)\big(\frac{u_s(y)+u_t(y)}{2}\big)T_\epsilon \Phi \d x\d y
\\
&\quad \quad 
+\iint_{\Omega _\epsilon } \delta u_{st}(y)\big(\frac{u_s(x)+u_t(x)}{2}\big)T_\epsilon \Phi \d x\d y
\\
&=\iint_{B_1\times\O}\tt_{-\epsilon x_-}\delta u_{st}(x_+)\tt_{\epsilon x_-}\big(\frac{u_s+u_t}{2}\big)(x_+)\phi(x_+)\psi (x_-) \d x_+\d x_-
\\
&\quad \quad 
+\iint_{B_1\times\O} \tt_{\epsilon x_-}\delta u_{st}(x_+)\tt_{-\epsilon x_-}\big(\frac{u_s+u_t}{2}\big)(x_+)\phi(x_+)\psi (x_-) \d x_+\d x_-\,.
\end{aligned}
\]
Using again the strong continuity of $(\tt_a)_{a\in\O}$ in $L^2,$ it holds 
\begin{equation}
\begin{aligned}
\langle\delta \u_{st},T_\epsilon \Phi \rangle 
\to
&\iint_{B_1\times\O}\psi (x_-) \delta u_{st}(x_+)\big(\frac{u_s+u_t}{2}\big)(x_+)\phi(x_+)\d x_+\d x_-
\\
&\quad \quad 
+\iint_{B_1\times\O} \psi (x_-)\delta u_{st}(x_+)\big(\frac{u_s+u_t}{2}\big)(x_+)\phi(x_+)\d x_+\d x_-
\\
&\equiv \langle\delta (u^2)_{st},\phi  \rangle\,.
\end{aligned}
\end{equation}

\item[\indent\textit{Step 3: convergence the driver.}]
Let $1>\delta >0.$
Since $C^\infty(\O)$ is dense in $L^2(\O),$
we can write $u=v+w$ where $v\in C^\infty$ is such that $|v|_{L^2}\leq 2|u|_{L^2}$ and $|w|_{L^2}\leq\delta .$
Hence for every $\delta >0,$ we have
\begin{multline}
\u=\v+\w\,,\quad\text{where}\enskip \v\equiv v\otimes v\in C^\infty(\OO)
\\
\text{and}\quad |\w|_{L^1(\OO)}\equiv|v\otimes w+w\otimes v+w\otimes w|_{L^1}\leq 4|u|_{L^2}\delta +\delta ^2 \leq C\delta \,.
\end{multline}
where we use \eqref{T_star_u}.
Since $\epsilon ^{-d}\psi (\frac{x_-}{\epsilon })$ approximates the identity, changing variables as before and then using dominated convergence, we have
\[
\langle S\v,T_\epsilon \Phi \rangle
\equiv\iint_{\OO}\big(Bv(x)v(y)+v(x)Bv(y)\big)\phi (\tfrac{x+y}{2}) (2\epsilon )^{-d}\psi (\tfrac{x-y}{2\epsilon })\d x\d y
\to\langle \hat B(v^2),\phi \rangle\,,
\]
and also
\begin{multline*}
\langle \SS\v,T_\epsilon \Phi \rangle
\equiv\iint_{\OO}\big(\BB v(x)v(y)+Bv(x)Bv(y) +v(x)\BB v(y)\big)\phi (\tfrac{x+y}{2})(2\epsilon )^{-d}\psi (\tfrac{x-y}{2\epsilon })\d x\d y
\\
\to\langle \hat \BB(v^2),\phi \rangle\,.
\end{multline*}
Using Proposition \ref{pro:renormalizable}, we have
\[
\begin{split}
\limsup_{\epsilon \to0}\langle S\w,  T_\epsilon \Phi \rangle
&\equiv \limsup_{\epsilon \to0}\langle T_\epsilon ^*\w, T_\epsilon ^{-1}S^*  T_\epsilon \Phi \rangle\\
&\leq C(|v|_{L^2}|w|_{L^2}+|w|_{L^2}^2)|\phi |_{W^{1,\infty}}\delta \leq C'|\phi |_{W^{1,\infty}}\delta \,.
\end{split}
\]
Similarly:
\[
\limsup_{\epsilon \to0}\langle \SS\w, T_\epsilon \Phi \rangle
\equiv \limsup_{\epsilon \to0}\langle T_\epsilon ^*\w, T_\epsilon ^{-1}\SS^*  T_\epsilon \Phi \rangle
\leq C|\phi |_{W^{2,\infty}}\delta \,.
\]
Since $\delta >0$ is arbitrary, we conclude that 
\begin{equation}
\lim_{\epsilon \to0}\langle S\u,T_\epsilon \Phi\rangle
=\langle \hat B (u^2),\phi \rangle\,.
\end{equation}
and
\begin{equation}
\begin{aligned}
\lim_{\epsilon \to0}\langle\SS \u ,T_\epsilon \Phi \rangle 
=\langle \hat\BB(u^2),\phi \rangle\,.
\end{aligned}
\end{equation} 

\item[\indent\textit{Conclusion.}]
By \eqref{bound:U_natural}-\eqref{strict_embedding} we have the following estimate, for $(s,t)\in\Delta :$
\[
\langle \u^{\natural,\epsilon }_{st},\Phi \rangle
\leq C\left(\|u\|_{\infty,2}^2\omega _Z(s,t)^{3\alpha }+\omega _\mu  (s,t)\omega_Z (s,t)^{\alpha}\right)|\phi |_{W^{3,\infty}}\,.
\]
From the Banach-Alaoglu theorem,
and since the other terms in the equation converge,
we see that for each $(s,t)\in\Delta ,$ there exists a linear functional $u^{2,\natural}_{st} \in (W^{3,\infty})^*,$ such that
\[
\langle \u^{\natural,\epsilon }_{st} ,\Phi \rangle
\to
\langle u^{2,\natural}_{st},\phi \rangle\,.
\]
for every $\phi $ in $W^{3,\infty}.$
From \eqref{bound:U_natural} we see that 
$u^{2,\natural}$ belongs to $\V^{1+}_{2,\rm{loc}}(I;(W^{3,\infty})^*),$
proving therefore that \eqref{eq:h} is fulfilled.
\end{proof}
We can now establish uniqueness of weak solutions in $\Bc.$
\begin{proof}[Proof of Theorem \ref{thm:solvability}, uniqueness part]
Testing \eqref{eq:h} against
$\phi :=1\in W^{3,\infty},$ and proceeding as in Section \ref{sec:energy}, we see from the Rough Gronwall Lemma 
that every weak solution to \eqref{zakai} in the sense of Definition \ref{def:weak_sol} satisfies
\begin{align*}
&\|u\|_{C([0,T];L^2)}^2+\min(1,m)\int_0^T|\nabla u_r|_{L^2}^2\d r\\
&\qquad\leq C\Big(\omega _Z,|\sigma|_{W^{3,\infty}},|\nu|_{W^{2,\infty}},M,\|b\|_{2r,2q},\|c\|_{r,q},\alpha ,T\Big)|u_0|^2_{L^2}\,.
\end{align*}
which gives \eqref{energy_bound}.
By linearity we deduce that there cannot be more than one weak solution for \eqref{zakai}, hence uniqueness is proven.
\end{proof}
\section{Existence and stability}
\label{sec:continuity}

Finally, we intend to prove existence and stability of weak solutions to \eqref{zakai}. To this end, we approximate the driving signal by smooth paths such that the classical PDE theory applies and yields existence of a unique approximate solution. The results of Section \ref{sec:energy} yield uniform a priori estimates and the passage to the limit then follows from a compactness argument.

Let $z^n:I \to\R^K,n\in\N$, be a sequence of smooth paths. We define their canonical lift by $Z^n=\delta z^n$ and $\ZZ^n_{st}:=\int_s^t \delta z^n_{s r}\otimes \d z_r^n$ and
assume that $\Z^n\equiv(Z^n,\ZZ^n)$ approximates the given rough path $\Z\equiv(Z,\ZZ)$ in the sense that
\begin{equation}
\label{cv:Z}
d_{\mathscr C^\alpha }(\Z^n,\Z)\cv{n \to\infty} 0\,,
\quad \text{see \eqref{RP_metric}.}
\end{equation}
Let 
\begin{equation}
\label{convergences_abc}
\begin{aligned}
&u_0^n\to u_0\enskip \text{in}\enskip L^2\,,
\\
&a^n\to a \enskip\text{in}\enskip L^\infty\,,
\quad\enskip \enskip  \text{with}\quad a^n \in\mathcal {P}_{m,M}\,,
\\
&b^n \to b \enskip\text{in}\enskip L^{2r}(I;L^{2q})\,,
\quad 
c^n \to c \enskip\text{in}\enskip L^r(I;L^q)\,,
\\
&\sigma ^n\to\sigma \enskip\text{in}\enskip W^{3,\infty},
\quad\quad \quad 
\nu ^n\to\nu \enskip\text{in}\enskip W^{2,\infty},
\end{aligned}
\end{equation}
where $\mathcal{P}_{m,M}$ denotes the set of coefficents $a^{ij}\in L^\infty(I\times\O)$ such that Assumption \ref{ass:a} holds,
and let 
\[
A^n:=\partial _j\big(a^{n;ij}\partial _j\cdot \big)+b^{n;i}\partial _i+c^{n}\,,
\]
\[
B^n:=Z^{n;k}(\sigma ^{n;ki}\partial _i + \nu ^{n;k})\,,\quad 
\BB^n:=\ZZ^{n;k\ell}(\sigma ^{n;ki}\partial _i + \nu ^{n;k})(\sigma ^{n;\ell j}\partial _j + \nu ^{n;\ell } )\,.
\]
We can assume without loss of generality that
uniformly in $n:$
\begin{multline}\label{uniform:abc}
|u_0^n|_{L^2}+
\|a^n\|_{\infty,\infty}+
\|b^n\|_{2r,2q}+
\|c^n\|_{r,q}+
|\sigma ^n|_{W^{3,\infty}}+
|\nu ^n|_{W^{2,\infty}}
\\
\leq
1+
|u_0|_{L^2}+
\|a\|_{\infty,\infty}+
\|b\|_{2r,2q}+
\|c\|_{r,q}+
|\sigma |_{W^{3,\infty}}+
|\nu |_{W^{2,\infty}}\,,
\end{multline}
and that
\begin{equation}
\label{uniform:p-unbounded}
\left[
\begin{aligned}
&|B_{st}^n|_{\L(W^{-k,2},W^{-k-1,2})}\leq \omega _B(s,t)^{\alpha }\,,\quad k\in\{0,1,2\}\\
&|\BB_{st}^n|_{\L(W^{-k,2},W^{-k-2,2})}\leq \omega _B(s,t)^{2\alpha }\,,\quad k\in\{0,1\}\,,\quad  \text{for}\enskip (s,t)\enskip \text{in}\enskip \Delta \,,
\end{aligned}
\right.
\end{equation}
where $\omega _B$ is as in \eqref{nota:omega_B}.

Recall that since $z^n$ is smooth, existence and uniqueness of a weak solution $u^n\in\Bc_{0,T}$ to 
\[
\partial _tu^n=A^nu^n+\dot B^n u^n,\qquad u^n|_{t=0}=u^n_0,
\]
in the sense of distributions,
follows from the classical PDE theory (see the discussion in Section~\ref{ss:main} for more details).
Consequently, by Proposition \ref{pro:energy}, together with \eqref{uniform:abc} and \eqref{uniform:p-unbounded}, the $\Bc_{0,T}$-norm of $u^n$ is uniformly bounded, namely,
\begin{equation}\label{energy_inequality_n}
\|u^n\|^2_{\Bc_{0,T}}=\sup_{0\leq t\leq T}|u_t^n|_{L^2}^2 + \int_0^T|\nabla u^n _r|^2_{L^2}\d r\leq C(1+|u_0|_{L^2}^2)\,.
\end{equation}
Hence the Banach-Alaoglu theorem ensures (up to a subsequence) that
\begin{align}
\label{convergence_2}
&u^n \to  u\quad\text{and}\quad\nabla u^n \to \nabla u\enskip\quad \text{weakly in}\enskip L^2([0,T]\times\O),
\end{align}
and by weak lower semicontinuity of the norm we obtain
\begin{equation}\label{eq:estu}
\|u\|^2_{\Bc_{0,T}}<\infty\,.
\end{equation}
By \eqref{convergence_2} and the strong convergence $\|a^n-a\|_{\infty,\infty}\to0$
it follows that:
\[
\int_s^t\langle -a^{n;ij}_r\partial _ju^n,\partial _i\phi \rangle\d r
\to
\int_s^t\langle -a^{ij}_r\partial _ju,\partial _i\phi \rangle\d r
\]
for each $\phi $ in $W^{1,2}.$
Moreover, using \eqref{convergences_abc} we have
\[
\|(b^n-b)\phi \|_{2,2}\leq \|b^n-b\|_{2r,2q}\|\phi \|_{\frac{2r}{r-2},\frac{2q}{q-2}}\leq \beta \|b^n-b\|_{2r,2q}|\phi |_{W^{1,2}}T^{\frac{r-2}{2r}}\to 0\,,
\]
and similarly
\[
\|(c^n-c)\phi \|_{2,2}\leq \beta \|c^n-c\|_{r,q}|\phi |_{W^{1,2}}T^{\frac{r-2}{2r}}\,
\to 0\,.
\]
As a consequence, using the strong/weak convergence principle, we have also
\[
\int_s^t\langle b^{n;i}\partial _iu^n +c^nu^n,\phi \rangle\d r
\to
\int_s^t\langle b^i\partial _iu +cu,\phi \rangle\d r\,.
\]
The weak convergence obtained above is however not sufficient to take the pointwise limit in time, which is needed in order to pass to the limit on the left hand side of the equation as well as in the rough integral. For that purpose, we will show that the sequence $(u^n)$ satisfies an equicontinuity property in the space $W^{-1,2}.$
\begin{proof}[Proof of uniform equicontinuity]
Using Lemma \ref{lem:drift_existence}, \eqref{borne_controle_lambda}, \eqref{uniform:p-unbounded} and \eqref{energy_inequality_n},
we have the estimate
\begin{equation}
\label{an_estimate}
\begin{aligned}
\n{\int_{s}^tA^nu^n\d r}{-1}
&\leq \omega _n(s,t)\equiv(t-s)^{1/2}\uu_n(s,t)^{1/2}+\bb_n(s,t)^{1/(2r)}\aa_n(s,t)^{1/2}(t-s)^{\frac{r-1}{2r}}
\\
&\quad \quad +\cc_n(s,t)^{1/r}\uu_n(s,t)^{\frac{r-1}{2r}}(t-s)^{\frac{r-1}{2r}}
\end{aligned}
\end{equation}
where we adapt the notations \eqref{controls_abc} in an obvious way.

Moreover, from similar computations as that of 
Corollary \ref{cor:u_natural} (the proof is left to the reader)
we see that $u^n$ is a weak solution of 
\[
\d u^n= A^nu^n\d t+ \d \B^n u^n\,,
\]
in the sense of Definition \ref{def:weak_sol}, with respect to the scale $(W^{k,2})_{k\in\N}.$ Namely:
\begin{equation}
\label{equation_n}
\langle\delta u^n_{st},\phi  \rangle = \int_s^t\langle A^n u^n,\phi \rangle\d r + \langle B^n_{st} u^n_s,\phi \rangle +\langle \BB^n_{st} u^n_s,\phi \rangle +\langle u^{\natural,n}_{st},\phi \rangle
\end{equation}
for each $\phi $ in $W^{3,\infty},$ and $(s,t)\in\Delta .$
Applying Proposition \ref{pro:apriori} (more specifically using \eqref{bounds:gubinelli2}), we have the bound
\begin{equation}\label{bound:u_natural_n}
\begin{aligned}
\n{\delta u^n_{st}}{-1}
\leq
C\left(\omega _n(s,t)+\omega _n(s,t)^\alpha +\omega _B(s,t)^\alpha \right).
\end{aligned}
\end{equation}

Now, recall that 
$\aa_n(s,t)\leq C(1+2M\| u\|_{\Bc_{0,T}})\leq C_1,$ and, by \eqref{consequence_rq}, that $\uu_n(s,t)\leq C\|u\|_{\Bc_{0,T}}\leq C_2.$
Using moreover \eqref{convergences_abc}, the controls $\bb_n$ and $\cc_n$ are equicontinuous in the sense that
for each $\epsilon >0$ there exist $\delta >0$ such that
\[
|s-t|\leq \delta(\epsilon) \quad 
\Longrightarrow
\quad 
\bb_n(s,t)+\cc_n(s,t)\leq \frac{\epsilon ^2}{\max(C_1,C_2)^2}\,.
\]
Letting
$\delta ' \leq \min(\delta (\epsilon ),\epsilon ^2)$ and substituting in \eqref{an_estimate} we see that 
\[
\omega _n(s,t)\leq \epsilon \,,\quad \text{for all}\quad n\in \N\,,\quad \text{provided}\enskip |t-s|\leq \delta '\,.
\]
which shows uniform equicontinuity for $\omega _n,n\geq 0.$
By \eqref{bound:u_natural_n}, the same property holds for $\n{\delta u^n_{st}}{-1},$
hence uniform equicontinuity in $W^{-1,2}$ is proved.
\end{proof}

Thanks to the compact embedding
\[
L^2(\O)\hookrightarrow W^{-1,2}_{\mathrm{loc}}(\O)\,,
\]
the bound \eqref{energy_inequality_n}
shows that $(u_s^n)_{n\in\N}$ has a compact closure for each $s$ in $I.$
Using equicontinuity, a well-known infinite-dimensional version of Ascoli Theorem (we refer, e.g.\ to \cite{kelley1975general}) ensures that, up to a subsequence:
\begin{equation}
\label{pointwise_cv}
u_s^n\to u_s\quad \text{in}\enskip W_{\mathrm{loc}}^{-1,2}(\O)\enskip \text{uniformly for }\enskip s\in I\,.
\end{equation}
By \eqref{convergence_2}, \eqref{pointwise_cv}, 
fixing a compactly supported $\phi$ in $W^{1,2}(\O),$ we have for every $(s,t)\in\Delta :$
\[
\langle u_t^n-u_s^n,\phi \rangle\to \langle u_t-u_s,\phi \rangle\,.
\]
Furthermore, by \eqref{assumption_sigma_nu},
for each $\phi \in W^{3,2}$ with compact support, we have for each $k,\ell \in\{1,\dots,K\}:$
\begin{equation}
\label{finiteness}
\n{(\sigma^k \cdot \nabla )^*\phi}{1}\,,\enskip 
\n{(\sigma^k \cdot \nabla )^*(\sigma^\ell  \cdot \nabla )^*\phi}{1}\,,\enskip 
\n{\nu(\sigma^k \cdot \nabla )^*\phi}{1}\,,\enskip 
\n{(\sigma^k \cdot \nabla )^*(\nu\phi)}{1}\enskip 
<\infty\,.
\end{equation} 
Hence, using \eqref{pointwise_cv}, \eqref{cv:Z} and \eqref{finiteness} we see that
\[
\langle u^n_s ,B_{st}^{n,*}\phi \rangle\to
\langle u_s ,B_{st}^{*}\phi \rangle
\quad \text{while}\quad 
\langle u^n_s ,\BB_{st}^{n,*}\phi \rangle\to
\langle u_s ,\BB_{st}^{*}\phi \rangle\,.
\]
Using in addition the estimate \eqref{estimate_remainder},
we can take the limit in \eqref{equation_n}, so that $u$ satisfies the corresponding weak formulation of \eqref{zakai} for every compactly supported test function in $W^{3,2}$. Due to the energy bound \eqref{eq:estu} we may then relax the assumptions on the test function $\phi$ and deduce that $u$ is indeed a weak solution of \eqref{zakai}, with respect to the scale $(W^{k,2})_{k\in\N}.$ 

Therefore the existence part of Theorem \ref{thm:solvability} follows. It was already shown in Section \ref{sec:uniqueness} that the weak solution $u\equiv \mathfrak S(u_0,a,b,c,\sigma ,\nu ,\Z)$ is unique. In addition, due to our construction, every subsequence of $(u^{n})_{n\in\N}$ contains a further subsequence which  converges towards the same limit $\mathfrak S(u_0,a,b,c,\sigma ,\nu ,\Z).$ Hence we deduce that the original sequence $(u^n)_{n\in\N}$  converges.
Moreover, thanks to \eqref{convergence_2}, \eqref{pointwise_cv},
continuity of $\mathfrak S$ holds with respect to each of its variables. Indeed, it is enough to observe that the above proof remains applicable if $\Z^n$ is not necessarily a smooth approximation of $\Z$ in $\mathscr C_g.$
This completes the proof of the Theorem \ref{thm:solvability} and Theorem \ref{thm:continuity}.
\hfill\qed

\appendix

\section{Auxiliary results}
\subsection{Convergence of finite-difference approximations}
\label{app:finite_difference}
Recall \eqref{nota:finite_difference}.
We have the following.
\begin{lemma}
\label{lem:finite-difference}
Let $1\leq p<\infty.$
and fix $a$ in $\O.$
We have for every $\varphi \in W^{1,\infty}(\R^d):$
\[
|\boldsymbol\Delta _\epsilon ^a\varphi |_{L^\infty}\leq |a||\nabla \varphi |_{L^\infty}\,.
\]
Moreover, as $\epsilon $ goes to $0,$ we have
\[
\boldsymbol\Delta _{\epsilon}^a\varphi \to a\cdot \nabla \varphi \quad \text{strongly in}\quad L^p(\O)\,,
\]
provided
\begin{itemize}
 \item either $p<\infty$ and $\varphi \in W^{1,p}$;
 \item or $p=\infty$ and $\varphi \in W^{2,\infty}.$
\end{itemize}
\end{lemma}
\begin{proof}
The first bound
is an easy consequence of Taylor Formula, since for every $a\in\O$
\begin{equation}\label{taylor}
\boldsymbol\Delta _\epsilon ^a\varphi (x)=a\cdot \int_0^1\nabla \varphi \big(x+\epsilon (2\theta -1)a\big)\d \theta\,.
\end{equation}

\textit{Case $p\in[1,\infty).$}
 By Taylor Formula, we have for a.e.\ $x$ in $\O:$
 \[
\boldsymbol\Delta _{\epsilon}^a\varphi -a\cdot \nabla \varphi (x)
=a\cdot \int_0^1\left(\nabla \varphi \big(x+\epsilon(2\theta -1)a\big)-\nabla \varphi (x)\right)\d \theta 
 \]
 whence
 \[
\begin{aligned}
\int_\O|\boldsymbol\Delta _{\epsilon}^a\varphi -a\cdot \nabla \varphi (x)|^p\d x
&\leq|a|^p\int_\O\int_0^1\left|\nabla \varphi \big(x+\epsilon(2\theta -1)a\big)-\nabla \varphi (x)\right|^p\d \theta \d x
\\
&=|a|^p\int_0^1\left(\int_\O\left|\nabla \varphi \big(x+\epsilon(2\theta -1)a\big)-\nabla \varphi (x)\right|^p\d x\right)\d \theta 
\\
&=|a|^p\int_0^1|(\tt_{-\epsilon (2\theta -1)a} -\id)\nabla \varphi (x)|_{L^p}^p\d \theta 
\\
&\to0\,,\quad \quad \quad \text{as}\enskip \epsilon \to0\,,
\end{aligned}
\]
using the strong $L^p$ continuity of $(\tt_a)_{a\in\O}$ when $p\in[1,\infty)$ and dominated convergence.

\item[\indent\textit{Case $p=\infty$.}]
Similarly, we have
\[
\begin{aligned}
|\boldsymbol\Delta _{\epsilon}^a\varphi -a\cdot \nabla \varphi|_{L^\infty}
&\leq \int_0^1\sup_{x\in\O}|(\tt_{-\epsilon (2\theta -1)a} -\id)\nabla \varphi (x)|\d \theta 
\\
&\leq \int_0^1\epsilon|2\theta -1| \sup_{x\in\O}\int_0^1|\nabla ^2\varphi\big(x +\theta '\epsilon (2\theta -1)a  \big) |\d \theta' \d \theta 
\\
&\leq C\epsilon |\varphi  |_{W^{2,\infty}}
\to0\,,
\quad \quad \text{as}\enskip \epsilon \to0\,,
\end{aligned}
\]
which proves the lemma.
\end{proof}
\subsection{The Sewing Lemma}\label{app:A1}
A proof of the following classical result, for the case where $E$ is a (finite-dimensional) normed vector space, can be found e.g.\ in \cite{gubinelli2004controlling,gubinelli2010rough}, the Banach space case being treated e.g.\ in \cite{friz2014course}. The result appears to have an immediate extension to the case of a complete locally convex topological vector space $E$ (l.c.v.s.), which is a repeatedly encountered scenario in PDE theory (see Remark \ref{exa:distributions} below).

As before, we set  $I:=[0,T],$ for some $T>0,$ and $\Delta \equiv\Delta _I,$ $\Delta ^2\equiv\Delta ^2_{I}$ to be the corresponding simplexes.
Given a l.c.v.s.\ $E$ equipped with a family of seminorms $(p_\gamma)_{\gamma\in \Gamma}$, and $a >0$ we define the space $\V_ 1^{a }(I;E)$ as the set of paths $h:I\to E$ such that for every $\gamma\in\Gamma$ and every $(s,t)\in\Delta ,$ there holds  $p_\gamma\big(\delta h_{st}\big)\leq \omega_{h,\gamma} (s,t)^{a }$  for $(s,t)\in\Delta ,$ for some control depending on $h$ and $\gamma.$ Note that $\V^{a}_1(I;E)$ is also a locally convex topological vector space given by the seminorms
\[
h\mapsto \sup_{\pp\in\mathscr{P}(I)}\left(\sum\nolimits_{(\pp)}p_\gamma(\delta h_{t_i t_{i+1}}) ^{1/a}\right)^a,\qquad \gamma\in\Gamma\,,
\]
(see \eqref{summation_pp}).
The space $\V_ 2^{a }(I;E),$ is defined in a similar fashion.  
Furthermore, $\V^{1+}_2(I;E)$ corresponds to those 2-index maps $g\equiv g_{st}$ such that for each $p_\gamma$ as above, there is a control $\omega _{g,\gamma}$ and $a_\gamma>1$ with $p_\gamma(g_{st})\leq \omega_{g,\gamma} (s,t)^{a_\gamma}$ for $(s,t)\in\Delta .$

\begin{proposition}[Sewing Lemma]\label{pro:sewing}
Let $E$ be a complete, locally convex topological vector space.
Let  $(p_\gamma )_{\gamma\in\Gamma}$ be a family of semi-norms.

Define
$\Zc^{1+}(I;E) $ as the set of $3$-index maps $h:\Delta ^2\to E$ such that
\begin{itemize}
\item there exists a continuous $B:\Delta \to E$ with $h=\delta B;$
\item for each $\gamma\in\Gamma,$ there is a control $\omega _{h,\gamma}:\Delta \to \R_+$ and $a_\gamma  >1,$ such that
\begin{equation}
\label{a_gamma}
p_\gamma \left(h_{s\theta t}\right)\leq \omega _{h,\gamma}(s,t)^{a_\gamma }\,,
\end{equation}
uniformly as $(s,\theta ,t)\in\Delta ^2.$
\end{itemize}

Then, there exists a linear map $\Lambda: \Zc^{1+}(I;E) \to \V_ 2^{1+} (I;E),$
\emph{continuous} in the sense that for every $\gamma \in\Gamma $ and  $h\in \Zc^{1+} (I;E)$ there holds
\begin{equation}\label{est:sew}
p_\gamma \left(\Lambda h_{st}\right)\leq C_{a_\gamma }\omega _{h,\gamma}(s,t)^{a_\gamma }\,,
\quad \text{for every}\quad  (s,t)\in\Delta \,,
\end{equation}
where the above constant only depends on the value of $a_\gamma>1.$
In addition, $\Lambda $ is a right inverse for $\delta ,$ namely
\begin{equation}\label{uniqueness_SL}
\delta\Lambda =\id|_{\Zc^{1+}}\,,
\end{equation}
and it is \emph{unique} in the class of linear mappings fulfilling the properties \eqref{est:sew}-\eqref{uniqueness_SL}.

Finally, for any $(s,t)\in\Delta ,$
we have the explicit formula:
\begin{equation}
 \label{explicit_Lambda}
 \Lambda_{st} h =\lim_{|\pp|\to0}\left(B_{st}-\sum\nolimits_{(\pp)}B_{t_i t_{i+1}}\right)\,,
\end{equation} 
where we use the summation convention \eqref{summation_pp}.
\end{proposition}
\begin{example}
\label{exa:distributions}
The above infinite-dimensional Sewing Lemma applies in $\mathscr D'(\mathcal O),$
the space of distributions over an open subset $\mathcal O$ of some Euclidean space, for which a family of semi-norms is provided by 
\[
p_\phi (v):=|\langle v,\phi \rangle|\,,\quad \quad \phi \in C^\infty_c(\mathcal O)\,,
\]
for $v$ in $\mathscr D'(\mathcal O).$

We could replace $\mathscr D'$ by the space of Schwarz distributions $\mathscr S',$
or any Banach space of linear functionals endowed with the weak-* topology.
\end{example}
\begin{proof}
The proof is similar to that of \cite{friz2014course}.
Fix $(s,t)\in\Delta ,$ and consider a partition 
$\pp :=\{s\equiv t_1< t_2<\dots<t_{k}\equiv t\}$ of $[s,t],$ such that $\#\pp=k\geq 2.$
Define
\[
\Lambda ^{\pp}h:=B_{st}-\sum_{1\leq i\leq k-1}B_{t_it_{i+1}}\,,
\]
where $B$ is such that $\delta B=h.$

Let $\gamma\in\Gamma.$
By the superadditivity of $\omega _{h,\gamma},$ there exists $i_1\in\{1,\dots,k\}$ such that
\[
\omega _{h,\gamma} (t_{i_1-1},t_{i_1+1})\leq \frac{2}{k-1}\omega _{h,\gamma}(s,t)\,.
\]
Moreover, we have the relation
\begin{equation}
\label{p_a_relation}
p_\gamma \big(\Lambda^{\pp\setminus\{t_{i_1}\}}h- \Lambda ^{\pp}h\big)=p_\gamma (\delta B_{t_{i-1},t_i,t_{i+1}})\leq  \left(\frac{2}{k-1}\omega _{h,\gamma}(s,t)\right)^{a_\gamma }\,.
\end{equation}
Replacing $\pp$ by $\pp\setminus\{t_{i_1}\},$ we can iterate this procedure until we end up with the trivial partition $\pp\setminus\{t_{i_1},\dots,t_{i_{k-2}}\}\equiv\{s,t\}$ for which $\Lambda ^{\{s,t\}}h=0$ (note that the order in which the points $t_i$ are dropped out may depend on $\gamma $ in $\Gamma$, but this is not a problem since the final expression does not).
Writing that 
\[
\Lambda ^{\pp}h=\left(\Lambda ^{\pp}-\Lambda ^{\pp\setminus\{t_{i_1}\}}\right)h
+\dots +\left(\Lambda ^{\pp\setminus\{t_{i_1},\dots,t_{i_{k-3}}\}}-\Lambda ^{\{s,t\}}\right)h\,,
\]
and using \eqref{p_a_relation} $k-2$ times,
we find the maximal inequality
\begin{equation}
\label{inegalite_zeta_function}
p_\gamma \left(\Lambda ^{\pp}h\right)\leq 2^{a_\gamma }\omega _{h,\gamma}(s,t)^{a_\gamma }\sum_{i=1}^{k-2}i^{-a_\gamma }
\leq 2^{a_\gamma }\omega _{h,\gamma}(s,t)^{a_\gamma }\sum_{i=1}^{\infty}i^{-a_\gamma }
\leq C_{a_\gamma }\omega _{h,\gamma}(s,t)^{a_\gamma }\,,
\end{equation}
and this holds for every $\gamma $ in $\Gamma.$

Now, let us consider a refined partition $\pp'\subset \pp.$
We have
\[
\Lambda ^{\pp}h -\Lambda ^{\pp'}h=-\sum _{t_i\in \pp,\, i<k}\Bigg( \underbrace{B_{t_i t_{i+1}} - \sum _{\{\tau,\tilde\tau\}\subset\pp' \cap [t_i,t_{i+1}],\,\tau<\tilde\tau}B_{\tau\tilde\tau}}_{\Lambda ^{\pp'\cap[t_i,t_{i+1}]}h}\Bigg)
\]
whence, using the maximal inequality \eqref{inegalite_zeta_function}
on each $[t_i,t_{i+1}],$ there comes:
\[
p_\gamma \Big(\Lambda ^{\pp}h -\Lambda ^{\pp'}h\Big)\leq
\sum _{t_i\in \pp,\, i<k}C_a \omega _{h,\gamma} (t_i,t_{i+1})^{a_\gamma }\,.
\]
Since $a_\gamma >1$, the r.h.s.\ above vanishes as the size of $\pp$ goes to $0,$
which by completeness of $E$ shows the convergence of $\Lambda ^\pp h$ towards some $\Lambda _{st}h,$ for any $(s,t)\in\Delta .$

Finally, one can follow the lines of \emph{Step 4} in \cite[Proposition 2.3]{gubinelli2010rough} to show that we have $\delta \Lambda h=h$. This completes the proof.
\end{proof}
\begin{cordef}\label{cor_def}
Given $\alpha \in(0,1],$ let $B$ in $\V_ 2^{\alpha }(I;E)$ and assume that $\delta B \in \Zc^{1+}.$
Define
\begin{equation}
\label{def:rough_integral}
\I(B):= B - \Lambda \delta B\,\in \V_ 2^{\alpha }(I;E)\,.
\end{equation}

Then, the linear map $\I:\V_ 2^{\alpha }(I;E)\to \V_ 2^{\alpha }(I;E),\,B\mapsto \I(B)$ fulfills the following properties
\begin{itemize}
 \item $\delta \I =0;$
 \item if $h\in \V_ 2^{\alpha }(I;E)$ is another 2-index map such that $\delta h=0$ and $h-B\in \V^{1+}(I;E),$ then $h=\I(B);$
 \item for any $B$ as above, $\I(B)$ is given by
 \begin{equation}
 \label{explicit_I}
 \I_{st} B =\lim_{|\pp|\to0}\sum\nolimits_{(\pp)}B_{t_i t_{i+1}}\,;
\end{equation} 
 \item let $E$ be a reflexive Banach space, and assume that $f:I\to \L(E,F),$ $g:I\to E$ are measurable, $f$ being continuous, and such that $g$ belongs to $\AC(I;E).$
 Let $\dot g\in L^1(I;E)$ denote the weak derivative of the path $g.$ Assume in addition that $\delta(f\delta g)\in \Zc^{1+}(I;F)$.
 Then, we have $\int_I|f_r\dot g_r|_{F}\d r<\infty$ and
\begin{equation}
\label{bochner_integral}
\I(f\delta g)_{st}=\int_s^tf_r\dot g_r\d r\quad\text{(Bochner integral in $F$)}\,,
\end{equation}
where $f\delta g$ is to be understood as the map $ (s,t)\in \Delta \mapsto f_s\delta g_{st}.$
\end{itemize}

For $B$ as above,
the 2-index map $(s,t)\in\Delta \mapsto\I(B)_{st}$ is called the \emph{rough integral} of $B.$
\end{cordef}
\begin{proof}
The three first statements are immediate consequences of Proposition \ref{pro:sewing}, (for a proof in the Banach space setting, we refer e.g.\ to \cite{gubinelli2004controlling,friz2014course}).

Let us check the last point.
First, note that the weak derivative of $g$ exists, because any reflexive space fulfills the Radon-Nikodym property (see \cite[Definition 3 p.\ 61 and Corollary 13 p.\ 76]{diestel1977vector}).
From the formula \eqref{explicit_I}, it holds that $\I(f\delta g)$ is the limit, as $n\to\infty$ of the partial sums
\[
I_n:= \sum_{(\pp_n)}f_{t_i^n} \delta g_{t_i^nt_{i+1}^n}\equiv \sum_{(\pp_n)}f_{t_i^n}\int_{t_i^n}^{t_{i+1}^n}\dot g_r \d r
=\int_I f_r\dot g_r\d r -\sum_{(\pp_n)}\int_{t_i^n}^{t_{i+1}^n}(f_r-f_{t_i^n})\dot g_r\d r
\,,
\]
where $\pp_n\equiv(t_i^n)$ is such that $|\pp_n|\to0.$
The mapping $f:I\equiv[0,T]\to \L(E,F)$ is continuous, hence uniformly continuous, so that the second term above goes to $0$ as $n\to\infty.$ 
Therefore, $\I(f\delta g)\equiv\lim I_n=\int_I f_r \dot g_r\d r,$ which proves \eqref{bochner_integral}.
\end{proof}

\subsection{Families of smoothing operators}
\label{ss:smoothing}

Let $R_\eta$ denote the family of smoothing operators defined
on $\varphi \in W^{k,\infty}\equiv W^{k,\infty}(\O),k\in\N,$ by
\begin{equation}
\label{nota:J_rho}
R_\eta\varphi(x):=[\varphi*\varrho_\eta](x)=
\left[\varphi *\varrho\Big(\frac{\cdot}{\eta }\Big)\eta^{-d}\right](x)
\equiv\int_{\O}\varphi(\xi)\varrho\Big(\frac{x-\xi}{\eta } \Big) \frac{\d\xi}{\eta ^d},\quad x\in \O\,,
\end{equation}
where $\varrho\in C^\infty(\R^d;\R)$
is a non-negative, \emph{radially symmetric} function that integrates to $1$, and such that $\spt\varrho\subset B_1.$
As a consequence, $R_\eta$ reproduces affine linear functions exactly and it is then possible to recover the error of order $\eta^2$ for $|(R_\eta-\id)\varphi|_{L^\infty}$  provided $\varphi $ belongs to $W^{2,\infty}$ (this is classical and follows from a Taylor expansion of the integrand).
More precisely, we have the following.
\begin{lemma}
\label{lem:varrho}
The family $(R_\eta )_{\eta\in(0,1)}$ is a 2-step family of smoothing operators over the scale $W^{k,\infty}(\R^d).$
\end{lemma}
\begin{remark}
One could also consider different mollifiers (no longer nonnegative) which would reproduce polynomials of higher order exactly, in order to obtain higher rates of convergence of $|(R_\eta-\id)\varphi|_{W^{n,\infty}}$ under suitable regularity assumption on $\varphi$.
Second order estimates in $\eta $ are however sufficient here.
\end{remark}

Since $R_\eta $ increases the support of test functions, it cannot define a smoothing family for the scale $(\F_k)_{k\in\N}$ defined in \eqref{scale2_2}.
To deal with that problem, we need to introduce a suitable cut-off function.
Let $\theta _\eta \in C_c^\infty(\R)$ such that
\begin{equation}
 \label{nota:theta}
0\leq \theta _\eta \leq 1\,,\quad 
\spt \theta _\eta \subset B_{1-2\eta }\subset \R\,,\quad 
\theta \equiv1 \quad \text{on}\quad B_{1-3\eta }\subset \R\,,
\end{equation} 
and such that for $k=1,2:$
\[
|\nabla ^k\theta _\eta |\leq \frac{C}{\eta ^k}\,.
\]
Next, we define
\begin{equation}
\label{nota:Theta}
\Theta _\eta (x):=\theta _\eta (|x|^2)\,,\quad \text{for}\quad x\in\O\,.
\end{equation}

The following has been shown in \cite{deya2016priori}.

\begin{lemma}
\label{lem:Theta}
There is a constant $C_\theta >0$ such that for $k=0,1,2,$ and every $\psi $ in $W^{k,\infty}(\O)$ compactly supported in $B_1:$
\begin{equation}
\label{Theta_1}
|\Theta _\eta \psi|_{W^{k,\infty}}\leq C_\theta |\psi |_{W^{k,\infty}}\,.
\end{equation}
If in addition we assume $\psi \in W^{k,\infty}(\O),$ with $0\leq \ell\leq k\leq 3$ then
\begin{equation}
\label{Theta_2}
|(1-\Theta _\eta )\psi |_{W^{\ell,\infty}}\leq C_\theta \eta ^{k-\ell}|\psi |_{W^{k,\infty}}\,.
\end{equation} 
\end{lemma}
\begin{corollary}
\label{cor:J_eta}
The linear mappings $J_\eta : \F_0(\Omega ) \to \F_0(\Omega ),\eta \in(0,1),$ defined by
\[
J_\eta \phi := \chi \circ \big(R_\eta \otimes (R_\eta \Theta_\eta)  (\phi\circ\chi ^{-1})\big)
\]
where we keep the notations of Lemma \ref{lem:varrho}, \eqref{new_coordinates} and \eqref{nota:Theta},
provide a 2-step family of smoothing operators with respect to the scale $(\F_k(\Omega ))_{k\in\N}.$
\end{corollary}
\begin{proof}
Since $\sqrt2\chi$ is a rotation, it is sufficient to show the corollary on the scale 
\begin{equation}
 F_k:=\left\{\phi \in W^{k,\infty}(\OO)\,,\spt\phi \subset \O\times B_1\right\}\,,
\end{equation}
endowed with the norm
$\nnn{\cdot }{k}:=|\cdot |_{W^{k,\infty}},$ and $J_\eta := R_\eta \otimes (R_\eta \Theta _\eta).$

Note first that for any fixed $x\in \O,$ and $\phi \in F_k:$
\begin{equation}
\label{support_preserving}
\spt\big(\id\otimes (R_\eta \Theta _\eta) \phi (x,\cdot )\big)
\subset \spt (\Theta _\eta \phi (x,\cdot )) + \spt(\varrho_\eta ) \subset B_1
\end{equation}
Since we have $J_\eta \phi =(R_\eta \otimes \id)(\id\otimes R_\eta \Theta _\eta) \phi ,$ we see that
\[
\spt J_\eta \phi \subset B_1\,,
\]
and because $J_\eta \phi $ is smooth, the property \ref{J1} follows.

Concerning \ref{J2},
let for instance fix $k=0,$ and
$\phi \in F_0.$ 
Using Lemma \ref{lem:varrho}, denoting by
$\psi ^y:=(\id\otimes R_{\eta }\Theta _\eta )\phi (\cdot ,y),$
we have for any $1\leq i\leq d$ and $x,y\in \O:$
\begin{multline*}
|\partial _{x_i} J_\eta \phi (x,y)|
\equiv |\partial _{x_i}(R_{\eta}\otimes \id)[\psi ^y](x)|
\leq \frac{C}{\eta }|\psi ^y|_{L^\infty_x}\leq \frac{C}{\eta }\int_\O\Theta _\eta (y')|\phi (\cdot ,y')|_{L^\infty_x}\varrho_\eta (y-y')\d y'
\\
\leq \frac{C }{\eta }\nnn{\phi}{0}\,.
\end{multline*}
Similarly, denoting by $\tilde\psi ^x:= (R_\eta\otimes \id)(1-\Theta _\eta)\phi(x,\cdot ),$
it holds
\begin{multline*}
|\partial _{y_i} J_\eta \phi (x,y)|
\leq |\partial _{y_i}(R_\eta \otimes R_\eta )\phi | +
|\partial _{y_i}(\id \otimes R_\eta )\tilde\psi ^x(y)|
\leq \frac{C }{\eta }\nnn{\phi}{0} + \frac{C}{\eta }|\tilde\psi ^x|_{L^\infty_y}
\\
\leq 
\frac{C}{\eta }\nnn{\phi }{0} + \frac{C}{\eta }\int_\O|(1-\Theta _\eta (y))\phi (x',\cdot )|_{L^\infty_y}\varrho_\eta (x-x')\d x'
\leq \frac{C'}{\eta }\nnn{\phi }{0}\,.
\end{multline*}
Inequalities corresponding to $k=1,2$ are shown in a similar way, using in addition \eqref{Theta_1}-\eqref{Theta_2}.
The bounds related to \ref{J3} are similar.
\end{proof}

\bibliographystyle{alpha}

\end{document}